\documentclass[10pt,a4paper]{amsart}
\usepackage[margin=3.6cm]{geometry}

\usepackage[numbers]{natbib}
\usepackage{amsmath}
\usepackage{amssymb}
\usepackage[mathscr]{euscript}
\usepackage{enumerate}
\usepackage{xspace}
\usepackage{color}
\usepackage{enumerate}
\usepackage{enumitem}
\usepackage{amsthm}

\begin{document}

\theoremstyle{plain}
\newtheorem{C}{Convention}
\newtheorem{fact}{Fact}
\newtheorem*{SA}{Standing Assumption}
\newtheorem{theorem}{Theorem}
\newtheorem{condition}{Condition}
\newtheorem{lemma}{Lemma}
\newtheorem{proposition}{Proposition}
\newtheorem{corollary}{Corollary}
\newtheorem{claim}[theorem]{Claim}
\newtheorem{definition}{Definition}
\newtheorem{Ass}[theorem]{Assumption}
\newcommand{\q}{Q}
\theoremstyle{definition}
\newtheorem{remark}{Remark}
\newtheorem{note}[theorem]{Note}
\newtheorem{example}{Example}
\newtheorem*{exampleN}{Example 1 (continued)}
\newtheorem{assumption}[theorem]{Assumption}
\newtheorem*{notation}{Notation}
\newtheorem*{assuL}{Assumption ($\mathbb{L}$)}
\newtheorem*{assuAC}{Assumption ($\mathbb{AC}$)}
\newtheorem*{assuEM}{Assumption ($\mathbb{EM}$)}
\newtheorem*{assuES}{Assumption ($\mathbb{ES}$)}
\newtheorem*{assuM}{Assumption ($\mathbb{M}$)}
\newtheorem*{assuMM}{Assumption ($\mathbb{M}'$)}
\newtheorem*{assuL1}{Assumption ($\mathbb{L}1$)}
\newtheorem*{assuL2}{Assumption ($\mathbb{L}2$)}
\newtheorem*{assuL3}{Assumption ($\mathbb{L}3$)}
\newtheorem{charact}[theorem]{Characterization}

\newcommand{\Law}{\ensuremath{\mathop{\mathrm{Law}}}}
\newcommand{\loc}{{\mathrm{loc}}}
\newcommand{\Log}{\ensuremath{\mathop{\mathscr{L}\mathrm{og}}}}
\newcommand{\Meixner}{\ensuremath{\mathop{\mathrm{Meixner}}}}
\newcommand{\of}{[\hspace{-0.06cm}[}
\newcommand{\gs}{]\hspace{-0.06cm}]}
\newcommand{\tri}{|\hspace{-0.06cm}|\hspace{-0.06cm}|}
\renewcommand{\epsilon}{\varepsilon}
\let\MID\mid
\renewcommand{\mid}{|}

\let\SETMINUS\setminus
\renewcommand{\setminus}{\backslash}

\def\stackrelboth#1#2#3{\mathrel{\mathop{#2}\limits^{#1}_{#3}}}

\renewcommand{\theequation}{\thesection.\arabic{equation}}
\numberwithin{equation}{section}

\newcommand\llambda{{\mathchoice
      {\lambda\mkern-4.5mu{\raisebox{.4ex}{\scriptsize$\backslash$}}}
      {\lambda\mkern-4.83mu{\raisebox{.4ex}{\scriptsize$\backslash$}}}
      {\lambda\mkern-4.5mu{\raisebox{.2ex}{\footnotesize$\scriptscriptstyle\backslash$}}}
      {\lambda\mkern-5.0mu{\raisebox{.2ex}{\tiny$\scriptscriptstyle\backslash$}}}}}

\newcommand{\prozess}[1][L]{{\ensuremath{#1=(#1_t)_{0\le t\le T}}}\xspace}
\newcommand{\prazess}[1][L]{{\ensuremath{#1=(#1_t)_{0\le t\le T^*}}}\xspace}

\newcommand{\tr}{\operatorname{tr}}
\newcommand{\lijepoa}{{\mathscr{A}}}
\newcommand{\lijepob}{{\mathscr{B}}}
\newcommand{\lijepoc}{{\mathscr{C}}}
\newcommand{\lijepod}{{\mathscr{D}}}
\newcommand{\lijepoe}{{\mathscr{E}}}
\newcommand{\lijepof}{{\mathscr{F}}}
\newcommand{\lijepog}{{\mathscr{G}}}
\newcommand{\lijepok}{{\mathscr{K}}}
\newcommand{\lijepoo}{{\mathscr{O}}}
\newcommand{\lijepop}{{\mathscr{P}}}
\newcommand{\lijepoh}{{\mathscr{H}}}
\newcommand{\lijepom}{{\mathscr{M}}}
\newcommand{\lijepou}{{\mathscr{U}}}
\newcommand{\lijepov}{{\mathscr{V}}}
\newcommand{\lijepoy}{{\mathscr{Y}}}
\newcommand{\cF}{{\mathscr{F}}}
\newcommand{\cG}{{\mathscr{G}}}
\newcommand{\cH}{{\mathscr{H}}}
\newcommand{\cM}{{\mathscr{M}}}
\newcommand{\cD}{{\mathscr{D}}}
\newcommand{\bD}{{\mathbb{D}}}
\newcommand{\bF}{{\mathbb{F}}}
\newcommand{\bG}{{\mathbb{G}}}
\newcommand{\bH}{{\mathbb{H}}}
\newcommand{\dd}{d}
\newcommand{\ddd}{\operatorname{d}}
\newcommand{\er}{{\mathbb{R}}}
\newcommand{\ce}{{\mathbb{C}}}
\newcommand{\erd}{{\mathbb{R}^{d}}}
\newcommand{\en}{{\mathbb{N}}}
\newcommand{\de}{{\mathrm{d}}}
\newcommand{\im}{{\mathrm{i}}}
\newcommand{\indik}{{\mathbf{1}}}
\newcommand{\D}{{\mathbb{D}}}
\newcommand{\E}{E}
\newcommand{\N}{{\mathbb{N}}}
\newcommand{\Q}{{\mathbb{Q}}}
\renewcommand{\P}{{\mathbb{P}}}
\newcommand{\ud}{\operatorname{d}\!}
\newcommand{\ii}{\operatorname{i}\kern -0.8pt}
\newcommand{\cadlag}{c\`adl\`ag }
\newcommand{\p}{P}
\newcommand{\F}{(\mathcal{F}_t)_{t \geq 0}}
\newcommand{\1}{\mathbf{1}}
\newcommand{\f}{\mathscr{F}^{\hspace{0.03cm}0}}
\newcommand{\lle}{\langle\hspace{-0.085cm}\langle}
\newcommand{\rre}{\rangle\hspace{-0.085cm}\rangle}
\newcommand{\llbr}{[\hspace{-0.085cm}[}
\newcommand{\rrbr}{]\hspace{-0.085cm}]}

\def\EM{\ensuremath{(\mathbb{EM})}\xspace}

\newcommand{\la}{\langle}
\newcommand{\ra}{\rangle}

\newcommand{\Norml}[1]{%
{|}\kern-.25ex{|}\kern-.25ex{|}#1{|}\kern-.25ex{|}\kern-.25ex{|}}

\title[Lyapunov Criteria for the Feller-Dynkin Property of MPs]{Lyapunov Criteria for the Feller-Dynkin Property of Martingale Problems} 
\author[D. Criens]{David Criens}
\address{D. Criens - Technical University of Munich, Center for Mathematics, Germany}
\email{david.criens@tum.de}

\keywords{Feller-Dynkin Process, \(C_b\)-Feller Process, Martingale Problem, Lyapunov Function, Switching Diffusions\vspace{1ex}}


\thanks{D. Criens - Technical University of Munich, Center for Mathematics, Germany,  \texttt{david.criens@tum.de}.}

\thanks{Acknowledgement: The author thanks Stefan Junk for valuable discussions. Moreover, he thanks the referees for carefully reading the article and for many helpful comments.}

\date{\today}
\maketitle

\frenchspacing
\pagestyle{myheadings}

\begin{abstract}
We give necessary and sufficient criteria for the Feller-Dynkin property of solutions to martingale problems in terms of Lyapunov functions. Moreover, we derive a Khasminskii-type integral test for the Feller-Dynkin property of multidimensional diffusions with random switching. For one dimensional switching diffusions with state-independent switching, we provide an integral-test for the Feller-Dynkin property. 
\end{abstract}

\section{Introduction}
It is a classical question for a Markov process whether its transition semigroup is a self-map on the space of bounded continuous functions and on the space of continuous functions vanishing at infinity, respectively.
If the first property holds we call the Markov process a \(C_b\)-Feller process and when the second property holds we call it a Feller-Dynkin process. In the literature Feller-Dynkin processes are often also called Feller processes, see, for instance, \cite{gall2016brownian,liggett2010continuous,RY}. Our terminology is borrowed from \cite{RW1}.

Let us give some examples for applications of the Feller-Dynkin property.
For empirical laws of i.i.d. processes it is interesting to study the probability of large deviations. Sanov's theorem implies that a large deviation principle holds, see, for instance, \cite[Section 6.2]{dembo2009large}. However, the rate function is in general given in terms of an entropy and therefore hard to understand. 
In \cite{kraaij2018} it was recently shown that for empirical laws of i.i.d. Feller-Dynkin processes whose generators have suitable cores the rate function can be decomposed into a rate function for the initial time and an integral over a Lagrangian depending on position and speed. This decomposition opens up the possibility for further analysis.
 
 From an analytical perspective, the Feller-Dynkin semigroup and its generator play an important role in the study of evolution equations.
Namely, if \((L, \mathcal{D}(L))\) is the generator of a Feller-Dynkin semigroup and \(\mathbb{X} = (C_0, \|\cdot\|_\infty)\) is the Banach space of continuous functions vanishing at infinity, then for any non-linearity \(h\) satisfying a Lipschitz condition the deterministic evolution equation 
\begin{align*} 
\dd u (t)=\big( L u(t) + h(t, u(t))\big)\dd t,\quad u(0) = f \in C_0,
\end{align*}
has a mild solution in \(\mathbb{X}\), see, for instance, \cite[Section 6.1]{pazy2012semigroups}.  While such an existence result is of purely analytic nature, the connection of the semigroup and its generator to a stochastic process can be useful to verify its prerequisites. Another point of contact between analysis and probability theory is the stochastic representation of solutions to evolution equations on \(\mathbb{X}\) via the Feynman-Kac formula, see, for instance, \cite[Theorem 3.47]{liggett2010continuous}.  
The stochastic interpretations provided by these representations can help to understand the behavior of solutions. 

Because Markov processes are usually defined by its infinitesimal description, it is particularly interesting to find criteria for the Feller properties in terms of the generalized infinitesimal generator of the Markov process.

In this article we give such criteria for Markov processes defined via abstract martingale problems (MPs). 
Our contributions are two-fold. First, we show that the Feller-Dynkin property of can be described by a Lyapunov-type criterion in the spirit of the classical Lyapunov-type criteria for explosion, recurrence and transience, see, e.g., \cite{khasminskii2011stochastic, pinsky1995positive}. More precisely, we prove a sufficient condition for the Feller-Dynkin property, see Theorem \ref{theo: Ly} below, and a condition to reject the Feller-Dynkin property, see Theorem \ref{prop:cont} below.
Under additional assumptions on the input data, we extend the sufficient condition for the Feller-Dynkin property to be necessary, see Theorem \ref{coro: 1} below.
The necessity is for instance useful when one studies coupled processes, i.e. processes whose infinitesimal description is built from the infinitesimal description of other processes. We illustrate this in our applications. 
Moreover, we provide a technical condition for a reduction or an enlargement of the input data of a MP, see Proposition \ref{prop: redc} below. A reduction helps to check the additional assumption of our necessary and sufficient criterion, while an enlargement simplifies finding Lyapunov functions for our sufficient conditions.
We apply our criteria to derive conditions for the Feller-Dynkin property of multidimensional diffusions with random switching. In particular, we derive a Khasminskii-type integral test for the Feller-Dynkin property.

Our second contribution is a systematic study of the Feller-Dynkin property of switching diffusions with state-independent switching. In other words, we consider a process \((Y_t, Z_t)_{t \geq 0}\), where \((Z_t)_{t \geq 0}\) is a continuous-time Feller-Dynkin Markov chain and \((Y_t)_{t \geq 0}\) solves the stochastic differential equation (SDE)
\[
\dd Y_t = b(Y_t, Z_t)\dd t + \sigma(Y_t, Z_t)\dd W_t, 
\]
where \((W_t)_{t \geq 0}\) is a Brownian motion. One may think of the process \((Y_t)_{t \geq 0}\) as a diffusion in a random environment given by the Markov chain \((Z_t)_{t \geq 0}\). 
The process \((Y_t)_{t \geq 0}\) has a natural relation to processes with fixed environments, i.e. solutions to the SDEs
\begin{align}\label{eq:stateSDE}
\dd Y^k_t = b(Y^k_t, k) \dd t + \sigma(Y^k_t, k)\dd W_t, 
\end{align}
where \(k\) is in the state space of \((Z_t)_{t \geq 0}\). When \((Y_t, Z_t)_{t \geq 0}\) is a \(C_b\)-Feller process and the SDEs \eqref{eq:stateSDE} satisfy weak existence and pathwise uniqueness, we show that \((Y_t, Z_t)_{t \geq 0}\) is a Feller-Dynkin process if and only if the processes in the fixed environments are Feller-Dynkin processes. Furthermore, using a limit theorem for switching diffusions, see Theorem \ref{theo: limit} below, we show that \((Y_t, Z_t)_{t \geq 0}\) is a \(C_b\)-Feller process whenever it exists uniquely and the coefficients are continuous. 
We also explain that the uniqueness of \((Y_t, Z_t)_{t \geq 0}\) is implied by weak existence and pathwise uniqueness of the diffusions in the fixed environments. 
For the one dimensional case we deduce an equivalent integral-test for the Feller-Dynkin property of \((Y_t, Z_t)_{t \geq 0}\) and for multidimensional settings we give a Khasminskii-type integral test.

We end this introduction with comments on related literature. To the best of our current knowledge, Lyapunov-type criteria for the Feller-Dynkin property are only used in specific case studies and a systematic study as given in this article does not appear in the literature.
For continuous-time Markov chains, explicit conditions for the Feller-Dynkin property can be found in \cite{LI2009653,doi:10.1112/jlms/s2-5.2.267}. 
In \cite{LI2009653} also a Lyapunov-type condition appears. 
Infinitesimal conditions for the Feller-Dynkin property of diffusions are given in \cite{Azencott1974}. 
In the context of jump-diffusions, linear growth conditions for the Feller-Dynkin property were recently proven in \cite{kuhn2018b,kuhn2018a}. The proofs include a Lyapunov-type argument based on Gronwall's lemma.
For switching diffusions the \(C_b\)-Feller and the strong Feller property are studied profoundly, see, for instance,  \cite{doi:10.1137/16M1059357, doi:10.1137/15M1013584, doi:10.1137/16M1087837,yin2009hybrid}. 
Here, we say that the strong Feller property  holds if the transition semigroup maps bounded functions to bounded continuous functions. It is clear that the strong Feller property implies the \(C_b\)-Feller property. We stress that neither the strong Feller property nor the Feller-Dynkin property implies the other. An easy example for a Feller-Dynkin process which does not have the strong Feller property is the linear motion and an example for a strong Feller process, which does not have the Feller-Dynkin property is given in Example \ref{ex: FD not SF} below.
We think that our study of the Feller-Dynkin property for switching diffusions is the first of its kind. Also our continuity criterion for the \(C_b\)-Feller property in the state-independent case seems to be new.

The article is structured as follows. In Section \ref{sec: set} we explain our setup. In particular, in Section \ref{sec: 2.2} we recall the different concepts for the Feller properties of martingale problems. In Section \ref{sec: main} we discuss Lyapunov-type conditions for the Feller-Dynkin property in a general abstract setting and in Section \ref{sec: app} we discuss the case of switching diffusions.
A Skorokhod-type existence result for state-independent switching diffusions can be found in Appendix \ref{appendix}.

\section{The Feller Properties of Martingale Problems}\label{sec: set}
\subsection{The Setup}\label{sec: setup}
	Let \(S\) be a locally compact Hausdorff space with countable base (LCCB space), define \(\Omega\) to be the space of all \cadlag functions \(\mathbb{R}_+ \to S\) and let \((X_t)_{t \geq 0}\) be the coordinate process on \(\Omega\), i.e. the process defined by \(X_t(\omega) = \omega(t)\) for \(\omega \in \Omega\) and \(t \in \mathbb{R}_+\).
	We set \(\mathcal{F} \triangleq \sigma(X_t, t \in \mathbb{R}_+)\) and
\(
\mathcal{F}_t \triangleq \bigcap_{s > t} \mathcal{F}^o_s,\) where \(\mathcal{F}^o_t \triangleq \sigma(X_s, s \in [0, t]).
\)
If not stated otherwise, all terms such as \emph{local martingale, supermartingale,} etc. refer to \(\F\) as the underlying filtration. In general, we equip \(\Omega\) with the Skorokhod topology (see \cite{EK, JS}). In this case, \(\mathcal{F}\) is the Borel \(\sigma\)-field on \(\Omega\), see \cite[Proposition 3.7.1]{EK}.

We use standard notation for function spaces, i.e. for example we denote by \(M(S)\) the set of Borel functions \(S \to \mathbb{R}\), by \(B(S)\) the set of bounded Borel functions \(S \to \mathbb{R}\), by \(C(S)\) the set of continuous functions \(S \to \mathbb{R}\) and by \(C_0(S)\) the space of continuous functions \(S \to \mathbb{R}\) which are vanishing at infinity, etc.
We take the following four objects as input data for our abstract MP:
\begin{enumerate}
	\item[(i)] A set \(D \subseteq C(S)\) of test functions.
	\item[(ii)] A map \(\mathcal{L} \colon D \to M(S)\) satisfying
	\begin{align}\label{eq: Lf int}
	\int_0^t \big|\mathcal{L} f(X_s (\omega))\big|\dd s < \infty
	\end{align}
	for all \(t \in \mathbb{R}_+,\omega \in \Omega\) and \(f \in D\). We think of \(\mathcal{L}\) as a candidate for an extended generator in the spirit of \cite[Definition VII.1.8]{RY}.
	\item[(iii)] A set \(\Sigma \in \mathcal{F}\), which can be seen as the state space for the paths.
	\item[(iv)] A Borel probability measure \(\eta\) on \(S\), which we use as initial law.
\end{enumerate}
\begin{definition}\label{def: MP}
A probability measure \(P\) on \((\Omega, \mathcal{F})\) is called a solution to the MP \((D, \mathcal{L}, \Sigma, \eta)\) if \(P(\Sigma) =1, P \circ X^{-1}_0 = \eta\) and for all \(f \in D\) the process
\begin{align}\label{eq: Mf}
f(X_t) - f (X_0) - \int_0^t \mathcal{L} f(X_{s}) \dd s,\quad t \in \mathbb{R}_+, 
\end{align}
is a local \(P\)-martingale. When \(\eta = \delta_x\) for some \(x \in S\), we write \((D, \mathcal{L}, \Sigma, x)\) instead of \((D, \mathcal{L}, \Sigma, \delta_x)\). Here, \(\delta_x\) denotes the Dirac measure on the point \(x \in S\). 
\end{definition}
\begin{example}\label{expl: 1}
	The following MP corresponds to the classical MP of Stroock and Varadhan \cite{SV}.
Let \(S \triangleq \mathbb{R}^d, D \triangleq C^{2}_b(\mathbb{R}^d)\),
\begin{align}\label{eq: SV}
\mathcal{L} f (x) \triangleq \langle \nabla f(x), b(x)\rangle + \tfrac{1}{2} \textup{ trace } (\nabla^2 f(x) a(x)), 
\end{align}
where \(\nabla\) denotes the gradient, \(\nabla^2\) denotes the Hessian matrix and \(b\colon \mathbb{R}^d \to \mathbb{R}^d\) and \(a \colon \mathbb{R}^d \to \mathbb{S}^d\) are locally bounded Borel functions with \(\mathbb{S}^d\) denoting the set of all real symmetric non-negative definite \(d \times d\) matrices, and \(\Sigma \triangleq \{\omega \in \Omega \colon t \mapsto \omega(t)\text{ is continuous}\}\). 
We have \(\Sigma \in \mathcal{F}\), because \(\Sigma\) is a closed subset of \(\Omega\), see \cite[Problem 3.25]{EK}.
\end{example}

In the remaining of this article we impose the following assumption.
\begin{SA}
	For all \(x \in S\) the MP \((D, \mathcal{L}, \Sigma, x)\) has a solution \(P_x\). 
\end{SA} 
Conditions for the existence of solutions in diffusion settings can be found in \cite{KaraShre, RY, SV}. For conditions in jump-diffusions setups we refer to \cite{MatterIII, criens19a, kuhn18}. Conditions for abstract MPs can be found in \cite{EK}. For switching diffusions with state-independent switching we provide a Skorokhod-type existence result in Appendix \ref{appendix}.

\subsection{The Markov, the \(C_b\)-Feller and the Feller-Dynkin Property of MPs}\label{sec: 2.2}
The family \((P_x)_{x \in S}\) is called a \emph{Markov family} or simply \emph{Markov} if the map \(x \mapsto P_x(A)\) is Borel for all \(A \in \mathcal{F}\) and for all \(x\in S, t \in \mathbb{R}_+\)
and all \(G \in \mathcal{F}\) we have \(P_x\)-a.s.
 \begin{align}\label{eq: SMP}
 P_x \big(\theta^{-1}_t G |\mathcal{F}_t\big) = P_{X_t}(G),
 \end{align}
 where \(\theta_t \omega (s) \triangleq \omega(t + s)\) denotes the shift operator. 
We call \eqref{eq: SMP} the \emph{Markov property}.
 The family \((P_x)_{x \in S}\) is called a \emph{strong Markov family} or simply \emph{strongly Markov} if \((P_x)_{x \in S}\) is Markov and for all \(x \in S\), all stopping times \(\xi\) and all \(G \in \mathcal{F}\) we have \(P_x\)-a.s. on \(\{\xi < \infty\}\)
 \begin{align}\label{eq: strong MP}
P_x \big(\theta^{-1}_\xi G |\mathcal{F}_\xi\big) = P_{X_\xi} (G).
 \end{align}
 The identity \eqref{eq: strong MP} is called the \emph{strong Markov property}.
As the following proposition shows, many families of solutions to MPs are strongly Markov. For reader's convenience we provide a sketch of the proof, which mimics the proof of \cite[Theorem 4.4.2]{EK}.
 \begin{proposition}\label{prop: markov}
If \(D\) is countable, \(D \subseteq C_b(S), \mathcal{L}(D) \subseteq B_\textup{loc}(S)\), \((P_x)_{x \in S}\) is unique and 
\(
\Sigma \subseteq \theta^{-1}_\xi \Sigma 
\)
for all bounded stopping times \(\xi\), then \((P_x)_{x \in S}\) is strongly Markov.
\end{proposition}
\begin{proof}[Sketch of Proof]
	Due to Proposition \ref{prop: existence initial law} in Appendix \ref{app: prop exietence initial law}, the map \(x \mapsto P_x(A)\) is Borel for all \(A \in \mathcal{F}\) and, due to the argument used in the solution to \cite[Problem 2.6.9]{KaraShre} (see \cite[p. 121]{KaraShre}) it suffices to show the strong Markov property for all bounded stopping times.
		Let \(\xi\) be a bounded stopping time, set \(P \equiv P_x\) and fix \(F \in \mathcal{F}_\xi\) with \(P(F) > 0\). Using the argument from the proof of \cite[Lemma 5.4.19]{KaraShre} one checks that the probability measures 
		\[
		P_1 \triangleq \frac{E^P \big[ \1_F P (\theta^{-1}_\xi \cdot |\mathcal{F}_\xi)\big]}{P(F)},\quad P_2 \triangleq \frac{E^P \big[\1_F P_{X_\xi}\big]}{P(F)}
		\]
		both solve the MP \((D, \mathcal{L}, \Sigma, \zeta)\), where \(\zeta \triangleq P(F)^{-1} E^P [\1_F \1 \{X_\xi \in \ \cdot\ \}]\).
Due to Proposition \ref{prop: existence initial law} in Appendix \ref{app: prop exietence initial law}, we have
	\(
	P_1 = P_2,\)
	which implies that 
	\[
	E^P \big[ \1_F P(\theta_\xi^{-1} G |\mathcal{F}_\xi)\big] = E^P \big[\1_F P_{X_\xi} (G)\big],\quad G \in \mathcal{F}.
	\]
	Because this identity holds trivially when \(P(F) = 0\), for all \(G \in \mathcal{F}\) we conclude that \(P\)-a.s.
	\(
	P(\theta^{-1}_{\xi} G |\mathcal{F}_\xi) = P_{X_\xi} (G).
	\)
	In other words, the strong Markov property holds for all bounded stopping times. 
\end{proof}
 If \((P_x)_{x \in S}\) is not unique it might still be possible to pick a Markov family from the set of solutions. For instance, in the setting of Example \ref{expl: 1}, this is the case when \(a\) and \(b\) are bounded and continuous, see \cite[Theorem 12.2.3]{SV}. Conditions for the selection of a Markov family in jump-diffusion cases can be found in \cite{kuhn18}.
 
 In the case where \((P_x)_{x \in S}\) is Markov, we can define a semigroup \((T_t)_{t \geq 0}\) of positive contraction operators on \(B(S)\) via  \[T_t f(x) \triangleq E_x\big[f(X_t)\big],\quad f \in B(S).\] 
 It is obvious that \(T_t\) is a positive contraction, i.e. if \(f (S) \subseteq [0, 1]\) then also \(T_t f(S) \subseteq [0, 1]\), and the semigroup property follows easily from the Markov property \eqref{eq: SMP}. 
 
 If \((P_x)_{x \in S}\) is Markov and \begin{align}\label{eq: FP}
 T_t ( C_b(S)) \subseteq C_b(S),\end{align} we call \((P_x)_{x \in S}\) a \emph{\(C_b\)-Feller family} or simply \(C_b\)-\emph{Feller}. The inclusion \eqref{eq: FP} is called the \(C_b\)-\emph{Feller property}.
 The \(C_b\)-Feller property of the family \((P_x)_{x \in S}\) has a natural relation to the continuity of \(x \mapsto P_x\) for which many conditions are known, see, e.g., \cite[Theorem IX.4.8]{JS} for conditions in a jump diffusion setting. 
 Here, \(x \mapsto P_x\) is said to be continuous if \(P_{x_n} \to P_x\) weakly as \(n \to \infty\) whenever \(x_n \to x\) as \(n \to \infty\).
 In the setup of Example \ref{expl: 1}, if \((P_x)_{x \in S}\) is unique, \((P_x)_{x \in S}\) is \(C_b\)-Feller whenever \(b\) and \(a\) are continuous. However, in the same setting,  
 if \((P_x)_{x \in S}\) is not unique, it might not be possible to choose a \(C_b\)-Feller family from the set of solutions, even if the coefficients are continuous and bounded, see \cite[Exercise 12.4.2]{SV}. 
 
 We call \((P_x)_{x \in S}\) a \emph{Feller-Dynkin family} or simply \emph{Feller-Dynkin} if it is a \(C_b\)-Feller family and \begin{align}\label{eq: FDP}
 T_t(C_0(S)) \subseteq C_0(S).\end{align}
 The inclusion \eqref{eq: FDP} is called the \emph{Feller-Dynkin property}.
 From a semigroup point of view, the definition of a Feller-Dynkin semigroup also includes strong continuity in zero, see, e.g., \cite[Definition III.6.5]{RW1}. In our case, when \((P_x)_{x \in S}\) is Feller-Dynkin, the semigroup \((T_t)_{t \geq 0}\) is strongly continuous in zero due to the right-continuous paths of \((X_t)_{t \geq 0}\), the dominated convergence theorem and \cite[Lemma III.6.7]{RW1}.
Any Feller-Dynkin family is also strongly Markov, see, e.g., \cite[Theorem 17.17]{Kallenberg}.
Let us also comment on the issue of uniqueness. If \((P_x)_{x \in S}\) is Feller-Dynkin and \((L, \mathcal{D}(L))\) is its generator, i.e.
\begin{align}\label{eq: gen1}
Lf \triangleq \lim_{t \searrow 0} \frac{T_t f - f}{t}
\end{align}
for \(f \in \mathcal{D}(L)\), where \begin{align}\label{eq: gen2}
\mathcal{D}(L) \triangleq \left\{f \in C_0(S) \colon \exists g \in C_0(S) \text{ such that } \lim_{t \searrow 0} \bigg\| \frac{T_t f - f}{t} - g \bigg\|_\infty = 0\right\}, 
\end{align}
then \(P_x\) is the unique solution to the MP \((D, L, \Sigma, x)\), where \(D\) is any core for \(L\), see \cite[Theorem 4.10.3]{kolokol2011markov}. 
Consequently, conditions for the Feller-Dynkin property imply in some cases also uniqueness. 

For an overview on different concepts of Feller properties from a semigroup point of view we refer to the first chapter in \cite{MatterIII}.

	If \(S = \mathbb{R}^d\) and \((P_x)_{x \in \mathbb{R}^d}\) is Feller-Dynkin with generator \((L, \mathcal{D}(L))\) such that \(C_c^\infty(\mathbb{R}^d) \subseteq \mathcal{D}(L)\), then \(L\) is of the following form
	\[
	L f (x) = - \int e^{i \langle x, y\rangle} q(x, y) \hat{f}(y)\dd y, \quad f \in C_c^\infty(\mathbb{R}^d),
	\]
	where \(i\) is the imaginary number, \(\hat{f} (y) \triangleq (2 \pi)^{- d} \int e^{- i \langle y, x\rangle} f(x)\dd x\) denotes the Fourier transform of \(f\) and 
	\begin{align*}
	q(x, \xi) = q(x, 0) &- i \langle b(x), \xi\rangle + \tfrac{1}{2} \langle a(x) \xi, \xi\rangle \\&+ \int\big(1 - e^{i \langle y, \xi\rangle} + i \langle \xi, y\rangle \1 \{\|y\| \leq 1\}\big) K(x, \dd y)
	\end{align*}
	for a L\'evy triplet \((b(x), a(x), K(x, \dd y))\), see \cite[Corollary 2.23]{MatterIII}. 
	The function \(q\) is called the \emph{symbol} of the family \((P_x)_{x \in \mathbb{R}^d}\). 
	Starting with a candidate \(q\) for a symbol corresponds to a MP with input data \(\Sigma \triangleq \Omega, D \triangleq C_c^\infty(\mathbb{R}^d)\) and
	\[
	\mathcal{L} f (x) \triangleq - \int e^{i \langle x, y\rangle} q(x,y) \hat{f}(y)\dd y, \quad f \in  D. 
	\]
	We refer to the second and the third chapter of \cite{MatterIII} for a survey on the approach via the symbol.

Most of the general conditions for the Feller-Dynkin property are formulated in terms of the semigroup \((T_t)_{t \geq 0}\) and therefore are often not easy to check, see, e.g., \cite[Theorem 1.10]{MatterIII} and the discussion below its proof.  
In the following section we give a criterion for the Feller-Dynkin property in terms of the existence of Lyapunov functions. 
\section{Lyapunov Criteria for the Feller-Dynkin Property}\label{sec: main}
Lyapunov-type criteria often appear in the context of explosion, recurrence and transience of a Markov process, see, e.g., \cite{khasminskii2011stochastic, pinsky1995positive}. In this section we present such criteria for the Feller-Dynkin property of \((P_x)_{x \in S}\). We start with a sufficient condition.

 \begin{theorem}\label{theo: Ly}
 	Fix \(t \in \mathbb{R}_+\) and suppose that \(T_t(C_0(S)) \subseteq C(S)\). Assume that for any compact set \(K \subseteq S\) there exists a function \(V \colon S \to \mathbb{R}_+\) with the following properties:
 	\begin{enumerate}
 		\item[\textup{(i)}] \(V \in D \cap C_0(S)\).
 		\item[\textup{(ii)}] \(\underline{V} \triangleq \min_{x \in K} V(x) > 0\).
 		\item[\textup{(iii)}] \(\mathcal{L} V \leq c V\) for a constant \(c > 0\).
 	\end{enumerate}
 	Then, \(T_t (C_0(S)) \subseteq C_0(S)\). The function \(V\) is called a \emph{Lyapunov function for \(K\)}.
 \end{theorem}
 \begin{proof}
 	We first explain that it suffices to show that 
 			for all compact sets \(K \subseteq S\) and all \(\epsilon > 0\) there exists a compact set \(O \subseteq S\) such that 
 			\[
 			P_x (X_t \in K) < \epsilon
 			\]
 			for all \(x \not \in O\). 
 		To see this, let \(f \in C_0(S)\) and \(\epsilon > 0\). By the definition of \(C_0(S)\), there exists a compact set \(K \subseteq S\) such that \[|f(x)| < \tfrac{\epsilon}{2}\] for all \(x \not \in K\). By hypothesis, there exists a compact set \(O \subseteq S\) such that 
 		\[
 		\sup_{y \in S} |f(y)|\ P_x(X_t \in K) < \tfrac{\epsilon}{2}
 		\]
 		for all \(x \not \in O\). Thus, for all \(x \not \in O\) we have 
 		\begin{align*}
 		\big| E_x\big[ f(X_t)\big] \big| &\leq E_x \big[ |f(X_t)| \big(\1 \{X_t \in K\} + \1 \{X_t \not \in K\}\big)\big] \\&\leq \sup_{y \in S} |f(y)|\ P_x(X_t \in K) + \tfrac{\epsilon}{2} \\&< \epsilon.
 		\end{align*}
 		In other words, \(T_t f \in C_0(S)\), i.e. the claim is proven.
 		
 Next, we verify that this condition holds under the hypothesis of the theorem.
 	Fix \(x \in S\) and a compact set \(K \subseteq S\). Let \(V\) be as described in the prerequisites of the theorem.
 	The following lemma is an easy consequence of the integration by parts formula. For completeness, we give a proof after the proof of Theorem \ref{theo: Ly} is complete.
 	\begin{lemma}\label{rem: spacetime}
 		Assume that \(f \in C(S)\) and \(\mathcal{L} f \in M(S)\) are such that \eqref{eq: Lf int} holds and such that the process \eqref{eq: Mf} is a local martingale and that \(c \colon \mathbb{R}_+ \to \mathbb{R}\) is an absolutely continuous function with Lebesgue density \(c'\). Then, the process
 		\begin{align}\label{eq: ibp process}
 		f(X_t) c(t) - f(X_0) c(0) - \int_0^t \big(f(X_{s}) c' (s) + c(s)\mathcal{L} f(X_{s})\big) \dd s, \quad t \in \mathbb{R}_+,
 		\end{align}
 		is a local martingale. 
 	\end{lemma}
 	Since \(V \in D\), the definition of the martingale problem and Lemma \ref{rem: spacetime} imply that the process
 	\begin{align*}
 	Y_s \triangleq V(X_s) e^{-c s} - \int_0^s e^{- cr} \left(\mathcal{L} V(X_{r}) - c V(X_{r})\right) \dd r,\quad s \in \mathbb{R}_+,
 	\end{align*}
 	is a local \(P_x\)-martingale. Using (iii), we see that \(Y_s \geq V(X_s) e^{- cs} \geq 0\) for all \(s \in \mathbb{R}_+\). Thus, since non-negative local martingales are supermartingales due to Fatou's lemma, \((Y_s)_{s \geq 0}\) is a \(P_x\)-supermartingale.
 Using Markov's inequality, we obtain that
 	\begin{align*}
 	P_x(X_t \in K) &\leq P_x(V(X_t) \geq \underline{V})
 	\\&\leq \underline{V}^{-1}E_x \big[V(X_t)\big]
 	\\&\leq e^{c t} \underline{V}^{-1}E_x\big[Y_t\big]
 	\\&\leq e^{c t} \underline{V}^{-1} E_x\big[Y_0\big]
 	\\&= e^{c t} \underline{V}^{-1} V(x).
 	\end{align*}
 	Take an \(\epsilon > 0\). 
 	Since we assume that \(V \in C_0(S)\), there exists a compact set \(O \subseteq S\) such that 
 	\[
 	V(y) <  e^{- ct} \underline{V} \epsilon
 	\]
 	for all \(y \not \in O\).
 	We conclude that 
 	\[
 	P_x(X_t \in K) \leq e^{c t} \underline{V}^{-1} V(x) < \epsilon
 	\]
 	for all \(x \not \in O\). This finishes the proof.
 \end{proof}
\noindent
\textit{Proof of Lemma \ref{rem: spacetime}:}
Denote the local martingale \eqref{eq: Mf} by \((M_t)_{t \geq 0}\). Moreover, set \[N_t \triangleq \int_0^t \mathcal{L} f(X_s)\dd s, \quad t \in \mathbb{R}_+.\]
As an absolutely continuous function, \(c\) is of finite variation over finite intervals.
Thus, integration by parts yields that 
\begin{align*}
\dd \big(M_t c(t)\big)
&= c(t) \dd M_t + \big(f(X_t) - f(X_0)\big)c'(t) \dd t - \dd \big(N_t c(t) \big) + c(t) \mathcal{L} f(X_t)\dd t.
\end{align*}
We see that the process \eqref{eq: ibp process} equals the local martingale \((\int_0^t c(s) \dd M_s)_{t \geq 0}\).
\qed\\\\
Next, we give a condition for rejecting the Feller-Dynkin property.

\begin{theorem}\label{prop:cont}
	Suppose that \(S\) is not compact and that there exist compact sets \(K, C \subset S\), a constant \(\alpha >0\) and a bounded function \(U \colon S \to \mathbb{R}_+\) with the following properties:
	\begin{enumerate}
		\item[\textup{(i)}] \(U \in D\).
		\item[\textup{(ii)}] \(\max_{y \in K} U(y) > 0\).
		\item[\textup{(iii)}] \(\inf_{y \in S \backslash C} U(y) > 0\).
		\item[\textup{(iv)}] \(\mathcal{L} U \geq \alpha U\) on \(S \backslash K\).
	\end{enumerate}
	Then, \((P_x)_{x \in S}\) cannot be Feller-Dynkin. The function \(U\) is called a \textup{Lyapunov function for the sets \(K, C\)}.
\end{theorem}
\begin{proof}
	For contradiction, assume that \((P_x)_{x \in S}\) is Feller-Dynkin. For a moment we fix \(x \in S\).
	Let \((\mathcal{F}^x_t)_{t \geq 0}\) be the \(P_x\)-completion of \((\mathcal{F}_t)_{t \geq 0}\), i.e. 
	\begin{align}\label{eq: comp}
	\mathcal{F}^x_t \triangleq \sigma \big(\mathcal{F}_t, \mathcal{N}_x\big) = \bigcap_{s > t} \sigma \big(\mathcal{F}^o_s, \mathcal{N}_x\big), 
	\end{align}
	where
	\[
	\mathcal{N}_x \triangleq \big\{F \subseteq \Omega \colon \exists G \in \mathcal{F} \text{ with } F \subseteq G, P_x (G) = 0\big\}, 
	\]
	see \cite[Lemma 6.8]{Kallenberg} for the equality in \eqref{eq: comp}.
		We set
	\begin{align}\label{eq: tau}
	\tau \triangleq \inf \big(t \in \mathbb{R}_+ \colon X_t \in K\big),
	\end{align}
	which is well-known to be an \((\mathcal{F}^x_t)_{t \geq 0}\)-stopping time, see \cite[Theorem 6.7]{Kallenberg}.
	
	\textit{Step 1:} The proof of the following observation is given after the proof of Theorem \ref{prop:cont} is complete.
	\begin{proposition}\label{theo: conv}
		Assume that \((P_x)_{x \in S}\) is Feller-Dynkin and denote its generator by \((L, \mathcal{D}(L))\) (see \eqref{eq: gen1} and \eqref{eq: gen2}). 
		For any compact set \(K\subseteq S\) and any \(\alpha > 0\) there exists a function \(V\colon S \to \mathbb{R}_+\) with the following properties:
		\begin{enumerate}
			\item[\textup{(i)}] \(V \in \mathcal{D}(L)\).
			\item[\textup{(ii)}] \(\min_{y \in K} V(y) > 0\).
			\item[\textup{(iii)}] \(LV \leq \alpha V\).
		\end{enumerate}
	\end{proposition}
	Let \(V\) be as in Proposition \ref{theo: conv}. 
	Due to Dynkin's formula (see \cite[Proposition VII.1.6]{RY}) and Lemma \ref{rem: spacetime} the process
	\[
	Z_t \triangleq e^{- \alpha t} V(X_t) + \int_0^t e^{- \alpha s} \big(\alpha V(X_s) - L V(X_s)\big) \dd s, \quad t \in \mathbb{R}_+,
	\]
	is a local \(P_x\)-martingale. 
	We stress that the conclusion of Dynkin's formula also holds for the right-continuous filtration \((\mathcal{F}_t)_{t \geq 0}\), because any (right-continuous) \((\mathcal{F}^o_t)_{t \geq 0}\)-martingale is also an \(\F\)-martingale. This follows from the downward theorem (\cite[Theorem II.51.1]{RW1}) as in the proof of \cite[Lemma II.67.10]{RW1}.
	Because \((Z_t)_{t \geq 0}\) is bounded (recall that \(\mathcal{D}(L) \subseteq C_0(S)\) and that \(L f \in C_0(S)\) for all \(f \in \mathcal{D}(L)\)), the process \((Z_t)_{t \geq 0}\) is even a true \(P_x\)-martingale. Consequently, for \(s < t\) we have \(P_x\)-a.s.
	\begin{equation}\label{eq: super}
	\begin{split}
	E_x \big[ e^{- \alpha t} V(X_t) | \mathcal{F}_s \big] &\leq E_x \big[Z_t |\mathcal{F}_s \big] - \int_0^s e^{- \alpha r} \big(\alpha V(X_r) - L V(X_r)\big) \dd r
	\\&= Z_s - \int_0^s e^{- \alpha r} \big(\alpha V(X_r) - L V(X_r)\big) \dd r = e^{- \alpha s} V(X_s),
	\end{split}
	\end{equation}
	which implies that the process \((e^{- \alpha t} V(X_t))_{t \geq 0}\) is a non-negative \(P_x\)-supermartingale, which has a terminal value due to the submartingale convergence theorem (see, e.g., \cite[Theorem 1.3.15]{KaraShre}). 
	In particular, due to \cite[Lemma 67.10]{RW1}, \((e^{- \alpha t} V(X_t))_{t \geq 0}\) is also a non-negative bounded \(P_x\)-supermartingale for the filtration \((\mathcal{F}^x_t)_{t \geq 0}\). Recalling that \(\tau\) as defined in \eqref{eq: tau} is an \((\mathcal{F}^x_t)_{t \geq 0}\)-stopping time, we deduce from the optional stopping theorem (see, e.g., \cite[Theorem 6.29]{Kallenberg}) that 
	\begin{equation}\label{eq: bound1}
	\begin{split}
	V(x) &\geq E_x \big[ e^{- \alpha \tau} V(X_\tau)\big] \\&\geq E_x \big[ e^{- \alpha \tau} V(X_\tau) \1 \{\tau < \infty\} \big] \\&\geq E_x \big[e^{- \alpha \tau} \big] \min_{y \in K} V(y).
	\end{split}
	\end{equation}
	Here, we use the fact that \(X_\tau \in K\) on \(\{\tau < \infty\}\), which follows from the right-continuity of \((X_t)_{t \geq 0}\) because \(K\) is closed.
	
	\textit{Step 2:} 
	In the following all terms such as \emph{local martingale, submartingale}, etc. refer to \((\mathcal{F}^x_t)_{t \geq 0}\) as the underlying filtration.
	Lemma \ref{rem: spacetime} and \cite[Lemma 67.10]{RW1} imply that the stopped process
	\[
	Y_t \triangleq e^{- \alpha (t \wedge \tau)} U(X_{t \wedge \tau}) + \int_0^{t \wedge \tau} e^{- \alpha s} \big(\alpha U(X_s) - \mathcal{L} U(X_s)\big) \dd s, \quad t \in \mathbb{R}_+,
	\]
	is a local \(P_x\)-martingale. Due to property (iv) of the function \(U\), we have \begin{align*}
	Y_t \leq e^{- \alpha (t \wedge \tau)} U(X_{t \wedge \tau}) \leq \textup{const.}\end{align*}
	for all \(t \in \mathbb{R}_+\). 
	We note that local martingales bounded from above are submartingales. To see this, let \((M_t)_{t \geq 0}\) be a local martingale bounded from above by a constant \(c\). Then, the process \((c - M_t)_{t \geq 0}\) is a non-negative local martingale and hence a supermartingale by Fatou's lemma. This implies that \((M_t)_{t \geq 0}\) is submartingale.
	Therefore, the process \((Y_t)_{t \geq 0}\) is a \(P_x\)-submartingale and it follows similar to \eqref{eq: super} that the stopped process \((e^{- \alpha ( t \wedge \tau) } U(X_{t \wedge \tau}))_{t \geq 0}\) is a non-negative bounded \(P_x\)-submartingale, which has a terminal value \(e^{- \alpha \tau} U(X_\tau)\) by the submartingale convergence theorem. 
	Because \(U\) is bounded, we note that on \(\{\tau = \infty\}\) up to a null set we have \(e^{- \alpha \tau} U(X_\tau) = 0\). 
	Another application of the optional stopping theorem yields that
	\begin{equation}\label{eq: bound2}
	\begin{split}
	U(x) &\leq E_x \big[ e^{- \alpha \tau} U(X_\tau)\big]
	\\&= E_x \big[ e^{- \alpha \tau} U(X_\tau)\1 \{\tau < \infty\} \big]
	\\&\leq \max_{y \in K} U(y) E_x \big[e^{- \alpha \tau} \big].
	\end{split}
	\end{equation}
	
	\textit{Step 3:}
	We deduce from \eqref{eq: bound1} and \eqref{eq: bound2} that for all \(x \not \in C\) 
	\[
	\frac{\inf_{y \in S \backslash C} U(y)}{\max_{y \in K} U(y)} \leq E_x \big[e^{- \alpha \tau}\big] \leq \frac{V(x)}{\min_{y \in K} V(y)}.
	\]
	Because \(V \in \mathcal{D}(L) \subseteq C_0(S)\), we find a compact set \(G \subset S\) such that for all \(x \not \in G\)
	\[
	V(x) \leq \frac{1}{2} \frac{\inf_{y \in S \backslash C} U(y) \min_{y \in K} V(y)}{\max_{y \in K} U(y)} > 0,
	\]
	which implies that for all \(x \not \in C \cup G \not = S\)
	\[
	0 < \frac{\inf_{y \in S \backslash C} U(y)}{\max_{y \in K} U(y)} \leq E_x \big[e^{- \alpha \tau}\big] \leq \frac{1}{2} \frac{\inf_{y \in S \backslash C} U(y)}{\max_{y \in K} U(y)}.
	\]
	This is a contradiction and the proof of Theorem \ref{prop:cont} is complete.
\end{proof}
\noindent
\textit{Proof of Proposition \ref{theo: conv}:}
		We construct \(V\) via the \(\alpha\)-potential operator of \((T_t)_{t \geq 0}\), i.e. the operator \(U_\alpha \colon C_0(S) \to C_0(S)\) defined by 
		\[
		U_\alpha f (x) \triangleq \int_0^\infty e^{- \alpha s}T_s f(x) \dd s,\quad f \in C_0(S), x \in S.
		\]
		Take a function \(f \in C_0(S)\) with \(0 \leq f \leq 1\) and \(f \equiv 1\) on \(K\). 
		Such a function exists due to Urysohn's lemma for locally compact spaces (see, e.g., \cite[Proposition 7.1.9]{cohn13}).
		We set \(V \triangleq U_\alpha f\).
		It is well-known that \(V = U_\alpha f \in \mathcal{D}(L)\) and 
		\begin{align}\label{eq: res1}
		(\alpha \1 - L) V = (\alpha \1 - L) U_\alpha f = f \geq 0,
		\end{align}
		see, e.g., \cite[Proposition 6.12]{gall2016brownian}. Thus, \(V\) has the first and the third property.
		It remains to show that \(V\) has the second property. Since \(U_\alpha\) is positivity preserving we have \(V \geq 0\).
		For contradiction, assume that \(\min_{y \in K} V(y) = 0\). Then, there exists an \(x_0 \in K\) such that \(V (x_0) = 0\) and we obtain
			\[
			L V(x_0) = \lim_{t \searrow 0} \tfrac{1}{t}\big(T_t V (x_0) - V(x_0)\big) =  \lim_{t \searrow 0} \tfrac{1}{t} E_{x_0} \big[ V(X_t)\big] \geq 0.
			\] 
			Therefore, we conclude from \eqref{eq: res1} that
		\[
		\alpha V (x_0) = f (x_0) + LV (x_0) = 1 + L V(x_0) \geq 1. 
		\]
		This is a contradiction and it follows that \(V\) has also the second property.
\qed
\begin{remark}\label{rem: eqivalence}
The arguments from the proofs of Theorems \ref{theo: Ly} and \ref{prop:cont} imply a version of \cite[Proposition 3.1]{Azencott1974} beyond a diffusion setting. More precisely, when \((P_x)_{x \in S}\) is \(C_b\)-Feller, the following are equivalent:
\begin{enumerate}
	\item[\textup{(i)}]
	\((P_x)_{x \in S}\) is Feller-Dynkin.
	\item[\textup{(ii)}]
	For all compact sets \(K \subset S\) and all constants \(\alpha > 0\) the function \(x \mapsto E_x \big[e^{- \alpha \tau}\big]\) vanishes at infinity, where \(\tau\) is defined in \eqref{eq: tau}.
	\item[\textup{(iii)}]
	For all compact sets \(K \subset S\) and all constants \(\alpha > 0\) the function \(x \mapsto P_x(\tau \leq \alpha)\) vanishes at infinity, where \(\tau\) is defined in \eqref{eq: tau}.
\end{enumerate}
The implication (i) \(\Rightarrow\) (ii) is shown in the proof of Theorem \ref{prop:cont}. The implication (ii) \(\Rightarrow\) (iii) follows from the inequality
\[
P_x\big(\tau \leq \alpha\big) \leq e^{\alpha^2} E_x \big[ e^{- \alpha \tau} \1 \{\tau \leq \alpha\}\big] \leq e^{\alpha^2} E_x \big[ e^{- \alpha \tau}\big],
\]
and the final implication (iii) \(\Rightarrow\) (i) follows from the fact that 
\[
P_x \big(X_\alpha \in K\big) \leq P_x \big(\tau \leq \alpha\big)
\]
and the argument in the proof of Theorem \ref{theo: Ly}. A version of the equivalence of (i) and (iii) is also given in \cite[Theorem 4.8]{GH17}.
\end{remark}

In some cases Theorem \ref{theo: Ly} and Proposition \ref{theo: conv} can be combined to one sufficient and necessary Lyapunov-type condition for the Feller-Dynkin property:
\begin{example}\label{ex: Q ma}
	Suppose that \(S\) is a countable discrete space and let \(Q = (q_{ij})_{i, j \in S}\) be a conservative \(Q\)-matrix, i.e. \(q_{ij} \in \mathbb{R}_+\) for all \(i \not = j\) and \[- q_{ii}= \sum_{j \not = i} q_{ij} < \infty.\] Set \(\Sigma \triangleq \Omega\), \[D \triangleq \big\{ f \in C_0(S) \colon Q f \in C_0(S)\big\},
	\]
	and \(\mathcal{L} \triangleq Q,\)
	where \(Q f\) is defined by
	\begin{align}\label{eq: Q sum}
	Q f (i) = \sum_{j \in S} q_{ij} f(j).
	\end{align}
	We stress that the r.h.s. of \eqref{eq: Q sum} converges absolutely whenever \(f \in C_0(S)\).
	If \((P_x)_{x \in S}\) is Feller-Dynkin, the corresponding generator \((L, \mathcal{D}(L))\) is given by \((\mathcal{L}, D)\), see \cite[Theorem 5]{doi:10.1112/jlms/s2-5.2.267}.
	Thus, when \((P_x)_{x \in S}\) is Markov (or, equivalently, \(C_b\)-Feller, because of the discrete topology), Theorem \ref{theo: Ly} and Proposition \ref{theo: conv} imply that the following are equivalent:
	\begin{enumerate}
		\item[\textup{(i)}] \((P_x)_{x \in S}\) is Feller-Dynkin.
		\item[\textup{(ii)}] For each \(x \in S\) there exists a function \(V \colon S \to \mathbb{R}_+\) such that \(V \in D\), \(V(x) > 0\), \(QV \leq c V\) for a constant \(c > 0\).
	\end{enumerate}
	This observation is also contained in \cite[Theorem 3.2]{LI2009653}.
\end{example}
Under reasonable assumptions on the input data, we can deduce a related equivalence for more general martingale problems.
To formulate it we need further terminology.
By an extension of the input data \((D, \mathcal{L})\) we mean a pair \((D', \mathcal{L}')\) consisting of \(D '\subseteq C(S)\) and \(\mathcal{L}' \colon D' \to M(S)\) such that \(D \subseteq D'\), \(\mathcal{L}' = \mathcal{L}\) on \(D\),
\[
\int_0^t \big|\mathcal{L}' f(X_s(\omega))\big| \dd s < \infty
\] 
for all \(t \in \mathbb{R}_+, \omega \in \Omega\) and \(f \in D'\), and such that for all \(x \in S\) the probability measure \(P_x\) solves the MP \((D', \mathcal{L}', \Sigma, x)\). 

\begin{theorem}\label{coro: 1}
	Suppose that for all \(f \in D \cap C_0(S)\) we have \(\mathcal{L} f \in C_0(S)\) and that \((P_x)_{x \in S}\) is \(C_b\)-Feller. Then, the following are equivalent: 
	\begin{enumerate}
		\item[\textup{(i)}] \((P_x)_{x \in S}\) is Feller-Dynkin.
		\item[\textup{(ii)}] The input data \((D, \mathcal{L})\) can be extended such that for any compact set \(K \subset S\) a Lyapunov function for \(K\) in the sense of Theorem \ref{theo: Ly} exists. 
	\end{enumerate}
\end{theorem}
\begin{proof}
	The implication (ii) \(\Rightarrow\) (i) is due to Theorem \ref{theo: Ly}.
	Assume that (i) holds, let \((L, \mathcal{D}(L))\) be the generator of \((P_x)_{x \in S}\) and set \(D' \triangleq D \cup \mathcal{D}(L)\) and \[\mathcal{L}'f \triangleq \begin{cases} \mathcal{L}f,&f \in D,\\ Lf,& f \in \mathcal{D}(L).\end{cases}\]
	Of course, we have to explain that \(\mathcal{L}'\) is well-defined, i.e. that \(L f = \mathcal{L} f\) for all \(f \in D \cap \mathcal{D}(L)\). 
	Because \(\mathcal{L}f\in C_0(S)\) for any \(f \in D \cap \mathcal{D}(L)\) by assumption, the process 
	\[
	f (X_t) - f(x) - \int_0^t \mathcal{L} f(X_{s})\dd s, \quad t \in \mathbb{R}_+,
	\]
	is a \(P_x\)-martingale for all \(x \in S\), because it is a bounded (on finite time intervals) local \(P_x\)-martingale. Consequently, \cite[Proposition VII.1.7]{RY} implies \(\mathcal{L}f = Lf\). 
	Due to Dynkin's formula, \(P_x\) solves also the MP \((D', \mathcal{L}', \Sigma, x)\) for all \(x \in S\). In other words, \((D', \mathcal{L}')\) is an extension of \((D, \mathcal{L})\). Now, (ii) follows from Proposition \ref{theo: conv}.
\end{proof}
Let us comment on the prerequisites of the previous theorem.
Even if the coefficients are continuous, in the case of Example \ref{expl: 1} it is not always true that \(\mathcal{L}f \in C_0(\mathbb{R}^d)\) whenever \(f \in D \cap C_0(\mathbb{R}^d) = C^2_b(\mathbb{R}^d) \cap C_0(\mathbb{R}^d)\). However, if we could take \(D = C_c^2(\mathbb{R}^d)\) instead of \(D = C^2_b(\mathbb{R}^d)\), then \(\mathcal{L} f \in C_0(\mathbb{R}^d)\) holds for all \(f \in D = D \cap C_0(\mathbb{R}^d)\) provided the coefficients are continuous. In other words, when we could reduce the input data, we would get an equivalent characterization of the Feller-Dynkin property from Theorem \ref{coro: 1}. Next, we explain that a reduction of the input data is often possible.

A sequence \((f_n)_{n \in \mathbb{N}} \subset M(S)\) is said to converge \emph{locally bounded pointwise} to a function \(f \in M(S)\) if 
\begin{enumerate}
	\item[(i)]  \(\sup_{n \in \mathbb{N}} \sup_{y \in K}|f_n (y)| < \infty\) for all compact sets \(K \subseteq S\); \item[(ii)] \(\lim_{n \to \infty} f_n(x) = f(x)\) for all \(x \in S\). \end{enumerate}
Moreover, we say that \((f_n)_{n \in \mathbb{N}} \subset B(S)\) converges \emph{bounded pointwise} to \(f \in M(S)\) if \(f_n \to f\) as \(n \to \infty\) locally bounded pointwise and
 \(\sup_{n \in \mathbb{N}} \|f_n\|_\infty< \infty\).

For a set \(A \subseteq C(S) \times M(S)\) we denote by \(\textup{cl} (A)\) the set of all \((f, g) \in C(S) \times M(S)\) for which there exist sequences \((f_n, g_n)_{n \in \mathbb{N}} \subset A\) such that \(f_n \to f\) as \(n \to \infty\) bounded pointwise and \(g_n \to g\) as \(n \to \infty\) locally bounded pointwise.
The following proposition can be viewed as an extension of \cite[Proposition 4.3.1]{EK}, which allows a local convergence in the second variable. 
\begin{proposition}\label{prop: redc}
Let \(D_1, D_2 \subseteq C(S), \mathcal{L}_1 \colon D_1 \to M(S)\) and \(\mathcal{L}_2 \colon D_2 \to M(S)\) be such that 
\[
\int_0^t \big(\big|\mathcal{L}_1 f (X_s(\omega))\big| + \big|\mathcal{L}_2 g(X_s(\omega))\big| \big) \dd s < \infty
\]
for all \(t \in \mathbb{R}_+, \omega \in \Omega, f \in D_1\) and \(g \in D_2\).
Suppose that 
\begin{align}\label{eq: incl}
\{(f, \mathcal{L}_2 f) \colon f \in D_2\} \subseteq \textup{cl}(\{(f, \mathcal{L}_1f) \colon f \in D_1\}).
\end{align}
If \(P\) is a solution to the MP \((D_1, \mathcal{L}_1, \Sigma, \eta)\), then \(P\) is also a solution to the MP \((D_2, \mathcal{L}_2, \Sigma, \eta)\).
\end{proposition}
\begin{proof}
	Due to \cite[Proposition 6.2.10]{waldmann2014topology}, there exists a sequence \((K_n)_{n \in \mathbb{N}} \subset S\) of compact sets such that \(K_n \subset \textup{int}(K_{n+1})\) and \(\bigcup_{n \in \mathbb{N}} K_n = S\). Now, define
	\begin{align}\label{eq: tau mit K}
	\tau_n \triangleq \inf\big(t \in \mathbb{R}_+ \colon X_t \not \in \textup{int}(K_n) \text{ or } X_{t-} \not \in \textup{int}(K_n)\big), \quad n \in \mathbb{N}.
	\end{align}
	It is well-known that \(\tau_n\) is a stopping time, see \cite[Proposition 2.1.5]{EK}, and that \(\tau_n \nearrow \infty\) as \(n \to \infty\), see \cite[Problem 4.27]{EK}. Take \(f \in D_2\). Due to \eqref{eq: incl} there exists a sequence \((f_n)_{n \in \mathbb{N}} \subset D_1\) such that \(f_n \to f\) as \(n \to \infty\) bounded pointwise and \(\mathcal{L}_1 f_n \to \mathcal{L}_2 f\) as \(n \to \infty\) locally bounded pointwise.
	For \(i = 1, 2\) and \(g \in D_i\) we set
	\[
	M^{g, i}_t \triangleq g(X_{t}) - g(X_0)- \int_0^{t} \mathcal{L}_ig(X_{s}) \dd s, \quad t \in \mathbb{R}_+.
	\]
	Since the class of local martingales is stable under stopping, the process \((M^{f_n, 1}_{t \wedge \tau_m})_{t \geq 0}\) is a local \(P\)-martingale. Furthermore, 
	\[
	\sup_{s \in [0, t]}\big|M^{f_n, 1}_{s \wedge \tau_m}\big| \leq 2\sup_{k \in \mathbb{N}} \|f_k\|_\infty + t \ \sup_{k\in \mathbb{N}} \sup_{y \in K_m} |\mathcal{L}_1f_k(y)| < \infty, 
	\]
	by the definition of (local) bounded pointwise convergence. Consequently, \((M^{f_n, 1}_{t \wedge \tau_m})_{t \geq 0}\) is a \(P\)-martingale by the dominated convergence theorem. 
	Since \[\sup_{s \in [0, t \wedge \tau_m)} |\mathcal{L}_1f_n(X_{s-})| \leq \sup_{k \in \mathbb{N}} \sup_{y \in K_m} |\mathcal{L}_1f_k(y)| < \infty,\]
	the dominated convergence theorem also yields that for any \(t \in \mathbb{R}_+\) we have \(\omega\)-wise \(M^{f_n, 1}_{t \wedge \tau_m} \to M^{f, 2}_{t \wedge \tau_m}\) as \(n \to \infty\).
	Thus, for all \(s < t\), applying the dominated convergence theorem a third time yields that \(M^{f, 2}_{t \wedge \tau_m}, M^{f, 2}_{s \wedge \tau_m} \in L^1(P)\) and that 
	for all \(G \in \mathcal{F}_s\) 
	\[
	E^P \big[M^{f, 2}_{t \wedge \tau_m} \1_G\big] = \lim_{n \to \infty} E^P \big[ M^{f_n,1}_{t \wedge \tau_m}\1_G\big] = \lim_{n \to \infty} E^P \big[ M^{f_n, 1}_{s \wedge \tau_m} \1_G\big] = E^P \big[M^{f, 2}_{s \wedge \tau_m}\1_G\big].
	\]
	In other words, the stopped process \((M^{f, 2}_{t \wedge \tau_m})_{t \geq 0}\) is a \(P\)-martingale. Because \(\tau_m \nearrow \infty\) as \(m \to \infty\), we conclude that \(P\) solves the MP \((D_2, \mathcal{L}_2, \Sigma, x)\). 
\end{proof}
\begin{exampleN}
	We have
	\[
	\big\{(f, \mathcal{L}f) \colon f \in C_b^2(\mathbb{R}^d)\big\} \subseteq \textup{cl}\big(\big\{(f, \mathcal{L}f)\colon f \in C^2_c(\mathbb{R}^d)\big\}\big).
	\]
	To see this, let \(g_n \in C^2_c(\mathbb{R}^d)\) be such that \(0 \leq g_n \leq 1\) and \(g_n \equiv 1\) on \(\{x \in \mathbb{R}^d \colon \|x\| \leq n\}\). For any \(f \in C_b^2(\mathbb{R}^d)\) it is easy to verify that \(f_n \triangleq f g_n \in C_c^2(\mathbb{R}^d)\), \(f_n \to f\) as \(n \to \infty\) bounded pointwise and \(\mathcal{L}f_n \to \mathcal{L}f\) as \(n \to \infty\) locally bounded pointwise. Consequently, a Borel probability measure on \(\Omega\) solves the MP \((C^2_b(\mathbb{R}^d), \mathcal{L}, \Sigma, \eta)\) if and only if it solves the MP \((C^2_c(\mathbb{R}^d), \mathcal{L}, \Sigma, \eta)\). This fact is of course well-known, see, e.g., \cite[Proposition 5.4.11]{KaraShre}. 
	In summary, if the family \((P_x)_{x \in \mathbb{R}^d}\) is unique and \(b\) and \(a\) are continuous, then \((P_x)_{x \in \mathbb{R}^d}\) is \(C_b\)-Feller (see \cite[Corollary 11.1.5]{SV}) and Theorem \ref{coro: 1} implies that the following are equivalent:
	\begin{enumerate}
		\item[\textup{(i)}]
		\((P_x)_{x \in \mathbb{R}^d}\) is Feller-Dynkin.
		\item[\textup{(ii)}]
		The input data \((C_c^2(\mathbb{R}^d), \mathcal{L})\) can be extended such that for all compact sets \(K \subset \mathbb{R}^d\) a Lyapunov function for \(K\) (in the sense of Theorem \ref{theo: Ly}) exists.
	\end{enumerate}
\end{exampleN}

The larger the set \(D\), the easier it is to find a suitable Lyapunov function and to apply Theorems \ref{theo: Ly} and \ref{prop:cont}. Thus, when we have applications in mind, we would like to choose \(D\) as large as possible. We stress that Proposition \ref{prop: redc} also works in this direction, i.e. it gives a condition such that \(D\) can be enlarged.

Proposition \ref{prop: redc} can also be used to verify the prerequisites of Proposition \ref{prop: markov} as the following example shows.
\begin{example}\label{ex: MC}
	Suppose that 
	\[
	A \triangleq \big\{ (f, \mathcal{L} f) \colon f \in D\big\} \subseteq C_0(S) \times C_0(S).
	\]
	Because \(C_0(S)\) endowed with the uniform metric is a separable metric space, the space \(A\) is a separable metric space endowed with the metric \(d\) given by
	\[
	d((f_1, g_1), (f_2, g_2)) \triangleq \|f_1  - f_2\|_\infty + \|g_1 - g_2\|_\infty, \quad (f_1, g_1), (f_2, g_2) \in A.
	\]
	Consequently, we find a countable set \(C \subseteq D\) such that for any \((f, g) \in A\) there exists a sequence \((f_n)_{n \in \mathbb{N}} \subset C\) with
	\[
	d ((f_n, \mathcal{L} f_n), (f, g)) \to 0 
	\]
	as \(n \to \infty\). Now, Proposition \ref{prop: redc} implies that a Borel probability measure on \(\Omega\) solves the MP \((D, \mathcal{L}, \Sigma, \eta)\) if and only if it solves the MP \((C, \mathcal{L}, \Sigma, \eta)\).
\end{example}

Let us summarize the observation from this section. We have seen a sufficient condition for the Feller-Dynkin property (Theorem \ref{theo: Ly}) and a sufficient condition to reject the Feller-Dynkin property (Theorem \ref{prop:cont}).
Moreover, we gave one sufficient and necessary condition under some additional assumptions (Theorem \ref{coro: 1}) and discussed its prerequisites (Proposition \ref{prop: redc}).

\section{The Feller-Dynkin Property of Switching Diffusions}\label{sec: app}
In this section we derive Khasminskii-type integral tests for the Feller-Dynkin property of diffusions with random switching. Moreover, we give an equivalent characterization for the state-independent case and present equivalent integral-type conditions for the Feller-Dynkin property for one dimensional state-independent switching diffusions.

Before we start our program, we fix some notation.
Let \(S_d\) be a countable discrete space and let \(S \triangleq \mathbb{R}^d \times S_d\)
 equipped with the product topology.
Take the following coefficients:
\begin{enumerate}
	\item[\textup{(i)}]
	\(b \colon S\to \mathbb{R}^d\) being Borel and locally bounded.
	\item[\textup{(ii)}]
	\(a \colon S \to \mathbb{S}^d\) being Borel and locally bounded.
	\item[\textup{(iii)}]
	For each \(x \in \mathbb{R}^d\), let \(Q(x) = (q_{ij}(x))_{i, j \in S_d}\) be a conservative \(Q\)-matrix (see Example \ref{ex: Q ma} for a definition), such that the map \(x \mapsto Q(x)\) is Borel.
\end{enumerate}

\subsection{Conditions for the Feller-Dynkin Property}\label{sec: FDP}
For \(i, j \in S_d\), we set
\begin{align*}
\overline{q}_{ij} \triangleq \begin{cases} \sup_{x \in \mathbb{R}^d} q_{ij}(x),& i \not = j,\\
- \sum_{k \not = i}\overline{q}_{ik},&i = j.
\end{cases}
\end{align*}
In this section, we impose the following standing assumption.
\begin{SA}
	For all \(i \in S_d\) we have \(|\overline{q}_{ii}| < \infty\) and
	\begin{align}\label{eq: assp Q}
	\sup_{j \in S_d} \sup_{x \in \mathbb{R}^d} |q_{jj}(x) - \overline{q}_{jj}| < \infty.
	\end{align}
\end{SA}
 Set \(\overline{Q}\triangleq (\overline{q}_{ij})_{i, j \in S_d}\) and note that \(\overline{Q}\) is a conservative \(Q\)-matrix. Denote 
\[C \triangleq \big\{ f \in C_0(S_d) \colon \overline{Q} f \in C_0(S_d)\big\}\]
and
\[
\Sigma_d \triangleq \big\{ \omega \colon \mathbb{R}_+ \to S_d\colon t \mapsto \omega(t) \text{ is  c\`adl\`ag}\big\}.
\]

We also impose the following standing assumption.
\begin{SA}
	 For all \(i \in S_d\) the MP \((C, \overline{Q}, \Sigma_d, i)\) has a unique solution \(P^d_i\) such that the family \((P^d_i)_{i \in S_d}\) is Feller-Dynkin. Here, the state space for the MP is assumed to be \(S_d\).
\end{SA}
If \(|S_d| < \infty\) this standing assumption holds. In the following remark we collect also some conditions when the previous standing assumption holds for the case \(|S_d| = \infty\).
 \begin{remark}\label{rem: new cond}
	\begin{enumerate}
		\item[(i)] Conditions for the existence of \((P^d_i)_{i \in S_d}\) can be found in \cite[Corollary 2.2.5, Theorem 2.2.27]{anderson2012continuous} and \cite[Theorem 16]{Mufa1986}.
		If, in addition to one of these conditions, we have
		\begin{align}\label{eq: assp Q2}
		\forall \lambda > 0, k \in S_d, \ \ \{y \in l_1 \colon y(\lambda \1 - \overline{Q}) = 0\} = \{0\} \text{ and } \overline{q}_{\cdot k} \in C_0(S_d),
		\end{align}
		then \((P^d_i)_{i \in S_d}\) is Feller-Dynkin, see \cite[Theorem 8]{doi:10.1112/jlms/s2-5.2.267}. Here, \(l_1\) denotes the space of all functions \(f \colon S_d \to \mathbb{R}\) with \(\sum_{i \in S_d} |f(i)| < \infty\).
		\item[(ii)]
		If \(\sup_{n \in S_d} |\overline{q}_{nn}| < \infty\), then \((P^d_i)_{i \in S_d}\) exists, see \cite[Corollary 2.2.5, Proposition 2.2.9]{anderson2012continuous}, and \(\{y \in l_1 \colon y(\lambda \1 - \overline{Q}) = 0\} = \{0\}\) holds for all \(\lambda > 0\), see \cite[pp. 273]{doi:10.1112/jlms/s2-5.2.267}. In this case, the second part of \eqref{eq: assp Q2} is necessary and sufficient for \((P^d_i)_{i \in S_d}\) to be Feller-Dynkin, see \cite[Theorem 9]{doi:10.1112/jlms/s2-5.2.267}.
		\item[(iii)]
		If
		\(S_d = \{0,1, 2, \dots\}\) and \(\overline{q}_{ij} = 0\) for all \(i \geq j + 2\), then \cite[Proposition 2]{LI2006461} yields that
		the following are equivalent:
		\begin{enumerate} \item[(a)]
		\(\{y \in l_1 \colon y(\lambda \1 - \overline{Q}) = 0\} = \{0\}\). 
		\item[(b)]
		\(\{y \in l_1^+ \colon y(\lambda \1 - \overline{Q}) = 0\} = \{0\}\). 
	\end{enumerate}
	Part (b) is necessary for \((P^d_i)_{i \in S_d}\) to be Feller-Dynkin, see \cite[Theorem 7]{doi:10.1112/jlms/s2-5.2.267}. Here, \(l^+_1\) denotes the set of all non-negative \(f \in l_1\).
	\end{enumerate}
\end{remark}
For reader's convenience we recall our notation: \(\Omega\) denotes the space of all \cadlag functions \(\mathbb{R}_+ \to S\) equipped with the Skorokhod topology, \(\mathcal{F}\) is the corresponding Borel \(\sigma\)-field, \((X_t)_{t \geq 0}\) is the coordinate process on \(\Omega\) and \(D\subseteq C(S)\) is a set of test functions. 

We suppose that
\begin{align*}
\big\{f, fg, g \colon f \in C_b^2(\mathbb{R}^d), g \in C\big\} \subseteq D,\end{align*}
and set
\[
\Sigma \triangleq \big\{(\omega^1, \omega^2) \in \Omega \colon \omega^1 \colon \mathbb{R}_+ \to \mathbb{R}^d \text{ is continuous}\big\}
\]
and
\begin{align}\label{gen: sd}
\mathcal{L} f (x, i) &\triangleq \mathcal{K} f (x, i) + \sum_{j \in S_d} q_{ij} (x) f(x, j),\quad (x, i) \in S,
\end{align}
where
\[
\mathcal{K} f(x, i) \triangleq \langle \nabla_x f(x, i), b(x, i)\rangle + \tfrac{1}{2} \textup{trace } (\nabla^2_x f(x, i) a(x, i)),\quad (x,i) \in S.
\]
In the proof of Lemma \ref{lem: uni1} below we explain that \(\Sigma\) is closed, which yields \(\Sigma \in \mathcal{F}\).
Reecalling the Standing Assumption in Section \ref{sec: setup}, we assume that for each \(x \in S\) there exists a solution \(P_x\) to the MP \((D, \mathcal{L}, \Sigma, x)\).

By our assumption that \((P^d_x)_{x \in S_d}\) is Feller-Dynkin, due to Proposition \ref{theo: conv} (see also Example \ref{ex: Q ma}), for any compact subset of \(S_d\) there exists a Lyapunov function (in the sense of Theorem \ref{theo: Ly}) for \((P^d_x)_{x \in S_d}\). We will combine these Lyapunov functions with Lyapunov functions for the diffusion part, which we can define under each of the following two conditions.
\begin{condition}\label{cond: suff1}
	There exist two locally H\"older continuous functions \(a_d \colon [\tfrac{1}{2}, \infty) \to (0, \infty)\) and \(b_d \colon [\tfrac{1}{2}, \infty) \to \mathbb{R}\) such that
	\begin{align*}
	\langle  x, a(x, i)x\rangle &\leq a_d \left(\tfrac{\|x\|^2}{2}\right),\\
	\textup{trace } a(x, i) + 2 \langle x, b(x, i)\rangle &\geq b_d\left(\tfrac{\|x\|^2}{2}\right) \langle x, a(x, i)x\rangle
	\end{align*}
	for all \(i \in S_d\) and \(x \in \mathbb{R}^d \colon \|x\| \geq 1\). Moreover, either
	\begin{align} \label{eq: FC1}
	p(r) \triangleq \int_1^r \exp \left(- \int_1^y b_d(z) \dd z\right) \dd y, \quad \lim_{r \to \infty} p(r) < \infty,
	\end{align}
	or \begin{align} \label{eq: FC2} 
	\lim_{r \to \infty} p(r) = \infty \text{ and } \int_1^\infty p'(y) \int_y^\infty \frac{\dd z}{a_d(z) p'(z)} \dd y = \infty.
	\end{align}
	Furthermore, we have 
	\begin{align}\label{eq: small bdd}
	\sup_{j \in S_d} \sup_{\|x\| \leq 1} \left(\|b(x, j)\| + \textup{trace } a(x, j) \right)< \infty.
	\end{align}
\end{condition}
\begin{condition}\label{cond: suff2} There exists a constant \(\beta > 0\) such that 
\begin{align*} 
\|b(x, i)\| \leq \beta (1 + \|x\|),\quad \textup{trace } a(x, i) \leq \beta (1 + \|x\|^2),
\end{align*}
for all \((x, i) \in S\). 
\end{condition}

\begin{proposition}\label{prop: sd}
If the family \((P_x)_{x \in S}\) is \(C_b\)-Feller and one of the Conditions \ref{cond: suff1} and \ref{cond: suff2} holds, then \((P_x)_{x \in S}\) is also Feller-Dynkin.
\end{proposition}
\begin{proof}
	We assume that Condition \ref{cond: suff1} holds.
	Fix an arbitrary compact set \(K \subset S\). 
	Since the projections \(\pi_1 \colon S \to \mathbb{R}^d\) and \(\pi_2 \colon S \to S_d\) are continuous for the product topology, the sets \(\pi_1(K)\) and \(\pi_2(K)\) are compact and \(K \subseteq \pi_1(K) \times \pi_2(K)\). 
	
	Because we assume the family \((P^d_x)_{x \in S_d}\) to be Feller-Dynkin,
	Proposition \ref{theo: conv} (see also Example \ref{ex: Q ma}) implies that there exists a function \(\zeta \colon S_d \to \mathbb{R}_+\) such that \(\zeta \in C, \zeta > 0\) on \(\pi_2(K)\) and \(\overline{Q} \zeta \leq c \zeta\) for a constant \(c > 0\).
	Applying the change of variable as explained in \cite[Section 4.1]{Azencott1974} together with \cite[Lemma 4.2]{Azencott1974}, we obtain that there exists a twice continuously differentiable decreasing solution \(u \colon [\frac{1}{2}, \infty) \to (0, \infty)\) to the differential equation
	\begin{align}\label{eq: ODE2}
	\tfrac{1}{2} a_d b_d u' + \tfrac{1}{2} a_d u'' = u, \quad u\big(\tfrac{1}{2}\big) = 1,
	\end{align}
	which satisfies \(\lim_{x \nearrow + \infty} u(x) = 0\). For the last property we require that either \eqref{eq: FC1} or \eqref{eq: FC2} holds. 
	We find a twice continuously differentiable function \(\phi \colon [0, \infty) \to (0, \infty)\) such that \(\phi \geq 1\) on \([0, \tfrac{1}{2}]\) and \(\phi = u\) on \((\tfrac{1}{2}, \infty)\).
		Now, we define 
	\[
	V(x, i) \triangleq \phi \left(\tfrac{\|x\|^2}{2}\right) \zeta (i),\quad (x, i) \in S.
	\]
	We see that \(V\geq 0\), \(V \in D\) and that \(V > 0\) on \(K\)
	and one readily checks that \(V \in C_0(S)\).
	It remains to show that \(\mathcal{L} V \leq \textup{const. } V\).
	For all \(i \in S_d\) and \(x \in \mathbb{R}^d \colon \|x\| > 1\) we have
	\begin{align*}
	\mathcal{K} V (x, i)
	&= \zeta (i) \tfrac{1}{2} \Big(\langle x,a(x, i) x\rangle u'' \left(\tfrac{\|x\|^2}{2}\right)+ \left(\textup{trace } a(x, i) + 2 \langle x, b(x, i)\rangle\right) u' \left(\tfrac{\|x\|^2}{2}\right) \Big)  
	\\&\leq \zeta (i) \tfrac{\langle x,a(x, i) x\rangle}{2}\left(u''\left(\tfrac{\|x\|^2}{2}\right) + b_d \left(\tfrac{\|x\|^2}{2}\right) u' \left(\tfrac{\|x\|^2}{2}\right)\right),
	\end{align*}
	where we used that \(u\) is decreasing, i.e. that \(u' \leq 0\). 
	Due to \eqref{eq: ODE2}, we have
	\[
	u'' + b_d u' = \tfrac{2 u}{a_d} \geq 0.
	\]
	Thus, we obtain
	\begin{align}\label{eq: L bdd}
	\mathcal{K} V(x, i) &\leq \zeta(i) \tfrac{1}{2} a_d \left(\tfrac{\|x\|^2}{2}\right) \left(u''\left(\tfrac{\|x\|^2}{2}\right) + b_d \left(\tfrac{\|x\|^2}{2}\right) u' \left(\tfrac{\|x\|^2}{2}\right)\right) = V(x, i)
	\end{align}
	for all \(i \in S_d\) and \(x \in \mathbb{R}^d \colon \|x\| > 1\).
	Due to \eqref{eq: small bdd}, we find a constant \(c^* \geq 1\) such that \(\mathcal{K} V (x, i) \leq c^* \zeta(i) \leq c^* V(x, i)\) for all \(i \in S_d\) and \(x \in \mathbb{R}^d \colon \|x\| \leq 1\).
	In summary, using
	\eqref{eq: assp Q} and \eqref{eq: L bdd}, we obtain
	\begin{align*}
	\mathcal{L} V(x, i) &\leq c^* V(x, i) + \bigg(\sum_{j \not = i} q_{ij} (x) \zeta (j) + q_{ii}(x) \zeta(i)\bigg) \phi \left(\tfrac{\|x\|^2}{2}\right) 
	\\&\leq c^* V(x, i) + \bigg(\sum_{j \in S_d} \overline{q}_{ij} \zeta (j) + (q_{ii}(x) - \overline{q}_{ii}) \zeta(i)\bigg)\phi \left(\tfrac{\|x\|^2}{2}\right) 
	\\&	\leq  \bigg( c^*+ c + \sup_{j \in S_d} \sup_{y \in \mathbb{R}^d}|q_{jj} (y) - \overline{q}_{jj}| \bigg) V(x, i)
	= \text{const. } V(x, i).
	\end{align*}
	Consequently, Theorem \ref{theo: Ly} implies the claim.
	
	For the case where Condition \ref{cond: suff2} holds, we only have to replace \(\phi(x)\) by \((1 + 2x)^{-1}\). The remaining argument stays unchanged. We omit the details.
\end{proof}
Conditions for the \(C_b\)-Feller property of \((P_x)_{x \in S}\) can be found in \cite{doi:10.1137/16M1059357, doi:10.1137/15M1013584, doi:10.1137/16M1087837,yin2009hybrid}.
We collect some of these in the following corollary, where we also assume that 
\begin{align*}
D \equiv  \big\{f \colon S \to \mathbb{R} \colon &x \mapsto f(x, j) \in C^2_b(\mathbb{R}^d), i \mapsto f(y, i) \in B(S_d) 
\text{ for all } (y, j) \in S \big\}.
\end{align*}
\begin{corollary}\label{coro: sd}
Suppose the following:
\begin{enumerate}
	\item[\textup{(i)}] \(S_d = \{0, 1, \dots, N\}\) for \(1 \leq N \leq \infty\), where we mean \(S_d = \mathbb{N}_0\) when \(N = \infty\).
	\item[\textup{(ii)}] There exists a constant \(c_1 > 0\) such that for all \((x, i) \in S\) we have \(q_{ij} (x) = 0\) for all \(j \in S_d\) with \(|j - i| > c_1\).
		\item[\textup{(iii)}] There exits a constant \(c_2 > 0\) such that for all \(i \in S_d\)
	\[
	\sup_{x \in \mathbb{R}_d} |q_{ii}(x)| \leq c_2 (i + 1).
	\] 
	\item[\textup{(iv)}] There exists a constant \(c_3 > 0\) such that for all \(i \in S_d\) and \(x, y \in \mathbb{R}^d\)
	\[
\sum_{j \not= i}	|q_{ij}(x) - q_{ij}(y)| \leq c_3 \|x - y\|.
	\]
	\item[\textup{(v)}] Condition \ref{cond: suff2} holds and there exists a constant \(c_4 > 0\) and a root \(a^\frac{1}{2}\) of \(a\) such that for all \(i \in S_d\) and \(x, y \in \mathbb{R}^d\)
	\[
	\|b(x, i) - b(x, i)\| + \|a^\frac{1}{2}(x, i) - a^\frac{1}{2} (y, i)\| \leq c_4 \|x - y\|.
	\]
\end{enumerate}
Then, a Feller-Dynkin family \((P_x)_{x \in S}\) exists.
\end{corollary}
\begin{proof}
	The existence of a family \((P_x)_{x \in S}\) follows from  \cite[Theorem 2.1]{doi:10.1137/16M1087837}. Furthermore, \cite[Theorem 3.3]{doi:10.1137/16M1087837} yields that \((P_x)_{x \in S}\) is \(C_b\)-Feller. Thus, Proposition \ref{prop: sd} implies that \((P_x)_{x \in S}\) is Feller-Dynkin, too.
\end{proof}
\begin{remark}\label{rem: FnSM}
	\begin{enumerate}
		\item[\textup{(i)}]
		Assumption (ii) in Corollary \ref{coro: sd} can be replaced by a weaker, but less explicit, condition of Lyapunov-type, see \cite[Assumption 1.2]{doi:10.1137/16M1087837}.		
		\item[\textup{(ii)}] In general, the conditions from Corollary \ref{coro: sd} do not imply the strong Feller property 
		of \((P_x)_{x \in S}\). For example, it is allowed to take the first coordinate as linear motion, which gives a process without the strong Feller property.
		
		If, in addition to (i) -- (v) in Corollary \ref{coro: sd}, we assume that there exists a constant \(c > 0\) such that for all \((x, i) \in S\) and \(y \in \mathbb{R}^d\)
			\[
			\langle y, a(x, i) y\rangle \geq c \|y\|^2,
			\]
		then \cite[Theorem 3.1]{doi:10.1137/15M1013584} implies that \((P_x)_{x \in S}\) has the strong Feller property, too. In this case, \((P_x)_{x \in S}\) has the \(C_b\)-Feller, the strong Feller and the Feller-Dynkin property. 	
	\end{enumerate}
\end{remark}
The following example illustrates that our results include cases where \(Q\) is unbounded.
\begin{example}\label{ex: unbounded}
Suppose that \(\overline{Q}\) corresponds to a classical birth-death chain, i.e. \(S_d \triangleq \{0, 1,2, \dots\}\) and
\[
\overline{q}_{ij} \triangleq \begin{cases} \lambda_i,& j = i + 1, i \geq 0,\\
\mu_i,&j = i-1, i \geq 1,\\
- (\lambda_i + \mu_i),&i = j, i \geq 0,\\
0,&\text{otherwise}, 
\end{cases}
\]
for strictly positive sequences \((\lambda_n)_{n \in \mathbb{N}}\) and \((\rho_n)_{n \in \mathbb{N}}\) and \(\mu_0= 0\) and \(\lambda_0 > 0\).
Set
\begin{align*}
r &\triangleq \sum_{n = 1}^\infty \left(\frac{1}{\lambda_n} + \frac{\mu_n}{\lambda_n \lambda_{n-1}} + \frac{\mu_n \mu_{n-1}}{\lambda_n \lambda_{n-1} \lambda_{n-2}} + \cdots + \frac{\mu_n \cdots \mu_2}{\lambda_n \cdots \lambda_2 \lambda_1}\right),\\
s &\triangleq \sum_{n = 1}^\infty \frac{1}{\mu_{n+1}} \left(1 + \frac{\lambda_n}{\mu_n} + \frac{\lambda_n \lambda_{n-1}}{\mu_n \mu_{n-1}} + \cdots + \frac{\lambda_n \lambda_{n-1} \cdots \lambda_2 \lambda_1}{\mu_n \mu_{n-1} \cdots \mu_2 \mu_1}\right).
\end{align*}
If \(r = s = \infty\) it is well-known that a Feller-Dynkin family \((P^d_i)_{i \in S_d}\) exists, see \cite[Theorems 3.2.2, 3.2.3]{anderson2012continuous} and Remark \ref{rem: new cond} (i) and (iii). In this case, if also one of the Conditions \ref{cond: suff1} and \ref{cond: suff2} holds, the family \((P_x)_{x \in S}\) is Feller-Dynkin whenever it is \(C_b\)-Feller. To be more concrete, if we choose 
\[
\lambda_n \triangleq n^\alpha \lambda,\quad \mu_n \triangleq n^\alpha \mu,\quad \alpha \geq 0, \lambda, \mu > 0, 
\]
then \(s = r =\infty\) if and only if either \(\alpha \leq 1\) or [\(\alpha \in (1, 2]\) and \(\lambda = \mu\)]. In other words, we find coefficients \(a, b\) and \(Q\) which satisfy the conditions from Corollary \ref{coro: sd} with an unbounded \(Q\). 
\end{example}

\subsection{Conditions \emph{not} to be Feller-Dynkin}
Next, we give conditions for rejecting the Feller-Dynkin property under the following standing assumption.
\begin{SA}
	\(|S_d| < \infty\). 
\end{SA}
Let \(\Sigma\) and \(\mathcal{L}\) be as in Section \ref{sec: FDP} and define
\[
D \triangleq \big\{f, fg, g \colon f \in C^2_b(\mathbb{R}^d), g \colon S_d \to \mathbb{R}\big\}.
\]
\begin{proposition}\label{prop: finite environment NFDP}
	Assume that there exist an \(r > 0\) and two locally H\"older continuous functions \(b_d \colon [r, \infty) \to \mathbb{R}\) and \(a_d \colon [r, \infty) \to (0, \infty)\) such that for all \(i \in S_d\) and \(x \in \mathbb{R}^d \colon \|x\| \geq 2r\)
	\begin{align*}
	\langle  x, a(x, i)x\rangle &\geq a_d \left(\tfrac{\|x\|^2}{2}\right),\\
	\textup{trace } a(x, i) + 2 \langle x, b(x, i)\rangle &\leq b_d\left(\tfrac{\|x\|^2}{2}\right) \langle x, a(x, i)x\rangle,
	\end{align*}
	and 	\begin{align*} 
	p (t) \triangleq \int_{r + 1}^t\exp \left(- \int_{r+1
	}^y b_d(z) \dd z\right) \dd y \to \infty \text{ as } t \to \infty,
	\end{align*}
	and \begin{align*} 
	\int_{r + 1}^\infty p'(y) \int_y^\infty \frac{\dd z}{a_d(z) p'(z)} \dd y < \infty.
	\end{align*}
Then \((P_x)_{x \in S}\) is not Feller-Dynkin. 
\end{proposition}
\begin{proof}
Applying the change of variable as explained in \cite[Section 4.1]{Azencott1974} together with \cite[Lemma 4.2]{Azencott1974}, we obtain that there exists a twice continuously differentiable decreasing solution \(u \colon [r, \infty) \to (0, \infty)\) to the differential equation
\begin{align*}
\tfrac{1}{2} a_d b_d u' + \tfrac{1}{2} a_d u'' = u, \quad u(r) = 1,
\end{align*}
which satisfies \(\lim_{x \nearrow + \infty} u(x) > 0\). 
We find a twice continuously differentiable function \(\phi \colon [0, \infty) \to (0, \infty)\) such that \(\phi \geq 1\) on \([0, r]\) and \(\phi = u\) on \((r, \infty)\).
It follows similarly to the proof of Proposition \ref{prop: sd} that \[U (x, i) \triangleq \phi\left(\tfrac{\|x\|^2}{2}\right),\quad (x, i) \in S,\]
has the properties from Theorem \ref{prop:cont} for the compact sets \(C \equiv K \triangleq \{x \in \mathbb{R}^d \colon \|x\| \leq \sqrt{2 r}\} \times S_d\), which implies the claim. 
\end{proof}
\subsection{Equivalent  Characterization for the State-Independent Case}\label{sec: sic}
In this section we study the state-independent case and characterize the Feller-Dynkin property via the Feller-Dynkin property of diffusions in fixed environments.
\subsubsection{The Setup}
We impose the following:
\begin{SA} We have \(S_d = \{1, \dots, N\}\) for \(1 \leq N \leq \infty\), \(Q(x) \equiv Q\) and there exists a continuous-time Markov chain with \(Q\)-matrix \(Q\). For us a Markov chain is always non-explosive. We denote its unique law by \((P^\star_i)_{i \in S_d}\), where the subscript indicates the starting value. Furthermore, \((P_i^\star)_{i \in S_d}\) is Feller-Dynkin.
\end{SA}
From now on we fix a root \(a^\frac{1}{2}\) of \(a\). Let \(\mathcal{L}\) and \(\Sigma\) be as in Section \ref{sec: FDP} and set 
\begin{align}\label{eq: new D}
D \triangleq \big\{f, fg, g \colon f \in C^2_b(\mathbb{R}^d), g \in C \big\}, \quad C \triangleq \big\{g \in C_0(S_d) \colon Q f \in C_0(S_d)\big\}.
\end{align}
Due to \cite[Theorem 5]{doi:10.1112/jlms/s2-5.2.267}, \((Q, C)\) is the generator of \((P^\star_i)_{i \in S_d}\) and, consequently, for each \(i \in S_d\) the probability measure \(P^\star_i\) is the unique solution to the MP \((C, Q, \Sigma_d, i)\).
It seems to be known that the family \((P_x)_{x \in S}\) has a one-to-one relation to a switching diffusion defined via an SDE, see, for instance, \cite{doi:10.1080/17442508108833180} for a partial result in this direction. However, we did not find a complete reference, such that we provide a statement and a proof.
\begin{lemma}\label{lem: representation}
	Fix \(y = (x, i) \in S\). A probability measure \(P_y\) solves the MP \((D, \mathcal{L}, \Sigma, y)\) if and only if there exists a filtered probability space with right-continuous complete filtration \((\mathcal{G}_t)_{t \geq 0}\) which supports a Markov chain \((Z_t)_{t \geq 0}\) for the filtration \((\mathcal{G}_t)_{t \geq 0}\) with \(Q\)-matrix \(Q\) and initial value \(Z_0 = i\) and a continuous, \((\mathcal{G}_t)_{t \geq 0}\)-adapted process \((Y_t)_{t \geq 0}\) satisfying the SDE
\begin{align}\label{eq: SDE Y}
\dd Y_t = b(Y_t, Z_t)\dd t + a^\frac{1}{2} (Y_t, Z_t)\dd W_t, \quad Y_0 = x, 
\end{align}
where \((W_t)_{t \geq 0}\) is a Brownian motion for the filtration \((\mathcal{G}_t)_{t \geq 0}\) such that the law of \((Y_t, Z_t)_{t \geq 0}\) is given by \(P_{y}\) and the \(\sigma\)-fields \(\sigma(W_t, t \in \mathbb{R}_+)\) and \(\sigma(Z_t, t \in \mathbb{R}_+)\) are independent.
\end{lemma}
\begin{proof}
	The implication \(\Leftarrow\) is a consequence of the integration by parts formula.

It remains to show the implication \(\Rightarrow\). 
We consider the completion of the filtered probability space \((\Omega, \mathcal{F}, (\mathcal{F}_t)_{t \geq 0}, P_y)\) as underlying filtered probability space.
Denote \((X_t)_{t \geq 0} = (Y_t, Z_t)_{t \geq 0}\), where \((Y_t)_{t \geq 0}\) is \(\mathbb{R}^d\)-valued and \((Z_t)_{t \geq 0}\) is \(S_d\)-valued. In view of \cite[Remark 5.4.12]{KaraShre}, we can argue as in the proof of \cite[Proposition 5.4.6]{KaraShre} to conclude the existence of a Brownian motion \((W_t)_{t \geq 0}\) (possibly defined on a standard extension of the filtered probability space \((\Omega, \mathcal{F}, (\mathcal{F}_t)_{t \geq 0}, P_y)\), see \cite[Remark 3.4.1]{KaraShre}) such that \((Y_t)_{t \geq 0}\) satisfies the SDE \eqref{eq: SDE Y}. 
With abuse of notation, we denote the standard extension of \((\Omega, \mathcal{F}, (\mathcal{F}_t)_{t \geq 0}, P_y)\) again by \((\Omega, \mathcal{F}, (\mathcal{F}_t)_{t \geq 0}, P_y)\).
Due to \cite[Proposition 10.46]{J79} the martingale property is not affected by a standard extension.
Thus, we deduce from Examples \ref{ex: Q ma} and \ref{ex: MC}, Proposition \ref{prop: existence initial law} in Appendix \ref{app: prop exietence initial law} and \cite[Theorem 4.4.2]{EK} that \((Z_t)_{t \geq 0}\) is a Markov chain for the filtration \((\mathcal{F}_t)_{t \geq 0}\) with \(Q\)-matrix \(Q\) and \(Z_0 = i\).
It remains to explain that the \(\sigma\)-fields \(\sigma(W_t, t \in \mathbb{R}_+)\) and \(\sigma(Z_t, t \in \mathbb{R}_+)\) are independent.
We adapt an idea from \cite[Theorem 4.10.1]{EK}. For all \(f \in C\) the process
\begin{align*}
M^f_t &\triangleq f(Z_t) - f(i) - \int_0^t Q f(Z_s)\dd s, \quad t \in \mathbb{R}_+,
\end{align*}
is a \(P_y\)-martingale. For \(g \in C^2_c(\mathbb{R}^d)\) with \(\inf_{x \in \mathbb{R}^d} g(x) > 0\) set 
\begin{align*}
K^g_t &\triangleq g(W_t) \exp \bigg( - \frac{1}{2} \int_0^t \frac{\Delta g (W_s)}{g(W_s)} \dd s \bigg), \quad t \in \mathbb{R}_+,
\end{align*}
where \(\Delta\) denotes the Laplacian.
It\^o's formula yields that 
\[
\dd K^g_t =  \exp \bigg( - \frac{1}{2} \int_0^t \frac{\Delta g (W_s)}{g(W_s)} \dd s \bigg)\langle \nabla g(W_t), \dd W_t\rangle, 
\]
which implies that also \((K^g_t)_{t \geq 0}\) a \(P_y\)-martingale, because it is a bounded (on finite time intervals) local \(P_y\)-martingale.
Because \((Z_t)_{t \geq 0}\) has only finitely many jumps in a finite interval, \((M^f_t)_{t \geq 0}\) is of finite variation on finite intervals and we have \(P_y\)-a.s.
\[
[M^f, K^g]_t = 0 \text{ for all }t \in \mathbb{R}_+, 
\]
see \cite[Proposition I.4.49]{JS}. 
Here, \([\cdot, \cdot]\) denotes the quadratic variation process.
Consequently, integration by parts yields that \((M^f_t K^g_t)_{t \geq 0}\) is a local \(P_y\)-martingale and a true \(P_y\)-martingale due to its boundedness on finite time intervals. Fix an arbitrary bounded stopping time \(\psi\) and define
\[
Q (G) \triangleq \frac{E_y \big[ \1_G K^g_\psi \big]}{g(0)}, \quad G \in \mathcal{F}.
\]
Due to the optional stopping theorem, for all bounded stopping times \(\phi\) we have 
\[
E^Q \big[M^f_\phi\big] = \frac{E_y \big[ M^f_{\phi \wedge \psi} K^g_{\phi \wedge \psi}\big]}{g(0)} = 0.
\]
We conclude from \cite[Proposition II.1.4]{RY} that \((M^f_t)_{t \geq 0}\) is a \(Q\)-martingale. Consequently, in view of Example \ref{ex: Q ma}, we have 
\[
Q(\Gamma ) = P_y(\Gamma ), 
\]
where 
\[
\Gamma \triangleq \big\{Z_{t_1} \in F_1, \dots, Z_{t_n} \in F_n \big\}
\]
for arbitrary \(0 \leq t_1 < \dots < t_n < \infty\) and \(F_1, \dots, F_n \in \mathcal{B}(S_d)\). Suppose that \(P_y(\Gamma) > 0\) and set 
\[
\widehat{Q} (G) \triangleq \frac{E_y \big[ \1_G \1_\Gamma\big]}{P_y(\Gamma)},\quad G \in \mathcal{F}.
\]
We have 
\[
E^{\widehat{Q}} \big[ K^g_\psi\big] = \frac{E_y\big[K^g_{\psi} \1_\Gamma\big]}{P_y(\Gamma)} = \frac{Q(\Gamma) g(0)}{P_y(\Gamma)} = g(0).
\]
Thus, because \(\psi\) was arbitrary, we deduce from \cite[Proposition II.1.4]{RY} and \cite[Proposition 4.3.3]{EK} that \((W_t)_{t \geq 0}\) is a \(\widehat{Q}\)-Brownian motion and the uniqueness of the Wiener measure yields that 
\[
\widehat{Q}\big(W_{s_1} \in G_1, \dots, W_{s_k} \in G_k\big) = P_y\big(W_{s_1} \in G_1, \dots, W_{s_k} \in G_k\big)
\]
for arbitrary \(0 \leq s_1 < \dots < s_k < \infty\) and \(G_1, \dots, G_k \in \mathcal{B}(\mathbb{R}^d)\).
Using the definition of \(\widehat{Q}\), we conclude that 
\begin{align*}
P_y \big(Z_{t_1} \in F_1&, \dots, Z_{t_n} \in F_s, W_{s_1} \in G_1, \dots, W_{s_k} \in G_k \big) 
\\&= P_y \big(Z_{t_1} \in F_1, \dots, Z_{t_n} \in F_s\big) P_y \big(W_{s_1} \in G_1, \dots, W_{s_k} \in G_k \big),
\end{align*}
which implies the desired independence.
\end{proof}
\begin{remark}\label{rem: after lem 2}
An inspection of the proof of Lemma \ref{lem: representation} shows the following: 
\begin{enumerate}
\item[(i)] If \((Z_t)_{t \geq 0}\) is a Feller-Dynkin Markov chain and \((W_t)_{t \geq 0}\) is a Brownian motion both with deterministic initial values and for the same filtration, then the \(\sigma\)-fields \(\sigma(W_t, t \in \mathbb{R}_+)\) and \(\sigma (Z_t, t \in  \mathbb{R}_+)\) are independent. 
\item[\textup{(ii)}] 
As explained in Example \ref{ex: MC}, we find a countable set \(C^\star \subseteq C\) such that for all \(f \in C\) there exists a sequence \((f_n)_{n \in \mathbb{N}} \subset C^\star\) such that 
\[
\|f - f_n\|_\infty + \|Qf - Q f_n\|_\infty \to 0
\]
as \(n \to \infty\).
The set of solutions to the MP \((D, \mathcal{L}, \Sigma, y)\) remains unchanged if we redefine \(D\) to be the countable set 
\begin{align}\label{eq: D}
\big \{ f,  g^k_{ij}, g^k_i \colon 1 \leq i, j \leq d, k \in \mathbb{N}, f \in C^\star\big\},
\end{align}
where \(g_i^k, g^k_{ij}\) are functions in \(C^2_c(\mathbb{R}^d)\) such that \(g^k_i (x) = x_i\) and \(g^k_{ij} (x) = x_ix_j\) for all \(x \in \mathbb{R}^d \colon \|x\| \leq k\).
\end{enumerate}
\end{remark}

We set 
\[
\Sigma_c \triangleq \big\{ \omega\colon \mathbb{R}_+ \to \mathbb{R}^d \colon t \mapsto \omega(t) \text{ is continuous}\big\},
\]
and
	\begin{align}\label{eq: Ki}
\mathcal{K}^i f (x) \triangleq \langle  \nabla f (x), b(x, i)\rangle + \tfrac{1}{2} \textup{trace }(\nabla^2 f(x) a(x, i))
\end{align}
for \(f \in C_b^2(\mathbb{R}^d)\) and \((x, i) \in S\). 
We equip \(\Sigma_c\) with the local uniform topology. In this case the Borel \(\sigma\)-field is generated by the coordinate process on \(\Sigma_c\), see \cite[p. 30]{SV}.
A map \(F \colon \mathbb{R}^d \times \Sigma_c \to \Sigma_c\) is called universally adapted, if it is adapted to the filtration \((\bigcap_{\mu \in \mathcal{P}} \mathcal{G}^\mu_t)_{t \geq 0}\), where \(\mathcal{P}\) is the set of all Borel probability measures on \(\mathbb{R}^d\) and \((\mathcal{G}^\mu_t)_{t \geq 0}\) is the completion of the canonical filtration on \(\mathbb{R}^d \times \Sigma_c\) w.r.t. the product measure \(\mu \otimes \mathscr{W}\), where \(\mathscr{W}\) is the Wiener measure, see \cite[p. 346]{Kallenberg}.
\begin{definition}\label{def: stong existence}
	A family \((P^i_x)_{x \in \mathbb{R}^d}\) of solutions to the MP \((C^2_b(\mathbb{R}^d), \mathcal{K}^i, \Sigma_c)\) is said to exist strongly, if 
	a universally adapted Borel map \(F^i \colon \mathbb{R}^d \times \Sigma_c \to \Sigma_c\) exists such that
	on every filtered probability space with right-continuous complete filtration \((\mathcal{G}_t)_{t \geq 0}\), which supports a Brownian motion \(W = (W_t)_{t \geq 0}\) and an \(\mathbb{R}^d\)-valued \(\mathcal{G}_0\)-measurable random variable \(\pi\), the process \(F^i(\pi, W)\) solves the SDE
\begin{align}\label{eq: SDE some}
\dd Y^i_t = b(Y^i_t, i) \dd t + a^\frac{1}{2} (Y^i_t, i) \dd W_t, \quad Y^i_0 = \pi,
\end{align}
and every solution \((Y^i_t)_{t \geq 0}\) to \eqref{eq: SDE some} satisfies \((Y^i_t)_{t \geq 0} = F^i(\pi, W)\) up to a null set.
Here, the state space for the MP is \(\mathbb{R}^d\).
\end{definition}
\begin{remark}\label{rem:SE}
We stress that our definition of strong existence includes a version of pathwise uniqueness and that the function \(F^i\) in the previous definition is independent of the law of \(\pi\). 
A generalization of the classical Yamada-Watanabe theorem yields that \((P^i_x)_{x \in \mathbb{R}^d}\) exists strongly if and only if the SDE \eqref{eq: SDE some} satisfies weak existence and pathwise uniqueness for all degenerated initial values, see \cite[Theorem 18.14]{Kallenberg}. In the classical formulation of the Yamada-Watanabe theorem as given, for instance, in \cite{KaraShre} the function \(F^i\) depends on the law of \(\pi\). This dependence was removed in \cite{10.2307/2244838}.
\end{remark}
\subsubsection{Main Results}
Next, we state the main results for this section. The proofs can be found in the following subsections.
\begin{condition}\label{cond: abs}
	We have \(q_{ii} \not = 0\) for all \(i \in S_d\).
\end{condition}
\begin{condition}\label{cond: sec eq}  The family \((P_y)_{y \in S}\) is unique and \(C_b\)-Feller, and for all \((x, i) \in S\) the MP \(( C^2_b(\mathbb{R}^d),\mathcal{K}^i, \Sigma_c, x)\), where \(\mathcal{K}^i\) is given as in \eqref{eq: Ki}, has a unique solution \(P^i_x\). Furthermore, for all \(i \in S_d\) the family \((P^i_x)_{x \in \mathbb{R}^d}\) is \(C_b\)-Feller and exists strongly.
\end{condition}

The following observation is the main result of this section. 

\begin{theorem}\label{prop: equiv}
	Suppose that the Conditions \ref{cond: abs} and \ref{cond: sec eq} hold.
	The following are equivalent:
	\begin{enumerate}
		\item[\textup{(i)}] The family \((P_y)_{y \in S}\) is Feller-Dynkin. 
		\item[\textup{(ii)}] For all \(i \in S_d\) the family \((P^i_x)_{x \in \mathbb{R}^d}\) is Feller-Dynkin.
	\end{enumerate}
\end{theorem}
For the strong Feller property a related result is known, see \cite[Theorem 3.2]{doi:10.1137/15M1013584}.
One implication in the previous theorem can be generalized as the following proposition shows.
\begin{proposition}\label{lem: non-existence}
	Suppose that there exists an \(i \in S_d\) such that for all \(x \in \mathbb{R}^d\) the MP \((C^2_b(\mathbb{R}^d),\mathcal{K}^i,  \Sigma_c, x)\)
	has a (unique) solution \(P^i_x\) and that the family \((P^i_x)_{x \in \mathbb{R}^d}\) exists strongly and is \(C_b\)-Feller, but not Feller-Dynkin. Then, \((P_x)_{x \in S}\) is not Feller-Dynkin.
\end{proposition}

The next two results provide conditions implying Condition \ref{cond: sec eq}. 
\begin{proposition}\label{prop: sd feller}
	Suppose that \(b\) and \(a\) are continuous and that \((P_y)_{y \in S}\) is unique, then \((P_y)_{y \in S}\) is strongly Markov and \(C_b\)-Feller. 
\end{proposition}
\begin{proposition}\label{prop: uni}
	Suppose that Condition \ref{cond: abs} holds and that for all \(i \in S_d\) the family \((P^i_x)_{x \in \mathbb{R}^d}\) exists strongly, then a unique family \((P_y)_{y \in S}\) exists.
\end{proposition}
An existence result without uniqueness is given in Appendix \ref{appendix}. We collect some consequences of the preceding results.
\begin{corollary}\label{coro: complete}
	Suppose that \(d = 1\), that Condition \ref{cond: abs} holds and that for all \(i \in S_d\) the map \(x \mapsto b(x, i)\) is continuous and the map \(x \mapsto a^\frac{1}{2} (x, i)\) is locally H\"older continuous with exponent larger or equal than \(\tfrac{1}{2}\) and that \(a^\frac{1}{2}(\cdot, i) \not = 0\). Furthermore, for all \(i \in S_d\) suppose that
	\begin{align}\label{eq: FC}
	\lim_{x \to \pm \infty} \int_0^x \exp \bigg( - 2 \int_0^y \frac{b(z, i)}{a(z, i)} \dd z \bigg) \int_0^y \frac{2 \exp \big( 2 \int_0^u \frac{b(z, i)}{a(z, i)} \dd z \big)}{a (u, i)} \dd u \dd y = \infty.
	\end{align}
	Then, the family \((P_x)_{x \in S}\) exists uniquely, is strongly Markov and \(C_b\)-Feller. Moreover, the following are equivalent:
	\begin{enumerate}
		\item[\textup{(i)}] \((P_x)_{x \in S}\) is Feller-Dynkin.
		\item[\textup{(ii)}] For all \(i \in S_d\) one of the conditions \eqref{eq: a1} and \eqref{eq: a2} holds and one of the conditions \eqref{eq: b1} and \eqref{eq: b2} holds:
		\begin{align}\label{eq: a1}\int_0^\infty \exp \bigg( - 2 \int_0^y \frac{b(z, i)}{a(z, i)} \dd z \bigg) \dd y < \infty.\end{align}
		\begin{equation}\label{eq: a2}
		\left\{\quad \begin{split}
		\int_0^\infty\exp \bigg( - 2 \int_0^y \frac{b(z, i)}{a(z, i)} \dd z \bigg) \dd y &= \infty,\\ \int_0^\infty \exp \bigg( - 2 \int_0^y \frac{b(z, i)}{a(z, i)} \dd z \bigg) \int_y^\infty \frac{\exp \big( 2 \int_0^u \frac{b(z, i)}{a(z, i)} \dd z \big)}{a(u, i)}\dd u \dd y &= \infty.
		\end{split}\right.
		\end{equation}

		\begin{align}\label{eq: b1}\int_{-\infty}^0 \exp \bigg( 2 \int_y^0 \frac{b(z, i)}{a(z, i)} \dd z \bigg) \dd y < \infty.\end{align}
		\begin{equation}\label{eq: b2}\left\{\hspace{0.2cm}
		\begin{split}
		\int_{-\infty}^0 \exp \bigg( 2 \int_y^0 \frac{b(z, i)}{a(z, i)} \dd z \bigg) \dd y = \infty,\\ \int_{-\infty}^0 \exp \bigg( 2 \int_y^0 \frac{b(z, i)}{a(z, i)} \dd z \bigg) \int_{-\infty}^y \frac{\exp \big( - 2 \int_u^0 \frac{b(z, i)}{a(z, i)} \dd z \big)}{a(u, i)}\dd u \dd y = \infty.
		\end{split}\right.
		\end{equation}
	\end{enumerate}
\end{corollary}
\begin{remark}\label{rem: b vanish}
If \(b \equiv 0\), then the conditions in part (ii) of Corollary \ref{coro: complete} are satisfies if and only if for all \(i \in S_d\) the following hold:
\begin{align}\label{eq: 1D without drift}
\int_0^\infty \frac{u}{a(u, i)}\dd u = \int_{- \infty}^0 \frac{- u}{a(u, i)} \dd u = \infty.
\end{align}
\end{remark}
\begin{corollary}\label{coro: conv}
	Assume that Condition \ref{cond: abs} holds and that for all \(i \in S_d\) the maps \(x \mapsto b(x, i)\) and \(x \mapsto a^\frac{1}{2} (x, i)\) are locally Lipschitz continuous and that for all \((x, i) \in S\) the MP \((C^2_b(\mathbb{R}^d),\mathcal{K}^i, \Sigma_c, x)\) has a solution. Furthermore, suppose that for each \(i \in S_d\) there is an \(r_i > 0\) and two locally H\"older continuous functions \(b_i \colon [r_i, \infty) \to \mathbb{R}\) and \(a_i \colon [r_i, \infty) \to (0, \infty)\) such that for all \(x \in \mathbb{R}^d \colon \|x\| \geq 2r_i\)
	\begin{align*}
	\langle  x, a(x, i)x\rangle &\leq a_i \left(\tfrac{\|x\|^2}{2}\right),\\
	\textup{trace } a(x, i) + 2 \langle x, b(x, i)\rangle &\geq b_i\left(\tfrac{\|x\|^2}{2}\right) \langle x, a(x, i)x\rangle,
	\end{align*}
	and either
	\begin{align*} 
	p_i(r) \triangleq \int_1^{r_i} \exp \left(- \int_1^y b_i(z) \dd z\right) \dd y, \quad \lim_{r \to \infty} p_i(r) < \infty,
	\end{align*}
	or \begin{align*}  
	\lim_{r \to \infty} p(r) = \infty \text{ and } \int_1^\infty p_i'(y) \int_y^\infty \frac{\dd z}{a_i(z) p'_i(z)} \dd y = \infty.
	\end{align*}
	Then, \((P_x)_{x \in S}\) is Feller-Dynkin. 
\end{corollary}
Explicit conditions for the assumption that for all \((x, i) \in S\) the MP \((C^2_b(\mathbb{R}^d),\mathcal{K}^i, \Sigma_c, x)\) has a solution can, e.g., be found in \cite[Chapter 10]{SV}.
\begin{corollary}\label{prop: nf}
	Assume that there exists an \(i \in S_d\) such that the maps \(x \mapsto b(x, i)\) and \(x \mapsto a^\frac{1}{2} (x, i)\) are locally Lipschitz continuous and that for all \(x \in \mathbb{R}^d\) the MP \((C^2_b(\mathbb{R}^d),\mathcal{K}^i,  \Sigma_c, x)\) has a solution. Furthermore, suppose there is an \(r > 0\) and two locally H\"older continuous functions \(b_d \colon [r, \infty) \to \mathbb{R}\) and \(a_d \colon [r, \infty) \to (0, \infty)\) such that for all \(x \in \mathbb{R}^d \colon \|x\| \geq 2r\)
	\begin{align*}
	\langle  x, a(x, i)x\rangle &\geq a_d \left(\tfrac{\|x\|^2}{2}\right),\\
	\textup{trace } a(x, i) + 2 \langle x, b(x, i)\rangle &\leq b_d\left(\tfrac{\|x\|^2}{2}\right) \langle x, a(x, i)x\rangle,
	\end{align*}
	and 	\begin{align*} 
	p (t) \triangleq \int_{r + 1}^t\exp \left(- \int_{r+1
	}^y b_d(z) \dd z\right) \dd y \to \infty \text{ as } t \to \infty,
	\end{align*}
	and \begin{align*} 
	\int_{r + 1}^\infty p'(y) \int_y^\infty \frac{\dd z}{a_d(z) p'(z)} \dd y < \infty.
	\end{align*}
	Then, \((P_x)_{x \in S}\) is not Feller-Dynkin.
\end{corollary}

By \cite[Theorem 3.2]{doi:10.1137/15M1013584}, the family \((P_x)_{x \in S}\) has the strong Feller property if it is \(C_b\)-Feller and for all \(i \in S_d\) the families \((P^i_x)_{x \in \mathbb{R}^d}\) have the strong Feller property. Consequently, the strong Feller property and the Feller-Dynkin property are both inherited from the relative properties of processes in the fixed environments. We give a short example for a switching diffusion which has the strong Feller property, but not the Feller-Dynkin property.
\begin{example}	\label{ex: FD not SF}
	Let \(d = 1, S_d = \{1, 2\}, b\equiv 0\) and 
	\[
	a(x, i) \triangleq \begin{cases}
	1 + x^4,& i = 1,\\
	1,& i = 2,
	\end{cases}
	\]
	for \( (x, i) \in S.\)
	Due to \cite[Problem 5.5.27]{KaraShre}, \eqref{eq: FC} holds in the case \(b \equiv 0\).
	Thus, we conclude from Corollary \ref{coro: complete} that \((P_y)_{y \in S}\) exists uniquely and is \(C_b\)-Feller. Furthermore, due to \cite[Corollary 10.1.4]{SV}, \((P^i_x)_{x \in \mathbb{R}}\) has the strong Feller property for \(i = 1, 2\). Of course, the family \((P_x^2)_{x \in \mathbb{R}}\) consists of Wiener measures and is well-known to be strongly Feller. Therefore, \cite[Theorem 3.2]{doi:10.1137/15M1013584} implies that \((P_y)_{y \in S}\) has the strong Feller property, too.
	However, for \(i = 1\) the condition \eqref{eq: 1D without drift} fails because
	\[
	\int_0^\infty \frac{x\ \dd x}{1 + x^4} = \frac{\pi}{4}< \infty.
	\]
	Therefore, the family \((P_x)_{x \in S}\) is not Feller-Dynkin due to Corollary \ref{coro: complete}, see Remark \ref{rem: b vanish}.
\end{example}

\subsubsection{Proof of Proposition \ref{lem: non-existence}}
Since \((P^i_x)_{x \in \mathbb{R}^d}\) is \(C_b\)-Feller, one can show as in the proof of Theorem \ref{theo: Ly} that if for any compact set \(K \subset \mathbb{R}^d\) and any \(t > 0\) it holds that
\[
\limsup_{\|x\| \to \infty}P^i_x (X_t \in K) = 0, 
\]
then \((P^i_x)_{x \in \mathbb{R}^d}\) is Feller-Dynkin. Consequently, since we assume \((P^i_x)_{x \in \mathbb{R}^d}\) not to be Feller-Dynkin, there exists a sequence \((x_k)_{k \in \mathbb{N}} \subset \mathbb{R}^d\) with \(\|x_k\| \to \infty\) as \(k \to \infty\), a compact set \(K^o \subset \mathbb{R}^d\) and a \(t^o > 0\) such that 
\begin{align}\label{eq: mp not fd use}
\limsup_{k \to \infty} P^i_{x_k}(X_{t^o} \in K^o)  > 0.
\end{align}
The set \(G \triangleq K^o \times \{i\} \subset S\) is compact.
If we show that 
\begin{align}\label{eq: ts nf}
\limsup_{k \to \infty} P_{(x_k, i)} (X_{t^o} \in G) > 0, 
\end{align} 
then \((P_x)_{x \in S}\) cannot be Feller-Dynkin. To see this, assume for contradiction that \((P_x)_{x \in S}\) is Feller-Dynkin. Due to the locally compact version of Urysohn's lemma, there exists a function \(f \in C_0(S)\) such that \(0 \leq f \leq 1\) and \(f \equiv 1\) on \(G\). Consequently, we have 
\begin{align*}
P_{(x_k, i)} (X_{t^o} \in G) &= E_{(x_k, i)} \big[ f(X_{t^o}) \1 \{X_{t^o} \in G\}\big] \\&\leq E_{(x_k, i)} \big[ f(X_{t^o})\big] \to 0 \text{ as } k \to \infty, 
\end{align*}
because \((P_x)_{x \in S}\) is Feller-Dynkin. This, however, is a contradiction and we conclude that \((P_x)_{x \in S}\) cannot be Feller-Dynkin.
In summary, it suffices to show \eqref{eq: ts nf}.

	For a \cadlag \(S_d\)-valued process \((Z_t)_{t \geq 0}\), we set 
	\[
	\tau (Z) \triangleq \inf\big(t \in \mathbb{R}_+ \colon Z_t \not = Z_0\big), 
	\]
	which is a stopping time for any right-continuous filtration to which \((Z_t)_{t \geq 0}\) is adapted, see \cite[Proposition 2.1.5]{EK}.
	In the following let \((Y_t)_{t \geq 0}, (Z_t)_{t \geq 0}\) and \((W_t)_{t \geq 0}\) be as in Lemma \ref{lem: representation} for \(y = (x, i)\).
	On \(\{t \leq \tau (Z)\}\) we have 
	\[
	Y_t = x + \int_0^t b(Y_s, i) \dd s + \int_0^t a^\frac{1}{2} (Y_s, i) \dd W_s, 
	\]
	which is the SDE corresponding to the MP \((C^2_b(\mathbb{R}^d),\mathcal{K}^i,  \Sigma_c, x)\), see \cite[Corollary 5.4.8]{KaraShre}. 
	We now need a local version of pathwise uniqueness. The proof of the following lemma is given after the proof of Proposition \ref{lem: non-existence} is complete.
	\begin{lemma}\label{lem: loc pathwise uniqueness}
	Suppose that the SDE 
	\begin{align}\label{eq: SDE--}
	\dd Y_t = \mu(Y_t) \dd t + \sigma(Y_t)\dd W_t
	\end{align}
	satisfies weak existence and pathwise uniqueness (see \cite[Section IX.1]{RY}). In other words, we assume that the martingale problem corresponding to the SDE \eqref{eq: SDE--} exists strongly, see Remark \ref{rem:SE} and \cite[Section 5.4]{KaraShre}. Consider a filtered probability space with right-continuous complete filtration \((\mathcal{G}_t)_{t \geq 0}\), which supports a Brownian motion \((W_t)_{t \geq 0}\) and an \(\mathbb{R}^d\)-valued \(\mathcal{G}_0\)-measurable random variable \(\psi\). Take a \((\mathcal{G}_t)_{t \geq 0}\)-stopping time \(\tau\) and let \((Y_t)_{t \geq 0}\) be the solution to \eqref{eq: SDE--} with initial value \(\psi\). Then, all solutions to 
	\[
	\dd O_t = \mu(O_t)\1_{\{t \leq \tau\}} \dd t + \sigma (O_t)\1_{\{t \leq \tau\}} \dd W_t, \quad O_0 = \psi, 
	\]
	are indistinguishable from \((Y_{t \wedge \tau})_{t \geq 0}\).
	\end{lemma}
	
	Because we assume that \(P^i_x\) exists strongly, Lemma \ref{lem: loc pathwise uniqueness} and the independence of the \(\sigma\)-fields \(\sigma(W_t, t \in \mathbb{R}_+)\) and \(\sigma(Z_t, t \in \mathbb{R}_+)\), see Lemma \ref{lem: representation}, imply that 
	\begin{align*}
	P_{(x, i)} \big(X_{t^0} \in G\big) 
	&\geq P \big(Y_{t^o \wedge \tau (Z)} \in K^o, Z_{t^o} = i, t^o < \tau(Z)\big)
	\\&= P \big(F^i(x, W)_{t^o \wedge \tau(Z)} \in K^o, t^o < \tau(Z)\big)
	\\&= P \big(F^i (x, W)_{t^o} \in K^o\big) P \big(t^o < \tau(Z)\big)
	\\&= P^i_x \big(X_{t^o} \in K^o\big) P \big(t^o < \tau(Z)\big),
	\end{align*}
	where \(F^i\) is as in Definition \ref{def: stong existence}. 
		It is well-known that \(\tau(Z)\) is exponentially distributed with parameter \(- q_{ii}\), see, e.g., \cite[Lemma 10.18]{Kallenberg}.
Therefore, we have 
\begin{align*}
P_{(x, i)} \big(X_{t^o} \in G\big) &\geq 
P^i_x \big(X_{t^o} \in K^o\big) e^{q_{ii} t^o}.
\end{align*}
We conclude \eqref{eq: ts nf} from \eqref{eq: mp not fd use}. This finishes the proof.
\qed
\\

\noindent
\textit{Proof of Lemma \ref{lem: loc pathwise uniqueness}:}
Due to localization, we can assume that \(\tau\) is finite.
Let \((B_t)_{t \geq 0}\) be defined by 
\[
B_t \triangleq W_{t + \tau} - W_\tau,\quad t \in \mathbb{R}_+.
\] 
Due to \cite[Proposition V.1.5]{RY} and L\'evy's characterization (see, e.g., \cite[Theorem 3.3.16]{KaraShre}), the process \((B_t)_{t \geq 0}\) is a \((\mathcal{G}_{t + \tau})_{t \geq 0}\)-Brownian motion
and, due to the strong existence hypothesis, there exists a solution \((U_t)_{t \geq 0}\) to the SDE
\[
\dd U_t = \mu(U_t) \dd t + \sigma (U_t)\dd B_t, \quad U_0 = O_\tau.
\] 
Now, we set 
\[
V_t \triangleq \begin{cases} O_t,&t \leq \tau,\\
U_{t - \tau},&t>\tau.
\end{cases}
\]
Because \(U_0 = O_\tau\), the process \((V_t)_{t \geq 0}\) has continuous paths. We claim that \((U_{t - \tau}\1_{\{\tau < t\}})_{t \geq 0}\) is progressive. This implies that \((V_t)_{t \geq 0}\) is adapted. Note that \(t \mapsto U_{t - \tau} \1_{\{\tau < t\}}\) is left-continuous and that \(s \mapsto  U_{t - s} \1_{\{s < t\}}\) is right-continuous. Thus, by an approximation argument, it suffices to show that \((h_t)_{t \geq 0} \triangleq (U_{t - \rho} \1_{\{\rho <t\}})_{t \geq 0}\) is adapted for any stopping time \(\rho\) which takes values in the countable set \(2^{-n} \overline{\mathbb{N}}\) for some \(n \in \mathbb{N}\) and satisfies \(\rho \geq \tau\).
Let \(G \in \mathcal{B}(\mathbb{R}^d)\) and set \(N_{t} \triangleq 2^{-n} \overline{\mathbb{N}} \cap [0, t)\). We have 
\[
\{h_t \in G\} = \bigg(\bigcup_{k \in N_{t}}  \big( \{U_{t - k} \in G\} \cap \{\rho = k\}\big)\bigg) \cup \big(\{0 \in G\} \cap \{\rho \geq t\}\big)  \in  \mathcal{G}_t.
\]
Here, we use that \(\{U_{t - k} \in G\} \in \mathcal{G}_{t - k + \tau} \subseteq \mathcal{G}_{t - k + \rho}\) and the fact that \(\mathcal{G}_{t - k + \rho} \cap \{\rho = k\} \subseteq \mathcal{G}_t\). Therefore, \((U_{t - \tau}\1_{\{\tau <t\}})_{t \geq 0}\) is progressive.
On \(\{t \leq \tau\}\) we have 
\[
V_t = \psi + \int_0^t \mu(V_s)\dd s + \int_0^t \sigma (V_s)\dd W_s.
\]
Classical rules for time-changed stochastic integrals (see, e.g., \cite[Propositions V.1.4, V.1.5]{RY}) yield that on \(\{t > \tau\}\)
\begin{align}
V_t &= O_\tau + \int_0^{t - \tau} \mu(U_s)\dd s + \int_0^{t - \tau} \sigma (U_s)\dd B_s\label{eq:1}
\\&= V_\tau + \int_\tau^t \mu(U_{s - \tau}) \dd s + \int_{\tau}^t \sigma (U_{s - \tau}) \dd W_s\label{eq:2}
\\&= V_\tau + \int_\tau^t \mu(V_s) \dd s + \int_{\tau}^t \sigma (V_s) \dd W_s\nonumber
\\&= \psi + \int_0^t \mu(V_s)\dd s + \int_0^t \sigma (V_s)\dd W_s.\nonumber
\end{align}
Consequently, \((V_t)_{t \geq 0}\) solves the SDE 
\[
\dd V_t = \mu(V_t)\dd t + \sigma (V_t)\dd W_t, \quad V_0 = \psi.
\]
By the strong existence hypothesis, we conclude that a.s. \(V_t = Y_t\) for all \(t \in \mathbb{R}_+\). The definition of \((V_t)_{t \geq 0}\) implies the claim.
\qed
\subsubsection{Proof of Theorem \ref{prop: equiv}}
	The implication (i) \(\Rightarrow\) (ii) follows from Proposition \ref{lem: non-existence}.
	
	We prove the implication (ii) \(\Rightarrow\) (i) using an explicit construction of the family \((P_y)_{y \in S}\). 
	Take a filtered probability space \((\Theta, \mathcal{G}, (\mathcal{G}_t)_{t \geq 0}, P)\) satisfying the usual hypothesis of a right-continuous and complete filtration, which supports a Brownian motion \((W_t)_{t \geq 0}\) for the filtration \((\mathcal{G}_t)_{t \geq 0}\) and an \(S_d\)-valued continuous-time Markov chain \((Z_t)_{t \geq 0}\) for the filtration \((\mathcal{G}_t)_{t \geq 0}\) with \(Q\)-matrix \(Q\) and \(Z_0 = i\). Recalling Remark \ref{rem: after lem 2}, we note that the \(\sigma\)-fields \(\sigma (W_t, t \in \mathbb{R}_+)\) and \(\sigma (Z_t, t \in \mathbb{R}_+)\) are independent. 
	Define inductively 
	\begin{align}\label{eq: jump times}
	\tau_0 \triangleq 0, \quad \tau_n \triangleq \inf \big( t \geq \tau_{n - 1} \colon Z_t \not = Z_{\tau_{n - 1}}\big), \quad n \geq 1,
	\end{align}
	and
	\[
	\sigma_0 \triangleq 0,\quad \sigma_n  \triangleq \tau_n - \tau_{n - 1} = \inf \big( t \in \mathbb{R}_+ \colon Z_{t + \tau_{n -1}} \not = Z_{\tau_{n -1}} \big), \quad n \geq 1.
	\]
	Because no state of \((Z_t)_{t \geq 0}\) is absorbing due to Condition \ref{cond: abs}, we have a.s. \(\tau_n < \infty\) for all \(n \in \mathbb{N}\). 
	Furthermore, for all \(n \in \mathbb{N}\) the random time \(\tau_n\) is a \((\mathcal{G}_t)_{t \geq 0}\)-stopping time and  the random time \(\sigma_n\) is a \((\mathcal{G}_{t + \tau_{n - 1}})_{t \geq 0}\)-stopping time, see \cite[Proposition 1.1.12]{KaraShre} and \cite[Lemma 6.5, Theorem 6.7]{Kallenberg}.
	Due to \cite[Proposition V.1.5]{RY} and L\'evy's characterization, the process \((W^n_t)_{t \geq 0} = (W_{t + \tau_n} - W_{\tau_n})_{t \geq 0}\) is a \((\mathcal{G}_{t + \tau_n})_{t \geq 0}\)-Brownian motion and therefore independent of \(\mathcal{G}_{\tau_n}\).
	For all \(k \in S_d\) let \(F^k \colon \mathbb{R}^d \times \Sigma_c \to \Sigma_c\) be as in Definition \ref{def: stong existence}
and set \((Y^{0, x}_t)_{t \geq 0} \triangleq F^i (x, W)\). By induction, define further
\[
(Y^{n, x}_t)_{t \geq 0} \triangleq \sum_{k = 1}^N F^k(Y^{n-1,x}_{\sigma_n}, W^n) \1 \{Z_{\tau_n} = k\},\quad n \in \mathbb{N}, 
\]
and set 
	\[
	Y^x_t \triangleq x \1 \{t = 0\} + \sum_{n = 0}^\infty Y^{n, x}_{t - \tau_n} \1\{\tau_n < t \leq \tau_{n+1}\}, \quad t \in \mathbb{R}_+.
	\]
	The process \((Y^x_t)_{t \geq 0}\) has continuous paths and similar arguments as used in the proof of Lemma \ref{lem: loc pathwise uniqueness} show that \((Y^x_t)_{t \geq 0}\) is adapted, too.
	Next, five technical lemmata follow.
	\begin{lemma}\label{lem: 1}
	The law of \((Y^x_t, Z_t)_{t \geq 0}\) is given by \(P_{(x, i)}\).
	\end{lemma}
\begin{proof}
	The process \((V_t)_{t \geq 0} \triangleq F^{k} (Y^{n - 1, x}_{\sigma_n}, W^n)\) has the dynamics
	\begin{align*}
	\dd V_t &= b(V_t, k) \dd t + a^\frac{1}{2} (V_t,k) \dd W^n_t, \quad V_0 = Y^{n-1, x}_{\sigma_n}.
	\end{align*}
	Thus, due to classical rules for time-changed stochastic integrals, for \(t \in [\tau_n, \tau_{n+1}]\) on \(\{Z_{\tau_n} = k\}\) we have
	\begin{align*}
	Y^{n, x}_{t - \tau_n} &= F^{k}(Y^{n-1, x}_{\sigma_n}, W^n)_{t - \tau_n}
	\\&= Y^{n-1, x}_{\sigma_n} + \int_0^{t - \tau_n} b(V_s, k)\dd s+ \int_0^{t - \tau_n} a^\frac{1}{2}(V_s, k) \dd W^n_s
	\\&= Y^{n-1, x}_{\sigma_n} + \int_{\tau_n}^{t} b(Y^{n, x}_{s - \tau_n}, k)\dd s + \int_{\tau_n}^t a^\frac{1}{2} (Y^{n, x}_{s - \tau_n}, k) \dd W_s
	\\&= Y^{n-1, x}_{\sigma_n} + \int_{\tau_n}^{t} b(Y^x_s, k)\dd s + \int_{\tau_n}^t a^\frac{1}{2} (Y^x_s, k) \dd W_s
	\\&= Y^{n-1,x}_{\sigma_n} + \int_{\tau_n}^{t} b(Y^x_s, Z_{s})\dd s + \int_{\tau_n}^t a^\frac{1}{2} (Y^x_s, Z_{s}) \dd W_s.
	\end{align*}
	Iterating yields that for \(t \in [\tau_n, \tau_{n + 1}]\)
	\[
	Y^{n, x}_{t - \tau_n} = x + \int_0^t b(Y^x_s, Z_{s}) \dd s + \int_0^t a^\frac{1}{2} (Y^x_s, Z_{s})\dd W_s.
	\]
	Therefore, the process \((Y^x_t)_{t \geq 0}\) satisfies the SDE
\[
\dd Y^x_t = b(Y^x_t, Z_t) \dd t + a^\frac{1}{2} (Y^x_t, Z_t) \dd W_t,\quad Y^x_0 = x,
\]
and, consequently, the uniqueness of \(P_{(x, i)}\) and Lemma \ref{lem: representation} imply that the law of \((Y^x_t, Z_t)_{t \geq 0}\) coincides with \(P_{(x, i)}\).
\end{proof}
\begin{lemma}\label{lem: ind}
	For all Borel sets \(G \subseteq \Sigma_c\) we have a.s. \[P\big((W^n_t)_{t \geq 0} \in G| \sigma (\mathcal{G}_{\tau_n}, \sigma_{n + 1})\big) = P\big((W^n_t)_{t \geq 0} \in G\big).\]
\end{lemma}
\begin{proof}
	Let \(\mathscr{W}_z\) be the Wiener measure with starting value \(z\in \mathbb{R}^d\) and \(P^\star_k\) be the law of a Markov chain with \(Q\)-matrix \(Q\) and starting value \(k \in S_d\).
	Due to Remark \ref{rem: after lem 2}, Proposition \ref{prop: existence initial law} in Appendix \ref{app: prop exietence initial law}
	 and \cite[Proposition 4.1.5, Theorems 4.4.2]{EK}, the map \((z,k) \mapsto \mathscr{W}_z \otimes P^\star_k\) is Borel and the process \((W_t, Z_t)_{t \geq 0}\) is a strong Markov process in the following sense: For all \(F \in \mathcal{F}\) and all a.s. finite \((\mathcal{G}_t)_{t \geq 0}\)-stopping times \(\theta\) a.s.
	\[
	P \big((W_{t + \theta}, Z_{t + \theta})_{t \geq 0} \in F |\mathcal{G}_\theta\big) = \big(\mathscr{W}_{W_\theta} \otimes P^\star_{Z_\theta}\big) (F).
	\]
	Let \(F \subseteq \Sigma_d\) be Borel.
	The strong Markov properties of \((Z_t)_{t \geq 0}, (W_t)_{t \geq 0}\) and \((W_t, Z_t)_{t \geq 0}\) imply that a.s.
	\begin{align*}
	P \big((W_{t + \tau_n})_{t \geq 0} \in G&, (Z_{ t + \tau_n})_{t \geq 0} \in F | \mathcal{G}_{\tau_n}\big) 
	\\&=  \mathscr{W}_{W_{\tau_n}} (G)\ P^\star_{Z_{\tau_n}} (F)
\\&= P\big((W_{t + \tau_n})_{t \geq 0} \in G |\mathcal{G}_{\tau_n}\big)  P\big((Z_{t + \tau_n})_{t \geq 0} \in F|\mathcal{G}_{\tau_n}\big). 
	\end{align*}
	This implies that \(\sigma(W^n_t, t \in \mathbb{R}_+)\) and \(\sigma (\sigma_{n + 1})\) are independent given \(\mathcal{G}_{\tau_n}\). Thus, \cite[Proposition 5.6]{Kallenberg} and the independence of \(\sigma(W^n_t, t \in \mathbb{R}_+)\) and \(\mathcal{G}_{\tau_n}\) yield that a.s.
	\[
	P\big((W^n_t)_{t \geq 0} \in G | \sigma (\mathcal{G}_{\tau_n},\sigma_{n + 1})\big) = P\big((W^n_t)_{t \geq 0} \in G| \mathcal{G}_{\tau_n}\big) = P\big((W^n_t)_{t \geq 0} \in G\big),
	\]
	which is the claim.
\end{proof}

\begin{lemma}\label{lem: 2}
	For all \(n \in \mathbb{N}_0\) we have \(\|Y^{n, x}_{\sigma_{n+1}}\| \to \infty\) in probability as \(\|x\| \to \infty\).
\end{lemma}
\begin{proof}
	We use induction. Because the process \((Y^{0, x}_t)_{t \geq 0}\) has law \(P^i_x\) (by the uniqueness assumption) and \((Y^{0, x}_t)_{t \geq 0}\) is independent of \(\sigma_1 = \tau_1\), we can conclude the induction base from the hypothesis (ii) of Theorem \ref{prop: equiv}. More precisely, we have for all \(m \in \mathbb{N}\)
	\[
	P(\|Y^{0, x}_{\sigma_1}\| \leq m) = \int_0^\infty P^i_x (\|X_{s}\| \leq m) P(\sigma_1 \in \dd s) \to 0
	\] 
	as \(\|x\| \to \infty\), see the proof of Proposition \ref{lem: non-existence}. Suppose now that the claim holds for \(n \in \mathbb{N}_0\). Using the Lemmata \ref{lem: representation} and \ref{lem: ind} and \cite[Theorem 5.4]{Kallenberg}, we obtain
	\begin{equation}\label{eq: indep comp}
	\begin{split}
	P(\|&Y^{n+1, x}_{\sigma_{n + 2}}\| \leq m) 
	\\&= \sum_{k = 1}^N P(\|F^k(Y^{n, x}_{\sigma_{n +1}}, W^{n +1})_{\sigma_{n + 2}}\| \leq m, Z_{\tau_{n +1}} = k)
	\\&= \sum_{k = 1}^N E \big[ P(\|F^k(Y^{n, x}_{\sigma_{n +1}}, W^{n +1})_{\sigma_{n + 2}}\| \leq m |\sigma (\mathcal{G}_{\tau_{n + 1}}, \sigma_{n + 2})) \1 \{ Z_{\tau_{n +1}} = k\} \big]
	\\&= \sum_{k = 1}^N \int P(\|F^k(Y^{n, x}_{\sigma_{n + 1} (\omega)} (\omega), W^{n +1})_{\sigma_{n + 2}(\omega)}\| \leq m) \1\{Z_{\tau_{n +1}(\omega)} (\omega) = k\} P(\dd \omega) 
	\\	&= \sum_{k = 1}^N \int P^k_{Y^{n, x}_{\sigma_{n + 1} (\omega)}(\omega)}(\|X_{\sigma_{n + 2}(\omega)}\| \leq m) \1\{Z_{\tau_{n + 1}(\omega)} (\omega) = k\} P(\dd \omega).
	\end{split}
\end{equation}
	Take \((x_k)_{k \in \mathbb{N}} \subset \mathbb{R}^d\) such that \(\|x_k\| \to \infty\) as \(k \to \infty\). 
	A well-known characterization of convergence in probability is the following: A sequence \((Z^k)_{k \in \mathbb{N}}\) converges in probability to a random variable \(Z\) if and only if each subsequence of \((Z^k)_{k \in \mathbb{N}}\) contains a further subsequence which converges almost surely to \(Z\), see, e.g., \cite[Lemma 3.2]{Kallenberg}. 
	Consequently, \((x_k)_{k \in \mathbb{N}}\) contains a subsequence \((x_{n'_k})_{k \in \mathbb{N}}\) such that \[\Big\|Y^{n, x_{n'_k}}_{\sigma_{n + 1}}\Big\|\to \infty\] almost surely as \(k \to \infty\). Due to the dominated convergence theorem, we deduce from \eqref{eq: indep comp} that \[\Big\|Y^{n + 1, x_{n'_k}}_{\sigma_{n + 2}}\Big\| \to \infty\] in probability as \(k \to \infty\). Thus, applying again the subsequence criterion, we can extract a further subsequence such that the convergence holds almost surely. Finally, applying the subsequence criterion a third time (but this time the converse direction), we conclude the claim.
\end{proof}
\begin{lemma}\label{lem: 3}
	For all \(n \in \mathbb{N}_0, t > 0\) we have \(\|Y^{n, x}_{t - \tau_n}\| \to \infty\) on \(\{\tau_n < t\}\) in probability as \(\|x\| \to \infty\).
\end{lemma}
\begin{proof}
	Because \(\sigma(W^n_t, t \in \mathbb{R}_+)\) is independent of \(\mathcal{G}_{\tau_n}\), we show as in the proof of Lemma \ref{lem: 2} that
	\begin{align} 
	P (\|&Y^{n, x}_{t - \tau_n}\| \leq m, \tau_n < t)\nonumber \\&= \sum_{k = 1}^N \int P^k_{Y^{n-1, x}_{\sigma_n (\omega)} (\omega)} (\|X_{t - \tau_n(\omega)}\| \leq m) \1 \{\tau_n(\omega) < t\} \1 \{Z_{\tau_n(\omega)} (\omega) = k\}P(\dd \omega).\nonumber
	\end{align}
	Using Lemma \ref{lem: 2} and the argument in its proof, we see that the claim follows.
\end{proof}
\begin{lemma}\label{lem: main}
	For all compact sets \(K \subset \mathbb{R}^d\) and all \(t, \epsilon > 0\) there exists a compact set \(K^* \subset \mathbb{R}^d\) such that 
	\[
	P_{(x, i)} (X_t \in K \times S_d) <\epsilon
	\]
	for all \(x \not \in K^*\). 
\end{lemma}
\begin{proof}
Let \(f \in C_0(\mathbb{R}^d)\) be such that \(0 \leq f \leq 1\) and \(f \equiv 1\) on \(K\). As before \(f\) exists due to  Urysohn's lemma for locally compact spaces. We have 
\begin{align*}
E \big[ f(Y^x_t)\big] &= \sum_{n = 0}^\infty E \big[ f(Y^x_t) \1 \{\tau_n < t \leq \tau_{n + 1}\} \big]
\\&= \sum_{n = 0}^\infty E \big[ f(Y^{n, x}_{t - \tau_n}) \1 \{\tau_n < t \leq \tau_{n + 1}\} \big] \to 0
\end{align*}
as \(\|x\| \to \infty\), which follows from Lemma \ref{lem: 3} and the dominated convergence theorem. 
Thus, the map \(x \mapsto E\big[f(Y^x_t)\big]\) is an element of \(C_0(\mathbb{R}^d)\). Finally, noting that
\[
P_{(x, i)} (X_t \in K \times S_d) \leq E \big[ f(Y^x_t) \big]
\]
 implies the claim.
\end{proof}

We are in the position to complete the proof.
Fix \(t, \epsilon > 0\) and a compact set \(K \subset S\). Recall that \(\pi_1 \colon S \to \mathbb{R}^d\) and \(\pi_2 \colon S \to S_d\) are the usual projections.
Because \((P^\star_i)_{i \in S_d}\) is Feller-Dynkin, there exists a compact set \(K^* \subset S_d\) such that 
\[
P^\star_i \big(X_t \in \pi_2(K)\big) < \epsilon
\]
for all \(i \not \in K^*\). By Lemma \ref{lem: main}, for each \(i \in K^*\) we find a compact set \(K_i^* \subset \mathbb{R}^d\) such that 
\[
P_{(x, i)} \big(X_t \in \pi_1(K) \times S_d\big) < \epsilon
\]
for all \(x \not \in K^*_i\).
Define \(\widehat{K} \triangleq \big(\bigcup_{i \in K^*} K^*_i\big) \times K^*\subset S\). Clearly, \(\widehat{K}\) is compact. We claim that
\[
P_{(x, i)} \big(X_t \in K\big) < \epsilon 
\]
for all \((x, i) \not \in \widehat{K}\). 
To see this, note that 
\begin{align*}
\widehat{K}^c 
= \bigg(\Big(\bigcap_{i \in K^*} (K^*_i)^c \Big) \times K^* \bigg) \cup \bigg( \mathbb{R}^d \times (K^*)^c \bigg).
\end{align*}
Now, if \((x, i) \in \mathbb{R}^d \times (K^*)^c\) we have 
\[
P_{(x, i)} \big(X_t \in K\big) \leq P_{(x, i)} \big(X_t \in \pi_1(K) \times \pi_2(K)\big) \leq P^\star_i \big(X_t \in \pi_2(K)\big)< \epsilon.
\]
If \((x, i) \in \big(\bigcap_{j \in K^*} (K^*_j)^c \big) \times K^*\) we have \(x \not \in K^*_i\) and hence
\[
P_{(x, i)} \big(X_t \in K\big) \leq P_{(x, i)} \big(X_t \in \pi_1(K) \times S_d\big) < \epsilon.
\]
This proves the claim, which itself implies that \((P_x)_{x \in S}\) is Feller-Dynkin, see the proof of Theorem \ref{theo: Ly}.
\qed

\subsubsection{Proof of Proposition \ref{prop: sd feller}}
	Due to Remark \ref{rem: after lem 2}, the strong Markov property follows from Proposition \ref{prop: markov}. 

It remains to prove that \((P_x)_{x \in S}\) has the \(C_b\)-Feller property. It suffices to show that \(x \mapsto P_x\) is continuous, i.e. that \(x_n \to x\) implies \(P_{x_n} \to P_x\) weakly as \(n \to \infty\). In this case, because for all \(x \in S\) and \(t \in \mathbb{R}_+\) the map \(\omega \mapsto \omega(t)\) is \(P_x\)-a.s. continuous (see \cite[Proposition 3.5.2]{EK} and note that \(P_x(\Delta X_t \not = 0) = 0\)), the continuous mapping theorem implies that \((P_x)_{x \in S}\) has the \(C_b\)-Feller property.
	The continuity of \(x \mapsto P_x\) follows from Theorem \ref{theo: limit} below.  \qed
	
	\begin{theorem}\label{theo: limit}
		For all \(n \in \mathbb{N}\) let \(b_n \colon S \to \mathbb{R}^d\) and \(a_n \colon S \to \mathbb{S}^d\) be Borel functions such that for all \(m \in \mathbb{R}_+\)
		\begin{align}\label{eq: local boundedness}
		\sup_{n \in \mathbb{N}} \sup_{\tri y\tri \leq m} \big(\|b_n(y)\| + \|a_n(y)\| \big) < \infty,
		\end{align}
		where \(\tri \cdot \tri\) denotes the Euclidean norm on \(\mathbb{R}^{d + 1}\). 
		Assume that \(b \colon S \to \mathbb{R}^d\) and \(a \colon S \to \mathbb{S}^d\) are continuous functions and that for all \(m \in \mathbb{R}_+\)
		\begin{align}\label{eq: assp conv}
		\sup_{\tri y \tri \leq m} \big( \|b (y) - b_n (y)\| + \|a (y)- a_n (y)\|\big) \to 0
		\end{align}
		as \(n \to \infty\).
		Furthermore, let \((Q_n)_{n \in \mathbb{N}}\) be a sequence of \(Q\)-matrices on \(S_d\) such that for all \(n \in \mathbb{N}\) and \(i \in S_d\) the MP \((C_n, Q_n, \Sigma_d, i)\), where \[
		C_n \triangleq \big\{f \in C_0(S_d) \colon Q_n f \in C_0(S_d)\big\},
		\] has a unique solution \(P^n_i\) such that \((P^n_i)_{i \in S_d}\) is Feller-Dynkin. Let \(C^\star \subseteq C\) be as in Remark \ref{rem: after lem 2}.
		Suppose that for all \(f \in C^\star\) there exists a sequence \((f_n)_{n \in \mathbb{N}}\) consisting of \(f_n \in C_n\) such that
		\begin{align}\label{eq: Q conv}
		\|f - f_n\|_\infty + \|Q f - Q_n f_n\|_\infty \to 0
		\end{align}
		as \(n \to \infty\).
		Finally, take \((x_n)_{n \in \mathbb{N}} \subset \mathbb{R}^d\) and \((i_n)_{n \in \mathbb{N}} \subset S_d\) such that \(x_n \to x \in \mathbb{R}^d\) and \(i_n \to i \in S_d\) as \(n \to \infty\). Set \(\mathcal{L}\) as in \eqref{gen: sd}, \(\mathcal{L}_n\) as in \eqref{gen: sd} with \(b\) replaced by \(b_n\), \(a\) replaced by \(a_n\) and \(Q\) replaced by \(Q_n\), and \(D\) as in \eqref{eq: D}.
		If \(P^n\) is a solution to the MP \((D_n, \mathcal{L}_n, \Sigma, (x_n, i_n))\), where
		\[
		D_n \triangleq \big\{f, g \colon f \in C^2_c(\mathbb{R}^d), g \in C_n\big\},
		\]
		and for all \(y \in S\) the MP \((D, \mathcal{L}, \Sigma, y)\) has a unique solution \(P_y\), then \(P^n \to P_{(x, i)}\) weakly as \(n \to \infty\).
	\end{theorem}
	\begin{proof}
	We adapt the strategy from the proof of \cite[Theorem IX.3.39]{JS}.
Let us start with a clarification of our terminology: When we say that a sequence of \cadlag processes is tight, we mean that its laws are tight or, equivalently, relatively compact by Prohorov's theorem (see \cite[Theorem 3.2.2]{EK}). If we speak of an accumulation point of a sequence of processes, we refer to an accumulation point of the corresponding sequence of laws. 
	
Because of the discrete topology of \(S_d\) we can assume that \(i_n \equiv i\). For all \(n \in \mathbb{N}\) denote by \((Y^n_t)_{t \geq 0}, (Z^n_t)_{t \geq 0}\) and \((W^n_t)_{t \geq 0}\) the processes from Lemma \ref{lem: representation} corresponding to \(P^n\). 
For \(m \in \mathbb{R}_+\) we define
\begin{equation}\label{eq: def tau}
\begin{split}
\tau_{m} &\triangleq \inf\big(t \in \mathbb{R}_+ \colon \tri X_t\tri \geq m \text{ or } \tri X_{t-}\tri \geq m\big).
\end{split}
\end{equation}
We note that \(\tau_{m}\) is an \((\mathcal{F}^o_t)_{t \geq 0}\)-stopping time, see \cite[Proposition 2.1.5]{EK}.
For \(n \in \mathbb{N}\) and \(m \in \mathbb{R}_+\) we set
\[
	\tau_{n, m} \triangleq \tau_m \circ (Y^n_t, Z^n_t)_{t \geq 0}.
\]
Next, four technical lemmata follow.
\begin{lemma}\label{lem: tight1}
	For all \(m \in \mathbb{R}_+\) the sequence \(\{(Y^n_{t \wedge \tau_{n, m}}, Z^n_t)_{t \geq 0}, n \in \mathbb{N}\}\) is tight.
\end{lemma}
\begin{proof}
The Kato-Trotter theorem \cite[Theorem 17.25]{Kallenberg} implies that \(\{(Z^n_t)_{t \geq 0}, n \in \mathbb{N}\}\) is tight in \(\Sigma_d\) equipped with the Skorokhod topology. 
For all \(n \in \mathbb{N}\) the process \((Y^n_{t \wedge \tau_{n, m}})_{t \geq 0}\) has continuous paths. 
Below, we show that \(\{(Y^n_{t \wedge \tau_{n, m}})_{t \geq 0},n \in \mathbb{N}\}\) is tight in \(\Sigma_c\) equipped with the local uniform topology. In this case, \cite[Problem 4.25]{EK} 
 implies that \(\{(Y^n_{t \wedge \tau_{n, m}})_{t \geq 0},n \in \mathbb{N}\}\) is also tight in the space of \cadlag functions \(\mathbb{R}_+\to \mathbb{R}^d\) equipped with the Skorokhod topology, which we denote by \(D(\mathbb{R}_+, \mathbb{R}^d)\).

We claim that this already implies the tightness of \(\{(Y^n_{t \wedge \tau_{n, m}}, Z^n_t)_{t \geq 0}, n \in \mathbb{N}\}\). To see this, we use the characterization of tightness given in \cite[Corollary 3.7.4]{EK}. Let us recall it as a fact:
\begin{fact}\label{fact}
	Let \((E, r)\) be a Polish space.
	A sequence \((\mu^n)_{n \in \mathbb{N}}\) of Borel probability measures on \(D(\mathbb{R}_+, E)\) is tight if and only if the following hold:
	\begin{enumerate}
		\item[\textup{(a)}]
		For all \(t \in \mathbb{Q}_+\) and \(\epsilon > 0\) there exists a compact set \(C(t, \epsilon) \subseteq E\) such that
		\[
		\limsup_{n \to \infty} \mu^n(X_t \not \in C(t, \epsilon))\leq \epsilon. \]
		\item[\textup{(b)}]
		For all \(\epsilon > 0\) and \(t > 0\) there exists a \(\delta > 0\) such that 
		\[
		\limsup_{n \to \infty} \mu^n(w'(X, \delta, t) \geq \epsilon) \leq \epsilon, 
		\]
		where
		\begin{align*}
		w'(\alpha, \theta, t) &\triangleq \inf_{\{t_i\}}  \max_{i} \sup_{u, v \in [t_{i-1}, t_i)} r(\alpha(u), \alpha(v)),
		\end{align*}
		with \(\{t_i\}\) ranging over all partitions of the form \(0 = t_0 < t_1 < \dots < t_{n -1} < t_n \leq t\) with \(\min_{1 \leq i < n}(t_i - t_{i-1}) \geq \theta\) and \(n \geq 1\).
	\end{enumerate}
\end{fact}
We equip \(S\) with the metric \(r((x, i), (y, j)) \triangleq \|x - y\| + \1\{i \not = j\}\), which generates the product topology on \(S\).
Let us first check that \(\{(Y^n_{t \wedge \tau_{n, m}}, Z^n_t)_{t \geq 0}, n \in \mathbb{N}\}\) satisfies Fact \ref{fact} (a). Fix \(t \in \mathbb{Q}_+\) and \(\epsilon > 0\). Using Fact \ref{fact}, the tightness of \(\{(Z^n_s)_{s \geq 0}, n \in \mathbb{N}\}\) in \(\Sigma_d\) and \(\{(Y^n_{s \wedge \tau_{n,m}})_{s\geq 0}, n \in \mathbb{N}\}\) in \(D(\mathbb{R}_+, \mathbb{R}^d)\) implies that there exists a compact set \(C_1(t, \epsilon) \subset S_d\) and a compact set \(C_2(t, \epsilon) \subset \mathbb{R}^d\) such that 
\begin{align*}
\limsup_{n \to \infty} P(Z^n_{t} \not \in C_1(t, \epsilon)) &\leq \tfrac{\epsilon}{2},\\
\limsup_{n \to \infty} P(Y^n_{t \wedge \tau_{n,m}} \not \in C_2(t, \epsilon)) &\leq \tfrac{\epsilon}{2}.
\end{align*}
The set \(K(t, \epsilon) \triangleq C_2(t, \epsilon) \times C_1(t, \epsilon)\subset S\) is also compact and we have 
\begin{align*}
\limsup_{n \to \infty} P&((Y^n_{t \wedge \tau_{n,m}}, Z^n_t) \not \in K(t, \epsilon)) \\&\leq \limsup_{n \to \infty} P(Y^n_{t \wedge \tau_{n,m}} \not \in C_2(t, \epsilon))  + \limsup_{n \to \infty} P(Z^n_{t} \not \in C_1(t, \epsilon))
\\&\leq \tfrac{\epsilon}{2} + \tfrac{\epsilon}{2} = \epsilon.
\end{align*}
In other words, \(\{(Y^n_{t \wedge \tau_{n,m}}, Z^n_t)_{t \geq 0}, n \in \mathbb{N}\}\) satisfies Fact \ref{fact} (a). Next, we explain that it also satisfies Fact \ref{fact} (b).
We claim that the continuous paths of \((Y^n_t)_{t \geq 0}\) imply that up to a null set
\begin{align}\label{eq: ineq JS}
w'((Y^n_{s \wedge \tau_{n, m}}, Z^n_s)_{s \geq 0}, \theta, t) \leq 2 w'((Y^n_{s\wedge \tau_{n,m}}, 0)_{s \geq 0},2 \theta, t) + w'((0, Z^n_s)_{s \geq 0}, \theta, t).
\end{align}
To see this, take \((\alpha, \omega)\in \Sigma_c \times \Sigma_d\). Let \(\{t_i\}\) be a partition of the form \(0 = t_0 < t_1 < \dots < t_{n-1} < t_n \leq t\) with \(\min_{1 \leq i \leq n} (t_i - t_{i-1}) \geq \theta\). By adding points if necessary, we can assume that \(\max_{1 \leq i \leq n} (t_i - t_{i - 1}) \leq 2 \theta\). In this case, we have 
\begin{align*}
\sup_{u, v \in [t_{i-1}, t_i)} &r((\alpha(u), 0), (\alpha(v), 0)) \\&\leq \sup \big\{ r((\alpha(u), 0), (\alpha(v), 0)) \colon 0 \leq u, v \leq t, |u - v| \leq 2 \theta\big\}.
\end{align*}
Due to \cite[Lemma 15.3]{HWY}, we have 
\[
\sup \big\{ r((\alpha(u),0), (\alpha(v), 0)) \colon 0 \leq u, v \leq t, |u - v| \leq 2 \theta\big\} \leq 2 w'( (\alpha, 0), 2 \theta, t).
\]
Therefore, we conclude that 
\[
w'((\alpha, \omega), \theta, t) \leq 2w'((\alpha, 0), 2 \theta, t) + w'((0, \omega),  \theta, t), 
\]
which implies \eqref{eq: ineq JS}.
Fix \(\epsilon > 0\) and \(t > 0\) and let \(\delta > 0\) be such that 
\begin{align*}
\limsup_{n \to \infty} P\big(w'((Y^n_{s \wedge \tau_{n,m}}, 0)_{s \geq 0}, 2\delta, t) \geq \tfrac{\epsilon}{4}\big) &\leq \tfrac{\epsilon}{4}, \\
\limsup_{n \to \infty} P\big(w'((0, Z^n_s)_{s \geq 0}, \delta, t) \geq \tfrac{\epsilon}{2}\big) &\leq \tfrac{\epsilon}{2}.
\end{align*}
This \(\delta\) exists due to Fact \ref{fact} (b) and the fact that \(w'\) is increasing in \(\delta\).
Note that for two non-negative random variables \(V\) and \(U\) we have 
\begin{align*}
P(U + V \geq 2\epsilon) 
&\leq P(U \geq \epsilon) + P(V \geq \epsilon).
\end{align*}
Hence, we deduce from \eqref{eq: ineq JS} that
\begin{align*}
\limsup_{n \to \infty} P(w'((Y^n_{s \wedge \tau_{n,m}}, Z^n_s)_{s \geq 0}, \delta, t) \geq \epsilon ) &\leq \tfrac{3 \epsilon}{4} \leq \epsilon.
\end{align*}
We conclude from Fact \ref{fact} that \(\{(Y^n_{t \wedge \tau_{n,m}}, Z^n_t)_{t \geq 0}, n \in \mathbb{N}\}\) is tight.

	It remains to show that \(\{(Y^n_{t \wedge \tau_{n, m}})_{t \geq 0},n \in \mathbb{N}\}\) is tight in \(\Sigma_c\). Let \(p > 2\) and recall the inequalities
	\begin{align}\label{eq: easy}
	\big(v + u\big)^p \leq 2^p \big(v^p + u^p\big),\quad v, u \geq 0,\qquad \bigg\| \int_0^t f(s)\dd s \bigg\| \leq \int_0^t \|f(s)\|\dd s.
	\end{align}
	Let \(T \in \mathbb{R}_+\) and \(s < t \leq T.\) We  write \(x \preceq y\) whenever \(x \leq \textup{const. } y\) where the constant only depends on \(T, p, m\) and \(\eqref{eq: local boundedness}\). We deduce from the triangle inequality, \eqref{eq: easy} and \cite[Remark 3.3.30]{KaraShre} (i.e. a multidimensional version of the Burkholder-Davis-Gundy inequality) that
	\begin{equation}\label{eq: tight bound}
	\begin{split}
	E \big[ &\|Y^n_{t \wedge \tau_{n,m}} - Y^n_{s \wedge \tau_{n, m}}\|^p \big] \\&= E \bigg[ \bigg\|\int_{s \wedge \tau_{n, m}}^{t \wedge \tau_{n, m}} b_n(Y^n_r, Z^n_r) \dd r + \int_{s \wedge \tau_{n, m}}^{t \wedge \tau_{n, m}} a_{n}^{\frac{1}{2}}(Y^n_r, Z^n_r) \dd W^n_r \bigg\|^p \bigg]
	\\&\leq 2^p E \bigg[ \bigg\|\int_{s \wedge \tau_{n, m}}^{t \wedge \tau_{n, m}} b_n(Y^n_r, Z^n_r) \dd r \bigg\|^p \bigg] + 2^p E\bigg[ \bigg\|\int_{s \wedge \tau_{n, m}}^{t \wedge \tau_{n, m}} a_{n}^{\frac{1}{2}}(Y^n_r, Z^n_r) \dd W^n_r \bigg\|^p \bigg]
	\\&\preceq  E \bigg[ \bigg(\int_{s \wedge \tau_{n, m}}^{t \wedge \tau_{n, m}} \|b_n(Y^n_r, Z^n_r)\| \dd r \bigg)^p \bigg] + E\bigg[ \bigg(\int_{s \wedge \tau_{n, m}}^{t \wedge \tau_{n, m}} \|a_n(Y^n_r, Z^n_r)\| \dd r \bigg)^\frac{p}{2} \bigg]
	\\&\preceq \big( |t - s|^p +  |t - s|^\frac{p}{2} \big)
	\\&\preceq |t - s|^\frac{p}{2}. 
	\end{split}\end{equation}
Furthermore, we have
	\[
	\sup_{n \in \mathbb{N}} E \big[ \| Y^n_0\|\big] = \sup_{n \in \mathbb{N}} \|x_n\| < \infty, 
	\]
	because convergent sequences are bounded.
	Consequently, \cite[Problem 2.4.11, Remark 2.4.13]{KaraShre} (i.e. Kolmogorov's tightness criterion) imply that \(\{(Y^n_{t \wedge \tau_{n, m}})_{t \geq 0}, n \in \mathbb{N}\}\) is tight in \(\Sigma_c\). This completes the proof.
\end{proof}
The following lemma is a version of Lemma \ref{lem: loc pathwise uniqueness} for uniqueness in law instead of pathwise uniqueness.
\begin{lemma}\label{lem: lu}
	Let \(\rho\) be an \((\mathcal{F}^o_t)_{t \geq 0}\)-stopping time and suppose that \(P\) is a probability measure on \((\Omega, \mathcal{F})\) such that \(P (X_0 = x) = P(\Sigma) = 1\) and 
	\begin{align}\label{eq: Mf2}
	M^f_{t \wedge \rho} = f(X_{t \wedge \rho}) - f(X_0) - \int_0^{t \wedge \rho} \mathcal{L} f(X_s)\dd s, \quad t \in \mathbb{R}_+,
	\end{align}
	is a \(P\)-martingale for all \(f \in D\). Then, \(P = P_x\) on \(\mathcal{F}^o_\rho\).
\end{lemma}
\begin{proof}
The claim of this lemma is closely related to the concept of local uniqueness as introduced in \cite{JS} and it can be proven with the strategy from \cite[Theorem III.2.40]{JS}. To each \(G \in \mathcal{F}\) we can associate a (not necessarily unique) set \(G' \in \mathcal{F}^o_\rho \otimes \mathcal{F}\) such that 
\[
G \cap \{\rho < \infty\} = \big\{\omega \in \Omega \colon \rho(\omega) < \infty, (\omega, \theta_{\rho(\omega)}\omega) \in G'\big\},
\]
see \cite[Lemma III.2.44]{JS}. Now, set
\[
Q (G) \triangleq P(G \cap \{\rho = \infty\}) + \iint  \1_{\{\rho(\omega) < \infty\}} \1_{G'} (\omega, \omega^*) P_{\omega(\rho(\omega))} (\dd \omega^*)  P (\dd \omega).
\]
Due to \cite[Lemma III.2.47]{JS}, \(Q\) is a probability measure on \((\Omega, \mathcal{F})\). 
For \(G \in \mathcal{F}^o_0\) we can choose \(G' = G \times \Omega\). Consequently, we have
\(
Q(X_0 = x) = P(X_0 = x) = 1. 
\)
Set 
\[
\Sigma^* \triangleq \big\{ \omega \in \Omega \colon (\omega_{t \wedge \rho(\omega)})_{t \geq 0} \in \Sigma \big\} \supseteq \Sigma
\]
and note that 
\[
\Sigma \cap \{\rho< \infty\} = \big\{ \omega \in \Omega \colon \rho (\omega) < \infty, (\omega, \theta_{\rho(\omega)} \omega) \in \Sigma^* \times \Sigma\big\}.
\]
Consequently, we have 
\begin{align*}
Q(\Sigma) &= P(\Sigma \cap \{\rho = \infty \}) +\int \1_{\{\rho (\omega) < \infty\}} \1_{\Sigma^*}(\omega) P_{\omega(\rho(\omega))}(\Sigma) P(\dd \omega) 
\\&= P (\Sigma\cap \{\rho = \infty\}) + P (\Sigma^* \cap \{\rho < \infty\}) \geq P(\Sigma)
= 1.
\end{align*}
Fix a bounded \((\mathcal{F}^o_t)_{t \geq 0}\)-stopping time \(\psi\).
For \(\omega, \alpha \in \Omega\) and \(t \in \mathbb{R}_+\) we set
\[
z(\omega, \alpha) (t) \triangleq \begin{cases} \omega(t),&t < \rho(\omega),\\
\alpha(t - \rho(\omega)),&t \geq \rho(\omega),\end{cases}
\]
and
\[
V(\omega, \alpha) \triangleq \begin{cases} \big(\psi \vee \rho - \rho\big) (z(\omega, \alpha)), &\alpha (0) = \omega (\rho(\omega)),\\
0,&\text{otherwise}.\end{cases}
\]
Due to \cite[Theorem IV.103]{DellacherieMeyer78} the map
\(V\) is \(\mathcal{F}^o_\rho \otimes \mathcal{F}\)-measurable and \(V(\omega, \cdot)\) is an \((\mathcal{F}^o_t)_{t \geq 0}\)-stopping time for all \(\omega \in \Omega\). Furthermore, it is evident from the definition that
\[
\psi (\omega) \vee \rho(\omega) = \rho(\omega) + V(\omega, \theta_{\rho(\omega)} \omega)
\]
for \(\omega \in \Omega\). 
We take \(f \in D\) and note that for \(\omega \in \{\rho < \psi\}\)
\begin{align*}
M^f_{V(\omega, \theta_{\rho(\omega)} \omega)}(\theta_{\rho(\omega)}\omega) &= M^f_{\psi(\omega) - \rho(\omega)} (\theta_{\rho(\omega)} \omega )  
= M^f_{\psi(\omega)}(\omega) - M^f_{\rho(\omega)}(\omega).
\end{align*}
Because \((M^f_{t \wedge \rho})_{t \geq 0}\) is a \(P\)-martingale and \(\psi\) is bounded, the optional stopping theorem yields that
\[
E^Q \big[ M^f_{\rho \wedge \psi}\big] = E^{P} \big[M^f_{\rho \wedge \psi}\big] = 0.
\]
Therefore, we have 
\begin{align*}
E^Q \big[ M^f_{\psi} \big] &= E^Q \big[ M^f_{\psi} - M^f_{\rho \wedge \psi}\big] 
\\&= E^Q \big[ \big(M^f_{\psi} - M^f_{\rho}\big) \1_{\{\rho < \psi\}}\big]
\\&= E^Q\big[M^f_{V (\cdot, \theta_\rho)} ( \theta_\rho ) \1_{\{\rho < \psi\}}\big]
\\&=\int  E^{P_{\omega(\rho(\omega))}} \big[ M^f_{V (\omega, \cdot)}\big] \1_{\{\rho(\omega) < \psi(\omega)\}} P(\dd \omega)= 0,
\end{align*}
again due to the optional stopping theorem (recall that \(V(\omega, \cdot)\) is bounded and that \((M^f_{t})_{t \geq 0}\) is a \(P_y\)-martingale for all \(y \in S\)). We conclude from \cite[Proposition II.1.4]{RY} and the downwards theorem (\cite[Theorem II.51.1]{RW1}) that \((M^f_{t})_{t \geq 0}\) is a \(Q\)-martingale, which
 implies that \(Q\) solves the MP \((D, \mathcal{L}, \Sigma, x)\).
The uniqueness assumption yields that \(Q = P_x\). Because also for \(G \in \mathcal{F}^o_\rho\) we can choose \(G' = G \times \Omega\), we obtain that 
\[
P_x(G) = Q(G)  = P(G).
\]
This finishes the proof.
\end{proof}
\begin{lemma}\label{lem: uni1}
For all \(m \in \mathbb{N}\), all accumulation points of \(\{(Y^n_{t \wedge \tau_{n, m}}, Z^n_t)_{t \geq 0},n \in \mathbb{N}\}\) coincide with \(P_{(x, i)}\) on \(\mathcal{F}^o_{\tau_{m-1}}\).
\end{lemma}
\begin{proof}
	We recall some continuity properties of functions on \(\Omega\). For \(\omega \in \Omega\), define
\begin{align*}
J(\omega) &\triangleq \big\{t > 0 \colon \omega (t) \not = \omega(t-)\big\},\\
V(\omega) &\triangleq \big\{k > 0 \colon \tau_k (\omega) < \tau_{k + }(\omega)\big\},\\
V'(\omega)&\triangleq \big\{u > 0 \colon \omega(\tau_u(\omega)) \not = \omega(\tau_u(\omega)-) \text{ and } \tri\omega(\tau_u(\omega)-)\tri = u\big\},
\end{align*}
which are countable sets, see \cite[Lemma VI.2.10]{JS}. The map \(\omega \mapsto \omega(t)\) is continuous at \(\omega\) whenever \(t \not \in J(\omega)\), see \cite[Proposition 3.5.2]{EK}, and the map \(\omega \mapsto \tau_m (\omega)\) is continuous at \(\omega\) whenever \(m \not \in V(\omega)\), see \cite[Problem 13, p. 151]{EK} and \cite[Proposition VI.2.11]{JS}. Furthermore, the map \(\omega \mapsto \omega (\cdot \wedge \tau_m(\omega))\) is continuous at \(\omega\) whenever \(m \not \in V(\omega) \cup V'(\omega)\), see \cite[Problem 13, p. 151]{EK} and \cite[Proposition VI.2.12]{JS}.

Fix \(f \in D\) and let \(Q^m\) be an accumulation point of \(\{(Y^n_{t \wedge \tau_{n, m}}, Z^n_t)_{t \geq 0},n \in \mathbb{N}\}\). Without loss of generality we assume that the law of \((Y^n_{t \wedge \tau_{n, m}}, Z^n_t)_{t \geq 0}\) converges weakly to \(Q^m\) as \(n \to \infty\). The set 
\[
F \triangleq \big\{ t > 0 \colon Q^m (t \in V \cup V') > 0 \big\}
\]
is countable, see the proof of \cite[Proposition IX.1.17]{JS}. Thus, we find a \(t_m \in [m-1, m]\) such that \(t_m \not \in F\). 
Set 
\[
U \triangleq \left\{t \in \mathbb{R}_+ \colon Q^m\big(t \in J(X_{\cdot \wedge \tau_{t_m}})\big) = 0\right\}.
\]
By \cite[Lemma 3.7.7]{EK}, the complement of \(U\) in \(\mathbb{R}_+\) is countable. Thus, \(U\) is dense in \(\mathbb{R}_+\). 
Next, we explain that for all \(z \in \mathbb{R}_+\) the map
\[
\omega \mapsto  I_{t \wedge \tau_z(\omega)}(\omega) \triangleq \int_0^{t \wedge \tau_{z}(\omega)} \mathcal{L}f (\omega(s))\dd s
\]
is continuous at all continuity points of \(\omega \mapsto \tau_{z} (\omega)\). 
Let \((\omega_n)_{n \in \mathbb{N}} \subset\Omega\) and \(\omega \in \Omega\) be such that \(\omega_n \to \omega\) and \(\tau_{z}(\omega_n) \to \tau_{z}(\omega)\) as \(n \to \infty\). 
We deduce from \cite[Proposition 3.5.2]{EK}, the fact that \(J (\omega)\) is countable, the dominated convergence theorem and the continuity of \(x \mapsto \mathcal{L} f (x)\), which is due to the hypothesis that \(b\) and \(a\) are continuous, that 
\[
\big|  I_{t \wedge \tau_z(\omega)} (\omega) - I_{t \wedge \tau_z(\omega)}(\omega_{n}) \big| \to 0
\]
as \(n \to \infty\).
We obtain 
\begin{align*}
\big| I_{t \wedge \tau_z(\omega)} (\omega) &- I_{t \wedge \tau_z(\omega_{n})}(\omega_{n})\big| \\&\leq \big|  I_{t \wedge \tau_z(\omega)} (\omega) - I_{t \wedge \tau_z(\omega)}(\omega_{n}) \big| + \big| I_{t \wedge \tau_z(\omega)}(\omega_{n}) - I_{t \wedge \tau_z(\omega_{n})}(\omega_{n})\big|
\\&\leq \big|  I_{t \wedge \tau_z(\omega)} (\omega) - I_{t \wedge \tau_z(\omega)}(\omega_{n}) \big| +  \big\| \mathcal{L} f\big\|_\infty\ \big| t \wedge \tau_z(\omega) - t \wedge \tau_z(\omega_{n})\big| \to 0
\end{align*}
as \(n \to \infty\),
where we use that \(\tau_z(\omega_n) \to \tau_z(\omega)\) as \(n \to \infty\).
It follows that for each \(t \in U\) there exists a \(Q^m\)-null set \(N_t\) such that the map 
	\begin{align}\label{eq: M f cont}
	\omega \mapsto M^f_{t \wedge \tau_{t_m}(\omega)}(\omega) = f(\omega(t \wedge \tau_{t_m}(\omega))) - f(\omega(0))-  \int_0^{t \wedge \tau_{t_m}(\omega)} \mathcal{L} f(\omega(s))\dd s
	\end{align}
	is continuous at all \(\omega \not \in N_t\). 
	For a moment we fix \(t \in U\).
	Suppose that \(f \in D\) is independent of the \(\mathbb{R}^d\)-coordinate (i.e. \(f \in C^\star\)) and let \((f_n)_{n \in \mathbb{N}}\) be a sequence of functions \(f_n \in C_n\) such that \eqref{eq: Q conv} holds. Define \((M^{f, n}_t)_{t \geq 0}\) as in \eqref{eq: Mf2} with \(f\) replaced by \(f_n\) and  \(\mathcal{L}\) replaced by \(\mathcal{L}_n\). Furthermore, fix \(\omega \not\in N_t\) and let \((\omega_n)_{n \in \mathbb{N}} \subset \Omega\) be a sequence such that \(\omega_n \to \omega\) as \(n \to \infty\). Then, for any bounded continuous function \(v \colon \Omega \to \mathbb{R}\) we have 
	\begin{equation} \label{eq: cont map theo}
	\begin{split}
	\big|M^f_{t \wedge \tau_{t_m}(\omega)} &(\omega) v(\omega) - M^{f, n}_{t \wedge \tau_{t_m}(\omega_n)} (\omega_n) v(\omega_n)\big| \\&\leq \big|M^f_{t \wedge \tau_{t_m}(\omega)} (\omega) v(\omega)- M^{f}_{t \wedge \tau_{t_m}(\omega_n)} (\omega_n) v(\omega_n) \big| \\&\qquad\quad+ \|v\|_\infty \big|M^f_{t \wedge \tau_{t_m}(\omega_n)} (\omega_n) - M^{f, n}_{t \wedge \tau_{t_m}(\omega_n)} (\omega_n) \big| \to 0
	\end{split}
	\end{equation}
	as \(n \to \infty\), where the first term converges to zero because of the continuity of \eqref{eq: M f cont}
	at \(\omega\) and the second term converges to zero because 
	\begin{align*}
	\big|M^f_{t \wedge \tau_{t_m}(\omega_n)} &(\omega_n) - M^{f, n}_{t \wedge \tau_{t_m}(\omega_n)} (\omega_n) \big| \leq 2  \|f - f_n\|_\infty + t \|Q f - Q_n f_n\|_\infty \to 0
	\end{align*}
	as \(n \to  \infty\) by \eqref{eq: Q conv}. 
	Similarly, \eqref{eq: cont map theo} holds if \(f \in D\) depends only on the \(\mathbb{R}^d\)-coordinate provided \((M^{f, n}_t)_{t \geq 0}\) is defined as in \eqref{eq: Mf2} with \(\mathcal{L}\) replaced by \(\mathcal{L}_n\). In this case, the second term in \eqref{eq: cont map theo} converges to zero because 
	\begin{align*}
	\big|&M^f_{t \wedge \tau_{t_m}(\omega_n)} (\omega_n) - M^{f, n}_{t \wedge \tau_{t_m}(\omega_n)} (\omega_n) \big| \\&\quad
	\leq \textup{const. } t \sup_{\tri y \tri \leq m} \Big(\|b(y) - b_n(y)\| + \| a(y) - a_n(y)\| \Big) \to 0
	\end{align*}
	as \(n \to \infty\), due to \eqref{eq: assp conv}. 
	We conclude from \cite[Theorem 3.27]{Kallenberg} that for all \(f \in D\) and \(t \in U\)
	\begin{align}\label{eq: conv mart imp}
	E^{P^{n, m}} \Big[ M^{f, n}_{t \wedge \tau_{t_m}} v \Big] \to E^{Q^m} \Big[M^f_{t \wedge \tau_{t_m}} v\Big]
	\end{align}
	as \(n \to \infty\), where \(P^{n, m}\) denotes the law of \((Y^n_{t \wedge \tau_{n, m}}, Z^n_t)_{t \geq 0}\).
	
	Fix \(s < t\). Because \(U\) is dense in \(\mathbb{R}_+\), we find a sequence \((z_n)_{n \in \mathbb{N}} \subset U\) such that \(z_n \searrow t\) as \(n \to \infty\) and a sequence \((u_n)_{n \in \mathbb{N}} \subset U\) such that \(u_n \searrow s\) as \(n \to \infty\). W.l.o.g. we can assume that \(u_n \leq z_n\) for all \(n \in \mathbb{N}\).
Let \(v \colon \Omega \to \mathbb{R}\) be continuous, bounded and \(\mathcal{F}_s\)-measurable. 
Using the dominated convergence theorem, the right-continuity of \((X_t)_{t \geq 0}\) and \eqref{eq: conv mart imp}, we obtain 
\begin{equation}\label{eq: weak use}
\begin{split}
E^{Q^m} \big[ M^f_{t \wedge \tau_{t_m}} v \big]  
= \lim_{k \to \infty} E^{Q^m} \big[ M^f_{z_k \wedge \tau_{t_m}} v \big]
	= \lim_{k \to \infty}  \lim_{n \to \infty} E^{P^{n, m}} \big[ M^{f, n}_{z_k \wedge \tau_{t_m}} v \big].
	\end{split}		
\end{equation}
The process \((M^{f, n}_{q \wedge \tau_{t_m}})_{q \geq 0}\)
is a \(P^{n, m}\)-martingale. To see this, note that \[\tau_{t_m} \circ (Y^n_{s \wedge \tau_{n, m}}, Z^n_s)_{s \geq 0} = \tau_{n, t_m},\] see \cite[Lemma III.2.43]{JS}, and recall that martingales are stable under stopping.
Consequently, using again \eqref{eq: conv mart imp} and the dominated convergence theorem, we conclude from \eqref{eq: weak use} that 
\begin{align*}
E^{Q^m} \big[ M^f_{t \wedge \tau_{t_m}} v \big] = \lim_{k \to \infty}  \lim_{n \to \infty} E^{P^{n, m}} \big[ M^{f, n}_{u_k \wedge \tau_{t_m}} v \big]
=  E^{Q^m} \big[ M^{f}_{s \wedge \tau_{t_m}} v \big].
\end{align*}
Recall that \(s < t\) and \(v\) were arbitrary. 

We claim that this already implies that \((M^f_{q \wedge \tau_{t_m}})_{q \geq 0}\) is a \(Q^m\)-martingale.
Take \(g \in C_b(S)\) and let \((m_k)_{k \in \mathbb{N}} \subset (0, \infty)\) be such that \(m_k \searrow 0\) as \(k \to \infty\). We set
\[
g^k (q) \triangleq \frac{1}{m_k} \int_{q}^{q+ m_k} g(X_r)\dd r,\quad k \in \mathbb{N}, q \in \mathbb{R}_+,
\]
and note that \(g^k (q) \colon \Omega \to \mathbb{R}\) is continuous, bounded and \(\mathcal{F}^o_{q + m
_k}\)-measurable and that
\(g^k(q) \to g(X_q)\) as \(k \to \infty\). Thus, using an approximation argument, we can deduce from the fact that \(E^{Q^m} \big[ M^f_{t \wedge \tau_{t_m}} v \big] = E^{Q^m} \big[ M^f_{s \wedge \tau_{t_m}} v \big]\) holds for all \(s < t\) and all continuous, bounded and \(\mathcal{F}_s\)-measurable \(v\) that
\[
E^{Q^m} \Big[ M^f_{t \wedge \tau_{t_m}} \prod_{i = 1}^l g_i(X_{q_i}) \Big]
=  E^{Q^m} \Big[ M^{f}_{s \wedge \tau_{t_m}} \prod_{i = 1}^l g_i(X_{q_i}) \Big],
\]
for all \(s < t\), \(l \in \mathbb{N}\), \(g_1, \dots, g_l\in C_b(S)\) and \(q_1, \dots, q_l \in [0, s]\).
Using a monotone class argument and the downwards theorem shows that \((M^f_{t \wedge \tau_{t_m}})_{t \geq 0}\) is a \(Q_m\)-martingale.

Because \(\omega \mapsto \omega(0)\) is continuous, we have \(Q^m(X_0 = (x, i)) = 1\) due to the continuous mapping theorem.
Due to \cite[Problem 4.25]{EK} the set \(\Sigma = \Sigma_c\times \Sigma_d\) is a closed set in the product Skorokhod topology on \(\Omega = D(\mathbb{R}_+, \mathbb{R}^d) \times \Sigma_d\), and \cite[Proposition 3.5.3]{EK} implies that \(\Sigma\) is closed in \(\Omega\), too.
Thus, by the Portmanteau theorem, we have \(Q^m (\Sigma) = 1\).
It follows from Lemma \ref{lem: lu} that \(Q^m\) coincides with \(P_{(x, i)}\) on \(\mathcal{F}_{\tau_{t_m}}^o\) and thus also on \(\mathcal{F}^o_{\tau_{m-1}}\), because \(t_m \geq m -1\) implies \(\tau_{t_m} \geq \tau_{m-1}\). This completes the proof.
\end{proof}
\begin{lemma}\label{lem: tight2}
	The sequence \(\{(Y^n_t, Z^n_t)_{t \geq 0}, n \in \mathbb{N}\}\) is tight. 
\end{lemma}
\begin{proof}
We use again Fact \ref{fact}.
As in the proof of the previous lemma, let \(P^{n, m}\) be the law of \((Y^n_{t \wedge \tau_{n, m}}, Z^n_t)_{t \geq 0}\) and \(P^n\) be the law of \((Y^n_t, Z^n_t)_{t \geq 0}\). We fix \(t \in \mathbb{R}_+\).
Due to \cite[Problem 13, p. 151]{EK} and \cite[Lemma 15.20]{HWY}, the set 
\(
\{\tau_{m-1} \leq t\}
\)
is closed. Moreover, \(\{\tau_{m-1} \leq t\} \in \mathcal{F}^o_{\tau_{m-1}}\), because \(\tau_{m -1}\) is an \((\mathcal{F}^o_t)_{t \geq 0}\)-stopping time. We deduce from the Portmanteau theorem and Lemma \ref{lem: uni1} that 
\begin{equation}\label{port}\begin{split}
\limsup_{n \to \infty}\ P^{n, m} (\tau_{m-1} \leq t) \leq P_{(x, i)} (\tau_{m-1} \leq t).
\end{split}\end{equation}
Fix \(\epsilon > 0\).
Since \(P_{(x, i)}(\tau_{m-1} \leq t) \searrow 0\) as \(m \to \infty\), we find an \(m^o \in \mathbb{N}_{\geq 2}\) such that 
\begin{align}\label{eq: P bound}
P_{(x, i)}(\tau_{m^o -1} \leq t) \leq \tfrac{\epsilon}{2}.
\end{align}
Because \((P^{n, m^o-1})_{n \in \mathbb{N}}\) is tight due to Lemma \ref{lem: tight1}, we deduce from Fact \ref{fact} that there exists a compact set \(C(t, \epsilon) \subseteq S\) such that 
\begin{align}\label{eq: t imp}
\limsup_{n \to \infty} P^{n, m^o - 1} (X_{t} \not \in C(t, \epsilon)) \leq \tfrac{\epsilon}{2}.
\end{align}
In view of \cite[Lemma III.2.43]{JS} we obtain
\begin{align*}
P^n (X_t \not \in C(t, \epsilon)) & = P^{n}(X_{t} \not \in C(t, \epsilon), \tau_{m^o-1} > t) + P^n(X_t \not \in C(t, \epsilon), \tau_{m^o-1} \leq t)
\\&\leq P^{n, m^o-1} (X_{t} \not \in C(t, \epsilon)) + P^{n, m^o}(\tau_{m^o-1} \leq t).
\end{align*}
From this, \eqref{port}, \eqref{eq: P bound} and \eqref{eq: t imp}, we deduce that 
\[
\limsup_{n \to \infty} P^n(X_t \not \in C(t, \epsilon)) \leq \epsilon.
\]
This proves that the sequence \((P^n)_{n \in \mathbb{N}}\) satisfies (a) in Fact \ref{fact}. 

Next, we show that \((P^n)_{n \in \mathbb{N}}\) satisfies (b) in Fact \ref{fact}. Let \(\epsilon, t\) and \(m^o\) be as before. Because \((P^{n, m^o-1})_{n \in \mathbb{N}}\) is tight due to Lemma \ref{lem: tight1} there exists a \(\delta > 0\) such that 
\begin{align}\label{eq: t imp2}
\limsup_{n \to \infty} P^{n, m^o-1} \left(w' ((X_{s})_{s \geq 0}, \delta, t) \geq \epsilon\right) \leq \tfrac{\epsilon}{2}.
\end{align}
Thus, similar as above, using \eqref{port}, \eqref{eq: P bound} and \eqref{eq: t imp2}, we obtain 
\begin{align*}
\limsup_{n \to \infty}&\ P^n (w'((X_s)_{s \geq 0}, \delta, t) \geq \epsilon) \\&\leq \limsup_{n \to \infty}P^{n, m^o-1} (w'((X_{s})_{s \geq 0}, \delta, t) \geq \epsilon) + \limsup_{n \to \infty} P^{n, m^o} (\tau_{m^o-1} \leq t) \\&\leq \epsilon.
\end{align*}
In other words, \((P^n)_{n \in \mathbb{N}}\) satisfies also (b) in Fact \ref{fact} and the proof is complete.
\end{proof}
We are in the position to complete the proof of Theorem \ref{theo: limit}. To wit, in view of \cite[Corollary to Theorem 5.1]{bil99}, because \(\{(Y^n_t, Z^n_t)_{t \geq 0}, n \in \mathbb{N}\}\) is tight by the previous lemma, for \(P^n \to P_{(x, i)}\) weakly as \(n \to \infty\), it remains to show that any accumulation point \(Q\) of \(\{(Y^n_t, Z^n_t)_{t \geq 0},n \in \mathbb{N}\}\) coincides with \(P_{(x, i)}\).
It follows as in the proof of Lemma \ref{lem: uni1} that the process \((M^f_{t})_{t \geq 0}\) is a \(Q\)-martingale for all \(f \in D\). 
Since \(\omega \mapsto \omega(0)\) is continuous, we also have \(Q(X_0 = (x, i)) = 1\) and, because \(\Sigma\) is closed in \(\Omega\), the Portmanteau theorem yields that \(Q (\Sigma) = 1\).
It follows that \(Q\) solves the MP \((D, \mathcal{L}, \Sigma, (x, i))\). Due to the uniqueness assumption, \(Q = P_{(x, i)}\) and the proof is complete.
\end{proof}
\subsubsection{Proof of Proposition \ref{prop: uni}}
	The existence is shown in the proof of Theorem \ref{prop: equiv}.
	The uniqueness follows from a Yamada-Watanabe argument, which we only sketch.
	Fix \(y = (x, i) \in S\) and suppose that \(P_y\) and \(Q_y\) solve the MP \((\mathcal{L}, D, \Sigma, y)\).
	Using similar arguments as in the proof of \cite[Theorem 8.3]{J80}, we obtain the following: We find a filtered probability space satisfying the usual hypothesis on which we can realize \(P_y\) as the law of the process \((Y_t, Z_t)_{t \geq 0}\), where \((Z_t)_{t \geq 0}\) is a Markov chain with \(Q\)-matrix \(Q\) and \(Z_0 = i\) and 
	\begin{align*}
	\dd Y_t = b(Y_t, Z_t) \dd t + a^\frac{1}{2} (Y_t, Z_t) \dd W_t, \quad Y_0 = x,
	\end{align*}
	where \((W_t)_{t \geq 0}\) is a Brownian motion.
	On the same probability space, we can realize \(Q_y\) as the law of \((V_t, Z_t)_{t \geq 0}\), where
	\[
	\dd V_t = b(V_t, Z_t) \dd t + a^\frac{1}{2} (V_t, Z_t) \dd W_t, \quad V_0 = x.
	\]
	We stress that the driving system \((Z_t, W_t)_{t \geq 0}\) coincides for \((Y_t)_{t \geq 0}\) and \((V_t)_{t \geq 0}\).
		Now, we claim that \(Y_t = V_t\) for all \(t \in \mathbb{R}_+\) up to a null set. This immediately implies \(Q_y = P_y\). We prove this claim by induction. Let \((\tau_n)_{n \in \mathbb{N}}\) be the stopping times as defined in \eqref{eq: jump times}.
		We stress that a.s. \(\tau_n \nearrow \infty\) as \(n \to \infty\).
		On \(\{t \leq \tau_1\}\) we have 
		\begin{align*}
		Y_t &= x + \int_0^t b(Y_s, i) \dd s + \int_0^t a^\frac{1}{2}(Y_s, i) \dd W_s,\\
		V_t &= x + \int_0^t b(V_s, i) \dd s + \int_0^t a^\frac{1}{2}(V_s, i) \dd W_s.
		\end{align*}
		The strong existence hypothesis and Lemma \ref{lem: loc pathwise uniqueness} imply that \(Y_t = V_t\) for all \(t \leq \tau_1\) up to a null set. 
		Suppose that \(n \in \mathbb{N}\) is such that \(Y_t = V_t\) for all \(t \leq \tau_{n}\) up to a null set. Using classical rules for time-changed stochastic integrals, we obtain that on \(\{t \leq \tau_{n + 1} - \tau_n\} \cap \{Z_{\tau_n} = k\}\) 
		\begin{align*}
		Y_{t + \tau_n} &= Y_{\tau_n} + \int_{\tau_n}^{t + \tau_n} b(Y_s, k) \dd s + \int_{\tau_n}^{t + \tau_n} a^\frac{1}{2}(Y_s, k) \dd W_s
		\\&= Y_{\tau_n} + \int_{0}^{t} b(Y_{s + \tau_n}, k) \dd s + \int_{0}^{t} a^\frac{1}{2}(Y_{s + \tau_n}, k) \dd W^n_s
		\end{align*}
		and
		\begin{align*}
		V_{t + \tau_n} &= V_{\tau_n} + \int_{0}^{t} b(V_{s + \tau_n}, k) \dd s + \int_{0}^{t} a^\frac{1}{2}(V_{s + \tau_n}, k) \dd W^n_s,
		\end{align*}
		where 
		\[
		W^n_t \triangleq W_{t + \tau_n} - W_{\tau_n},\quad t \in \mathbb{R}_+.
		\]
		We conclude again from the strong existence hypothesis and Lemma \ref{lem: loc pathwise uniqueness} that \(Y_{t + \tau_n} = V_{t + \tau_n}\) for all \(t \leq \tau_{n + 1} - \tau_n\) up to a null set. Consequently, \(Y_t = V_t\) for all \(t \leq \tau_{n +1}\) up to a null set and our claim follows.
\qed

\subsubsection{Proof of Corollary \ref{coro: complete}}
	Due to \cite[Theorems 5.5.15, 5.5.29]{KaraShre} and \cite[Corollary 11.1.5]{SV}, for all \(i \in S_d\) the family \((P^i_x)_{x \in \mathbb{R}^d}\) exists uniquely and is \(C_b\)-Feller. 
	Using the local H\"older condition on the diffusion coefficient, \cite[Lemma IX.3.3, Proposition IX.3.2]{RY} and \cite[Theorem 18.14]{Kallenberg} imply that \((P^i_x)_{x \in \mathbb{R}^d}\) exists strongly. Consequently, \((P_x)_{x \in S}\) exists uniquely due to Proposition \ref{prop: uni}.  
	Now, \((P_x)_{x \in S}\) is strongly Markov and \(C_b\)-Feller due to Proposition \ref{prop: sd feller} and the equivalence of (i) and (ii) follows from Theorem \ref{prop: equiv}, Remark \ref{rem: eqivalence} and \cite[Theorem 8.4.1]{pinsky1995positive}.
\qed
\subsubsection{Proof of Corollary \ref{coro: conv}}
Due to \cite[Theorem 5.2.5]{KaraShre}, \cite[Theorem 18.14]{Kallenberg} and \cite[Corollary 11.1.5]{SV}, for all \(i \in S_d\) the family \((P^i_x)_{x \in \mathbb{R}^d}\) exists strongly and is \(C_b\)-Feller. Consequently, \((P_x)_{x \in S}\) exists uniquely due to Proposition \ref{prop: uni}.  
As in the proof of Proposition \ref{prop: sd}, we deduce from Theorem \ref{theo: Ly} that \((P^i_x)_{x \in \mathbb{R}^d}\) is Feller-Dynkin for all \(i \in S_d\). 
Now, \((P_x)_{x \in S}\) is strongly Markov and \(C_b\)-Feller due to Proposition \ref{prop: sd feller} and Feller-Dynkin due to Theorem \ref{prop: equiv}.
\qed
\subsubsection{Proof of Corollary \ref{prop: nf}}
Due to \cite[Theorem 5.2.5]{KaraShre}, \cite[Theorem 18.14]{Kallenberg} and \cite[Corollary 11.1.5]{SV}, the family \((P^i_x)_{x \in \mathbb{R}^d}\) exists strongly and is \(C_b\)-Feller.
	Moreover, as in the proof of Proposition \ref{prop: finite environment NFDP}, we deduce from Theorem \ref{prop:cont} that \((P^i_x)_{x \in \mathbb{R}^d}\) is not Feller-Dynkin. Finally, the claim follows from Proposition \ref{lem: non-existence}.
	\qed
\appendix 
\section{An Existence Theorem for Switching Diffusions}\label{appendix}
In this appendix we give an existence theorem for switching diffusions with state-independent switching.  
We pose ourselves in the setting of Section \ref{sec: sic}. 

\begin{theorem}\label{theo: existence}
	Let \(b \colon S \to \mathbb{R}^d\) and \(a \colon S \to \mathbb{S}^d\) be continuous functions such that for all \(m \in \mathbb{R}_+\)
	\begin{align}\label{eq: uni growth}
	\sup_{\|x\| \leq m} \sup_{i \in S_d} \big(\|b(x, i)\| + \|a (x, i)\| \big) < \infty.
	\end{align} Let \(\mathcal{K}^i\) be given as in \eqref{eq: Ki}.
	Suppose that there exists two constants \(c, \lambda > 0\), a function \(v \colon \mathbb{R}_+ \to (0, \infty)\) and a twice continuously differentiable function \(V \colon \mathbb{R}^d \to \mathbb (0,\infty)\) such that \(V(x) \geq v(\|x\|)\) for all \(x \in \mathbb{R}^d \colon \|x\| \geq \lambda\), \(\limsup_{n \to \infty} v(n) = \infty\) and 
\[
\mathcal{K}^i V (x) \leq c V(x), 
\]
for all \((x, i) \in S\). Then, for any Borel probability measure \(\eta\) on \(S\) there exists a solution to the MP \((D,\mathcal{L},  \Sigma, \eta).\)
\end{theorem}
\begin{proof}
	Due to Proposition \ref{prop: existence initial law} in Appendix \ref{app: prop exietence initial law}, it suffices to show the claim for degenerated initial laws, i.e. we assume that \(\eta(\{y\}) = 1\) for some \(y \in S\). 
	
	\textit{Step 1.} We first show the claim under the assumptions that \(b\) and \(a\) are continuous and bounded, i.e.
	\(
	\|b(x, i)\| + \|a(x, i)\| \leq c^*
	\)
	for all \((x, i)\in S\).
	Our initial step is a standard mollification argument.
	Let \(\phi\) be the standard mollifier, i.e. 
	\begin{align*}
	\phi(x) \triangleq \begin{cases} \theta \exp \big\{ - (1 - \|x\|^2)^{-1}\big\},&\text{if } \|x\| < 1,\\0,&\text{otherwise,}\end{cases}
	\end{align*}
	where \(\theta > 0\) is a constant such that \(\int \phi(x)\dd x = 1\). Let \(\sigma\) be a root of \(a\).
	For \((x, i) \in S\) we set 
	\begin{align*}
	b_n (x, i) &\triangleq n^d \int b(y, i) \phi(n (x - y)) \dd y, \\ \sigma_n (x, i) &\triangleq n^d \int \sigma(y, i) \phi(n (x - y)) \dd y.
	\end{align*}
	It is well-known that \(x \mapsto b_n(x, i)\) and \(x \mapsto \sigma_n(x, i)\) are smooth for all \(i \in S_d\) and that \(b_n \to b\) and \(\sigma_n \sigma_n^* \to a\) as \(n \to \infty\) uniformly on compact subsets of \(S\).
	Furthermore, using that \(\int \phi(x)\dd x=1\), we obtain
	\begin{align*}
	\|b_n(x, i)\| &\leq n^d \int \|b(y, i)\| \phi(n (x - y))\dd y 
	= \int \|b(x - n^{-1} z, i)\| \phi(z)\dd z \leq c^*
	\end{align*}
	and, in the same manner,
\(
\|\sigma_n(x, i)\| \leq c^*
\)
for all \((x, i) \in S\).
	Because smooth functions are locally Lipschitz continuous, we deduce from \cite[Theorem 18.16]{schilling2014brownian}, \cite[Theorem 18.14]{Kallenberg} and Proposition \ref{prop: uni} that for each \(n \in \mathbb{N}\) there exists a solution \(P^n\) to the MP \((D, \mathcal{L}_n, \Sigma, y)\), where \(\mathcal{L}_n\) is defined as in \eqref{gen: sd} with \(b\) replaced by \(b_n\) and \(a\) replaced by \(a_n\). If we show that the sequence \((P^n)_{n \in \mathbb{N}}\) is tight and that any accumulation point of it solves the MP \((D,\mathcal{L}, \Sigma, y)\) the claim of the theorem follows.
That any accumulation point of \((P^n)_{n \in \mathbb{N}}\) solves the MP \((D,\mathcal{L}, \Sigma, y)\) can be shown as in the proof of Theorem \ref{theo: limit} and that \((P^n)_{n \in \mathbb{N}}\) is tight follows as in the proof of Lemma \ref{lem: tight1}. Thus, the claim holds under the assumptions that \(b\) and \(a\) are continuous and bounded. 

\textit{Step 2.} We now tackle the general case. Let \(\psi^n\colon \mathbb{R}^d \to [0, 1]\) be a sequence of cutoff functions, i.e. non-negative smooth functions with compact support such that \(\psi^n (x) = 1\) for \(x \in \mathbb{R}^d \colon \|x\| \leq n\). We set 
\[
b_n(x, i) \triangleq \psi^n(x) b(x, i),\qquad a_n(x, i) \triangleq \psi^n(x) a(x, i), \quad (x, i) \in S.
\]
The functions \(b_n\) and \(a_n\) are continuous and bounded. Therefore, due to our first step, for each \(n \in \mathbb{N}\) there exist a solution \(P^n\) to the MP \((D, \mathcal{L}_n, \Sigma, y)\).  
We write \((X_t)_{t \geq 0} = (X^1_t, X^2_t)_{t \geq 0}\) and set 
\[
\tau_m \triangleq \inf \big(t \in \mathbb{R}_+ \colon \|X^1_t\| \geq m \textup{ or } \|X^1_{t-}\| \geq m \big),\quad m \in \mathbb{R}_+.
\]
Furthermore, we denote \(P^{n, m} \triangleq P^n \circ (X^1_{t \wedge  \tau_m}, X^2_t)_{t \geq 0}^{-1}\).
It follows as in the proof of Lemma \ref{lem: tight1} that the sequence \((P^{n, m})_{n \in \mathbb{N}}\) is tight for every \(m \in \mathbb{R}_+\).
We note that for all \(m \in \mathbb{R}_+\) 
\begin{align*}
\sup_{\tri x \tri \leq m} &\big(\|b (x) - b_n(x)\| + \|a(x) - a_n(x)\| \big) 
\\&\leq 2 \sup_{\tri x \tri \leq m} \big(\|b (x)\| + \|a(x)\| \big) \sup_{\|z\| \leq m} |1 - \psi^n(z)| \to 0
\end{align*}
as \(n \to \infty\).
Thus, recalling the proof of Lemma \ref{lem: tight2} and Step 1 reveal that the existence of a solution to the MP \((D, \mathcal{L}, \Sigma, y)\) follows once we prove that for each \(T > 0\) and \(\epsilon > 0\) we find an \(m \in \mathbb{R}_+\) such that 
\begin{align}\label{eq: to show 2}
\limsup_{n \to \infty} P^{n} (\tau_m \leq T) \leq  \epsilon.
\end{align}
Define \(\mathcal{K}^{i, n}\) as \(\mathcal{K}^i\) with \(b\) and \(a\) replaced by \(b_n\) and \(a_n\). We have 
\begin{align*}
\mathcal{K}^{i, n} V(x) = \psi^n(x) \mathcal{K}^i V(x) \leq c \psi^n(x) V(x) \leq c V(x)
\end{align*}
for all \((x, i) \in S\) and \(n \in \mathbb{N}\).
By Lemma \ref{rem: spacetime} the process
\[
U_t \triangleq e^{- c (t \wedge \tau_m)} V(X^1_{t \wedge \tau_m}) + \int_0^{t \wedge \tau_m} e^{- cs} \big(c V(X^1_s)  - \mathcal{K}^{X^2_s, n} V (X^1_s)\big)\dd s,\quad t \in \mathbb{R}_+,
\]
is a local \(P^{n}\)-martingale. Furthermore, because \(U_t \geq e^{- c(t \wedge \tau_m)} V(X^1_{t \wedge \tau_m}) \geq 0\) for all \(t \in \mathbb{R}_+\), the process \((U_t)_{t \geq 0}\) is a non-negative \(P^{n}\)-supermartingale. 
We deduce that for all \(m \geq \lambda \vee \|x\|\)
\begin{align*}
P^{n}( \tau_m \leq T) e^{- cT} v(m) &= E^{n} \Big[ \1_{\{\tau_m \leq T\}} e^{- cT} v(\|X^1_{\tau_m}\|)\Big]
\\&\leq E^{n} \Big[ \1_{\{\tau_m \leq T\}} e^{- c(T \wedge \tau_m)} V\big(X^1_{T \wedge \tau_m}\big) \Big]
\\&\leq E^{n} \Big[  e^{- c(T \wedge \tau_m)} V\big(X^1_{T \wedge \tau_m}\big) \Big]\
\\&\leq E^{n} \Big[ U_{T} \Big]
\leq V(x),
\end{align*}
where \(y = (x, i)\).
The assumption \(\limsup_{m \to \infty} v(m) = \infty\) yields that we find an \(m \geq\lambda\) such that 
\eqref{eq: to show 2} holds. 
This completes the proof.
\end{proof}
\begin{remark}
	\begin{enumerate}
		\item[\textup{(i)}]
	On one hand, the previous existence result does not require any uniqueness or strong existence assumption for the SDEs for the fixed environments. On the other hand, it does not provide a uniqueness statement.
	\item[\textup{(ii)}] Using \(V(x) = 1 + \|x\|^2\) yields that the growth condition
	\[
	2\langle x, b(x, k)\rangle + \textup{trace } a(x, k) \leq c \big(1 + \|x\|^2\big),\quad \text{for all }(x, k) \in S, 
	\] 
	implies the existence of a solution to the MP \((D, \mathcal{L}, \Sigma, \eta)\) whenever the coefficients \(b\) and \(a\) are continuous and satisfy \eqref{eq: uni growth}.
	\end{enumerate}
\end{remark}

\section{The Role of Initial Laws}\label{app: prop exietence initial law}
For the setting of Example \ref{expl: 1} it is known that the existence of (unique) solutions for all degenerated initial laws implies the existence of (unique) solutions for all initial laws, see \cite[Propositions 1 and 2]{10.2307/2244838}.  
The following proposition shows that these observations also hold in our setting. The proof is close to the diffusion case and we only sketch it.
\begin{proposition}\label{prop: existence initial law}
	Suppose that \(D\) is countable, that \(D \subseteq C_b(S)\) and that \(\mathcal{L}(D) \subseteq B_\textup{loc}(S)\). Furthermore, let \(\eta\) be a Borel probability measure on \(S\). If for all \(y \in S\) the MP \((D, \mathcal{L}, \Sigma, y)\) has a solution \(P_y\), then also the MP \((D, \mathcal{L}, \Sigma, \eta)\) has a solution. Moreover, if the family \((P_y)_{y \in S}\) is unique, then \(y \mapsto P_y(A)\) is Borel for all \(A \in \mathcal{F}\) and \(\int P_y\eta(\dd y)\) is the unique solution to the MP \((D, \mathcal{L}, \Sigma, \eta)\).
\end{proposition}
\begin{proof}[Sketch of Proof]
	We assume that the MP \((D, \mathcal{L}, \Sigma, y)\) has a solution for all \(y \in S\). 
	Let \(\eta\) be a Borel probability measure on \(S\) and let \(\mathcal{P}\) denote the set of all solutions to the MP \((D, \mathcal{L}, \Sigma, y)\) for all \(y \in S\). We consider \(\mathcal{P}\) as a subspace of the Polish space \(\mathscr{P}\) of probability measures on \((\Omega, \mathcal{F})\) equipped with the topology of convergence in distribution. 
	Let \((K_n)_{n \in \mathbb{N}} \subset S\) be a sequence of compact sets such that \(K_n \subset \textup{int}(K_{n + 1})\) and \(\bigcup_{n \in \mathbb{N}} K_n = S\). For all \(n \in \mathbb{N}\) define \(\tau_n \triangleq \inf (t \in \mathbb{R}_+ \colon X_t \not \in \textup{int}(K_n) \text{ or } X_{t-}\not \in \textup{int} (K_n))\) and for \(f \in D\) denote the process \eqref{eq: Mf} by \((M^f_t)_{t \geq 0}\).
	Because we assume that \(\mathcal{L}(D) \subseteq B_\textup{loc}(S)\), a probability measure \(P\) solves the MP \((D, \mathcal{L}, \Sigma, \eta)\) if and only if \(P(\Sigma) = 1, P \circ X^{-1}_0 = \eta\) and for all \(f \in D\) and \(n \in \mathbb{N}\) the stopped process \((M^f_{t \wedge \tau_n})_{t \geq 0}\) is a \(P\)-martingale for the filtration \((\mathcal{F}^o_t)_{t \geq 0}\). Because  \(D\) is  assumed to be countable, the argument outlined in \cite[Exercise 6.7.4]{SV}
	shows that \(\mathcal{P}\) is a Borel subset of \(\mathscr{P}\).
	Thus, \(\mathcal{P}\) is a Borel space in the sense of \cite[p. 456]{Kallenberg}.
	Let \(\Phi\colon \mathcal{P} \to S\) be such that \(\Phi (P)\) is the starting point associated to \(P \in \mathcal{P}\). 
	We note that \(\Phi\) is continuous and that 
	its graph \(G \triangleq \big\{ (P, \Phi(P)) \colon P \in \mathcal{P}\big\}\) is a Borel subset of \(\mathcal{P} \times S\).
	We have 
	\(
	\bigcup_{P \in \mathcal{P}} \big\{s \in S \colon s = \Phi(P)\big\} = S,
	\)
	by the assumption that there exist solutions for all degenerated initial laws.
	Using the section theorem \cite[Theorem A.1.8]{Kallenberg} we see that there exists a Borel map \(x \mapsto P_x\) and a \(\eta\)-null set \(N \in \mathcal{B}(S)\) such that \((P_x, x) \in G\) for all \(x \not \in N\). By the definition of \(G\), for all \(x \not \in N\) the probability measure \(P_x\) solves the MP \((D, \mathcal{L}, \Sigma, x)\). 
	It follows that the probability measure \(\int P_x \eta(\dd x)\) 
	solves the MP \((D,\mathcal{L}, \Sigma, \eta)\). 
	
	Assume now that \(P_x\) is the unique solution to the MP \((D, \mathcal{L}, \Sigma, x)\) for all \(x \in S\). Using Kuratovski's theorem as outlined in \cite[Exercise 6.7.4]{SV} shows that \(x \mapsto P_x\) is Borel.
	Let \(P\) be a solution to the MP \((D, \mathcal{L}, \Sigma, \eta)\). Arguing as in the proof of \cite[Lemma 5.4.19]{KaraShre} shows that
	there exists a null set \(N \in \mathcal{F}_0^o\) such that \(P(\cdot |\mathcal{F}_0^o)(\omega)\) solves the MP \((D, \mathcal{L}, \Sigma, X_0(\omega))\) for all \(\omega \not \in N\).
	By the uniqueness assumption, this yields that \(P\)-a.s. \(P_{X_0} = P(\cdot |\mathcal{F}^o_0)\). Using this observation together with the tower rule shows that \(P = \int P_x \eta(\dd x)\).
\end{proof}

It is often the case that the input data of a martingale problem can be reduced such that the prerequisites of Proposition \ref{prop: existence initial law} are met, see Proposition \ref{prop: redc} and Example \ref{ex: MC}.
\bibliographystyle{agsm}

\end{document}